\documentclass[a4paper,11pt]{article}
\usepackage{amsmath,amsfonts,amssymb,amsthm}
\usepackage{booktabs}
\usepackage{caption}
\usepackage{comment}
\usepackage{enumerate}
\usepackage[hmargin={28mm,28mm},vmargin={30mm,35mm}]{geometry}
\usepackage[ocgcolorlinks,allcolors={blue},breaklinks]{hyperref}
\usepackage{paralist}
\usepackage{subcaption}
\usepackage{graphicx}

\usepackage{newtxtext,newtxmath}

\newcommand{\email}[1]{\href{mailto:#1}{\tt #1}}

\graphicspath{{figures/}}

\theoremstyle{theorem}
\newtheorem{theorem}{Theorem}
\newtheorem{lemma}[theorem]{Lemma}
\newtheorem{proposition}[theorem]{Proposition}
\newtheorem{corollary}[theorem]{Corollary}
\theoremstyle{remark}
\newtheorem{remark}[theorem]{Remark}
\theoremstyle{definition}
\newtheorem{assumption}[theorem]{Assumption}

\newcommand{\VEC}[1]{\boldsymbol{#1}}
\newcommand{\MAT}[1]{\boldsymbol{#1}}
\newcommand{\trans}{^{\mathrm T}}

\DeclareMathOperator{\opcard}{card}
\newcommand{\card}[1]{\opcard(#1)}

\DeclareMathOperator{\curl}{curl}

\newcommand{\Real}{\mathbb{R}}
\newcommand{\Natural}{\mathbb{N}}

\newcommand{\Poly}[1]{\mathbb{P}^{#1}}
\newcommand{\RTN}[1][k]{\mathbb{RTN}^{#1}}
\newcommand{\IRTN}[1][k]{\VEC{I}_{\mathbb{RTN},h}^{#1}}
\newcommand{\IRTNT}[1][k]{\VEC{I}_{\mathbb{RTN},T}^{#1}}

\newcommand{\SCAL}{{\cdot}}
\newcommand{\GRAD}{\VEC{\nabla}}
\newcommand{\GRADs}{\MAT{\nabla}_{\rm s}}
\newcommand{\GRADss}{\MAT{\nabla}_{\rm ss}}
\newcommand{\DIV}{\VEC{\nabla}{\cdot}}
\newcommand{\LAPL}{{\Delta}}

\newcommand{\norm}[2][]{\|#2\|_{#1}}
\newcommand{\seminorm}[2][]{|#2|_{#1}}

\newcommand{\term}{\mathfrak{T}}

\newcommand{\st}{\,:\,}

\newcommand{\Hdiv}[1][\Omega]{\VEC{H}({\rm div};#1)}

\newcommand{\Th}[1][h]{\mathcal{T}_{#1}}
\newcommand{\Fh}[1][h]{\mathcal{F}_{#1}}
\newcommand{\Fhi}[1][h]{\mathcal{F}_{#1}^{{\rm i}}}
\newcommand{\Fhb}[1][h]{\mathcal{F}_{#1}^{{\rm b}}}

\newcommand{\normal}{\VEC{n}}

\newcommand{\darcy}{{\rm D}}
\newcommand{\stokes}{{\rm S}}

\newcommand{\strain}{\MAT{\varepsilon}}

\newcommand{\Cf}[1][T]{C_{{\rm f},#1}}

\newcommand{\uvec}[1]{\underline{\VEC{#1}}}
\newcommand{\cvec}[1]{\check{\VEC{#1}}}

\newcommand{\UT}{\uvec{U}_T^k}
\newcommand{\Uh}{\uvec{U}_h^k}
\newcommand{\UhD}{\uvec{U}_{h,0}^k}

\newcommand{\lproj}[2][h]{\pi_{#1}^{#2}}
\newcommand{\vlproj}[2][h]{\VEC{\pi}_{#1}^{#2}}
\newcommand{\elproj}[2][T]{\VEC{\pi}_{\strain,#1}^{#2}} 
\newcommand{\IT}[1][k]{\uvec{I}_T^{#1}}
\newcommand{\Ih}{\uvec{I}_h^k}

\newcommand{\rST}{\VEC{r}_{\stokes,T}^{k+1}}
\newcommand{\rDT}[1][k]{\VEC{r}_{\darcy,T}^{#1}}
\newcommand{\rDh}{\VEC{r}_{\darcy,h}^k}

\newcommand{\dT}{\VEC{\delta}_{\stokes,T}^l}
\newcommand{\dTF}{\VEC{\delta}_{\stokes,TF}^k}

\newcommand{\dDT}[1][l]{\VEC{\delta}_{\darcy,T}^{#1}}
\newcommand{\dDTF}[1][k]{\VEC{\delta}_{\darcy,TF}^{#1}}

\newcommand{\err}{\mathfrak{R}(\VEC{u},p)}
\newcommand{\Err}[1][h]{\mathcal{E}_{#1}}

\newcommand{\sca}[1]{\widehat{#1}}

\begin{document}

\title{A Hybrid High-Order discretisation of the Brinkman problem robust in the Darcy and Stokes limits}

\author{Lorenzo Botti\thanks{Department of Engineering and Applied Sciences, University of Bergamo (Italy), \email{lorenzo.botti@unibg.it}}
  \and
  Daniele A. Di Pietro\thanks{Institut Montpelli\'erain Alexander Grothendieck, Univ Montpellier, CNRS (France), \email{daniele.di-pietro@umontpellier.fr}}
  \and
  J\'er\^ome Droniou\thanks{School of Mathematical Sciences, Monash University, Melbourne (Australia), \email{jerome.droniou@monash.edu}}
}

\maketitle

\begin{abstract}
  In this work, we develop and analyse a novel Hybrid High-Order discretisation of the Brinkman problem.
  The method hinges on hybrid discrete velocity unknowns at faces and elements and on discontinuous pressures.
  Based on the discrete unknowns, we reconstruct inside each element a Stokes velocity one degree higher than face unknowns, and a Darcy velocity in the Raviart--Thomas--N\'ed\'elec space.
  These reconstructed velocities are respectively used to formulate the discrete versions of the Stokes and Darcy terms in the momentum equation, along with suitably designed penalty contributions.
  The proposed construction is tailored to yield optimal error estimates that are robust throughout the entire spectrum of local (Stokes- or Darcy-dominated) regimes, as identified by a dimensionless number which can be interpreted as a friction coefficient.
  The singular limit corresponding to the Darcy equation is also fully supported by the method.
  Numerical examples corroborate the theoretical results.
  This paper also contains two contributions whose interest goes beyond the specific method and application treated in this work:
  an investigation of the dependence of the constant in the second Korn inequality on star-shaped domains and
  its application to the study of the approximation properties of the strain projector in general Sobolev seminorms.
\end{abstract}


\section{Introduction}

In this work, we develop and analyse a novel Hybrid High-Order (HHO) method for the Brinkman problem robust across the entire range of (Stokes- or Darcy-dominated) local regimes.

Let $\Omega\subset\Real^d$, $d\in\{2,3\}$, denote a bounded connected open polygonal (if $d=2$) or polyhedral (if $d=3$) set that does not have cracks, i.e., it lies on one side of its boundary $\partial\Omega$.
Let two functions $\mu:\Omega\to\Real$ and $\nu:\Omega\to\Real$ be given corresponding, respectively, to the fluid viscosity and to the ratio between the viscosity and the permeability of the medium.
In what follows, we assume that there exist real numbers $\underline{\mu},\overline{\mu}$ and $\underline{\nu},\overline{\nu}$ such that, almost everywhere in $\Omega$,
\begin{equation}\label{eq:mu.nu:bounds}
  0<\underline{\mu}\le\mu\le\overline{\mu},\qquad
  0\le\underline{\nu}\le\nu\le\overline{\nu}.
\end{equation}
Let $\VEC{f}:\Omega\to\Real^d$ and $g:\Omega\to\Real$ denote volumetric source terms.
The Brinkman problem reads:
Find the velocity $\VEC{u}:\Omega\to\Real^d$ and the pressure $p:\Omega\to\Real$ such that
\begin{subequations}\label{eq:strong}
  \begin{alignat}{2}
    \label{eq:strong:momentum}
    -\DIV(2\mu\GRADs\VEC{u}) + \nu\VEC{u} + \GRAD p &= \VEC{f} &\qquad&\text{in $\Omega$},
    \\ \label{eq:strong:mass}
    \DIV\VEC{u} &= g &\qquad&\text{in $\Omega$},
    \\ \label{eq:bc}
    \VEC{u} &= \VEC{0} &\qquad&\text{on $\partial\Omega$},
    \\ \label{eq:zero.mean.pressure}
    \int_\Omega p &= 0,
  \end{alignat}
\end{subequations}
where $\GRADs$ denotes the symmetric part of the gradient.
The PDE \eqref{eq:strong} locally behaves like a Stokes or a Darcy problem depending on the value of a dimensionless parameter, which can be interpreted as a local friction coefficient.
Our goal is to handle both situations robustly, while keeping the usual convergence properties of HHO methods.

The literature on the discretisation of problem \eqref{eq:strong} is vast, and giving a detailed account lies out of the scope of the present work.
As noticed in \cite{Mardal.Tai.ea:02}, the construction of a finite element which is uniformly well-behaved for both the Stokes and Darcy problems is not trivial.
Some choices tailored to the Stokes problem fail to convergence in the Darcy limit (as is the case for the unstabilised Crouzeix--Raviart finite element \cite{Crouzeix.Raviart:73}), or experience a loss of convergence and, possibly, a lack of convergence for the divergence of the velocity (as is the case for the Taylor--Hood element \cite{Taylor.Hood:73} or the minielement \cite{Arnold.Brezzi.ea:84}).
Concerning the Crouzeix--Raviart element, a possible fix was proposed in \cite{Burman.Hansbo:05} based on jump penalisation terms inspired by Discontinuous Galerkin methods.
In \cite{Burman.Hansbo:07}, the same authors study a discretisation based on piecewise linear velocities and piecewise constant pressures for which (generalised) inf--sup stability is obtained through pressure stabilisation.
Stabilised equal-order finite elements are also proposed and analysed in \cite{Braack.Schieweck:11}.
A generalisation of the classical minielement is studied in \cite{Juntunen.Stenberg:10}, where uniform a priori and a posteriori error estimates are derived.
The use of Darcy-tailored, $H({\rm div};\Omega)$-conforming finite element methods is investigated in \cite{Konno.Stenberg:11}, where the continuity of the tangential component of the velocity across interfaces is enforced via symmetric interior penalty terms.
Finite element methods have also been developed starting from weak formulations different from the one discussed in Section \ref{sec:continuous} below.
Vorticity--velocity--pressure formulations are considered, e.g., in \cite{Anaya.Gatica.ea:15,Anaya.Mora.ea:16}.
Finally, new generation technologies have been recently proposed for the discretisation of problem \eqref{eq:strong}.
We cite, in particular, the isogeometric divergence-conforming B-splines of \cite{Evans.Hughes:13},
the Weak Galerkin method of \cite{Mu.Wang.ea:14},
the two-dimensional Virtual Element methods of \cite{Caceres.Gatica.ea:17,Vacca:18} (see also the related work \cite{Beirao-da-Veiga.Lovadina.ea:17}),
and the multiscale hybrid-mixed method of \cite{Araya.Harder.ea:17}.

In the HHO method studied here, for a given polynomial degree $k\ge 1$, the discrete unknowns for the velocity are vector-valued polynomials of total degree $\le k$ over the mesh faces and of degree $\le l\coloneq\max(k-1,1)$ inside the mesh elements.
The discrete unknowns for the pressure are scalar-valued polynomials of degree $\le k$ inside each element.
Based on the discrete velocity unknowns, we reconstruct, inside each mesh element $T$:
(i) a Stokes velocity inspired by \cite{Di-Pietro.Ern:15} which yields the strain projector of degree $(k+1)$ inside $T$ when composed with the local interpolator and
(ii) a Darcy velocity in the local Raviart--Thomas--N\'ed\'elec space \cite{Raviart.Thomas:77,Nedelec:80} of degree $k$.
The Stokes and Darcy velocity reconstructions are used to formulate the discrete counterparts of the first and second terms in \eqref{eq:strong:momentum}.
Coercivity is ensured by stabilisation terms that penalise the difference between the discrete unknowns and the interpolate of the corresponding reconstructed velocity.
Owing to this finely tailored construction, the resulting method behaves robustly across the entire range of local (Stokes- or Darcy-dominated) regimes.

We carry out an exhaustive analysis of the method.
We first show in Theorem \ref{thm:well-posedness} that the method is inf--sup stable and, based on this result, that the discrete problem is well-posed.
We next prove in Theorem \ref{thm:err.est} an estimate in $h^{k+1}$ (with $h$ denoting, as usual, the meshsize) for the energy-norm of the error defined as the difference between the discrete solution and the interpolate of the continuous solution.
This estimate is robust in the sense that the multiplicative constant in the right-hand side:
(i) is prevented from exploding in both the Stokes- and Darcy-limits by cutoff factors;
(ii) has an explicit dependence on the local friction coefficient that shows how the relative importance of the Stokes- and Darcy-contributions varies according to the local regime;
(iii) does not depend on the pressure, thereby ensuring robustness when $\VEC{f}$ has large irrotational part (see \cite{Di-Pietro.Ern.ea:16} and references therein for further insight into this point).
The Darcy velocity reconstruction in the Raviart--Thomas--N\'ed\'elec space plays a key role in achieving the aforementioned robust features while retaining optimal convergence.
We point out that, to the best of our knowledge, estimates for the Brinkman problem where the various local regimes are identified by a dimensionless number are new, and they contribute to shedding new light on aspects of this problem that had often been previously treated only in a more qualitative fashion.
Finally, it is worth mentioning that the theoretical results extend to the Darcy problem (corresponding to $\mu=0$ and $\underline{\nu}>0$) thanks to a stabilisation term that strengthens the coercivity norm for the Darcy term in \eqref{eq:strong:momentum}; see Remark \ref{rem:darcy} and the numerical tests in Section \ref{sec:numerical.examples}.

Besides the results specific to the Brinkman problem, this paper also contains two important contributions of more general interest.
The first contribution is a study of the dependence of the constant in the second Korn inequality for polytopal domains that are star-shaped with respect to every point of a ball.
We show, in particular, that this type of inequality holds uniformly inside each mesh element when considering regular mesh sequences, a key point to prove stability and error estimates for discretisation methods.
The second contribution of general interest, linked to the latter point, are optimal approximation results for the strain projector, stated in Theorem \ref{thm:elproj:approx} and Corollary \ref{cor:elproj:approx.trace}, which extend \cite[Lemma 2]{Di-Pietro.Ern:15} to more general Sobolev seminorms.
The proof hinges on the framework of \cite[Section 2.1]{Di-Pietro.Droniou:17*1} for the study of projectors on local polynomial spaces, based in turn on the classical theory of \cite{Dupont.Scott:80}.

The rest of the paper is organised as follows.
In Section \ref{sec:continuous} we recall a classical weak formulation of problem \eqref{eq:strong}.
In Section \ref{sec:setting} we discuss the discrete setting: mesh, local and broken polynomial spaces, and $L^2$-orthogonal projectors thereon.
In Section \ref{sec:discrete} we describe the construction underlying the HHO method, formulate the discrete problem, and state the main results (whose proofs are postponed to Section \ref{sec:proofs}).
Numerical results are collected in Section \ref{sec:numerical.examples}.
The paper is completed by an appendix made of two sections. \ref{app:korn} is dedicated to proving a uniform Korn inequality for star-shaped polytopal sets. This inequality is used in \ref{app:elproj.approx} to study the approximation properties of the strain projector on local polynomial spaces for such sets.
The material is structured so that multiple levels of reading are possible. Readers mainly interested in the numerical recipe and results can focus on Sections \ref{sec:continuous} to \ref{sec:numerical.examples}. %
Those interested in the details of the convergence analysis can additionally consult Section \ref{sec:proofs} and, possibly, \ref{sec:appen}.


\section{Continuous problem}\label{sec:continuous}

In what follows, for any $X\subset\Omega$, we denote by $({\cdot},{\cdot})_X$ the usual inner product of $L^2(X)$, by $\|{\cdot}\|_X$ the corresponding norm, and we adopt the convention that the subscript is omitted whenever $X=\Omega$.
The same notation is used for the spaces of vector- and tensor-valued functions $L^2(X)^d$ and $L^2(X)^{d\times d}$, respectively.
We assume henceforth that $\VEC{f}\in L^2(\Omega)^d$ and $g\in L^2(\Omega)$.
Setting
\begin{equation}\label{eq:U.P}
  \VEC{U}\coloneq H^1_0(\Omega)^d,\qquad
  P\coloneq\left\{q\in L^2(\Omega)\st\int_\Omega q=0\right\},
\end{equation}
the weak formulation of problem \eqref{eq:strong} reads:
Find $(\VEC{u},p)\in\VEC{U}\times P$ such that
\begin{subequations}\label{eq:weak}
  \begin{alignat}{2}\label{eq:weak:momentum}
    a(\VEC{u},\VEC{v}) + b(\VEC{v},p) &= (\VEC{f},\VEC{v}) &\qquad& \forall\VEC{v}\in\VEC{U},
    \\ \label{eq:weak:mass}
    -b(\VEC{u},q) &= (g,q) &\qquad& \forall q\in P,
  \end{alignat}
\end{subequations}
with bilinear forms $a:\VEC{U}\times\VEC{U}\to\Real$ and $b:\VEC{U}\times P\to\Real$ such that
$$
a(\VEC{u},\VEC{v})\coloneq (2\mu\GRADs\VEC{u},\GRADs\VEC{v}) + (\nu\VEC{u},\VEC{v}),\qquad
b(\VEC{v},q)\coloneq -(\DIV\VEC{v},q).
$$
We recall that, for a vector-valued function $\VEC{u}=(u_i)_{i=1,\ldots,d}$,
$\GRAD \VEC{u}$ is the matrix $(\partial_j u_i)_{i,j=1,\ldots,d}$ and the symmetric gradient of $\VEC{u}$ is $\GRADs \VEC{u}=\frac12(\GRAD \VEC{u}+(\GRAD \VEC{u})\trans)$. The well-posedness of \eqref{eq:weak} results from the Lax-Milgram theorem and the first Korn inequality, which states the existence of a constant $C$ such that, for all $\VEC{u}\in \VEC{U}$, $\norm{\GRAD \VEC{u}}\le C
\norm{\GRADs \VEC{u}}$; see, e.g., \cite[Lemma 5.3.2]{Allaire:09}.

We assume in what follows that both $\mu$ and $\nu$ are piecewise constant on a finite polygonal or polyhedral partition $P_\Omega=\{\Omega_i\st 1\le i\le N_\Omega\}$ of the domain.
The assumption that $\nu$ is piecewise constant is often verified in practice in subsoil modelling. On the other hand, the assumption that $\mu$ is piecewise constant does not have a particular physical meaning, and should be regarded as a reasonable compromise which enables us to address all the relevant mathematical difficulties related to the use of the symmetric gradient without having to deal with unnecessary technicalities.
We notice, in passing, that the extension of the method to the case where $\mu$ and $\nu$ vary polynomially inside each element is straightforward, and the analysis can be modified following the ideas of \cite{Di-Pietro.Ern:15*1}.
The case of smoothly varying $\nu$ is treated numerically in Section \ref{sec:numerical.examples:darcy} below.
The extension to nonlinear viscous terms is possible following the ideas of \cite{Botti.Di-Pietro.ea:17}, inspired in turn by \cite{Droniou.Lamichhane:15,Di-Pietro.Droniou:17,Di-Pietro.Droniou:17*1}.
The case when $\nu$ is a full tensor is a special case of the above.


\section{Discrete setting}\label{sec:setting}

We consider a conforming simplicial mesh $\Th$ of $\Omega$, i.e., a set of triangular (if $d=2$) or tetrahedral (if $d=3$) elements such that (i) every $T\in\Th$ has non-empty interior; (ii) two distinct mesh elements $T_1,T_2\in\Th$ have disjoint interiors; (iii) the intersection of two disjoint mesh elements is either the empty set or a common vertex, edge, or face (the latter case only if $d=3$); (iv) it holds $h=\max_{T\in\Th}h_T$, where $h_T$ denotes the diameter of $T\in\Th$.
It is additionally assumed that $\Th$ is compliant with the partition $P_\Omega$ on which both $\mu$ and $\nu$ are piecewise constant and we let, for all $T\in\Th$,
$$
\mu_T\coloneq\mu_{|T},\qquad
\nu_T\coloneq\nu_{|T}
$$
denote their constant values inside $T$.

For any mesh element $T\in\Th$, we denote by $\Fh[T]$ the set of its edges (if $d=2$) or faces (if $d=3$).
For the sake of conciseness, the term face will be used henceforth for both the two- and three-dimensional cases.
For any $T\in\Th$ and any $F\in\Fh[T]$, we denote by $\normal_{TF}$ the unit vector normal to $F$ pointing out of $T$.
The sets of internal and boundary faces are respectively denoted by $\Fhi$ and $\Fhb$, and we set $\Fh\coloneq\Fhi\cup\Fhb$.
The diameter of a face $F\in\Fh$ is denoted by $h_F$.
For any mesh element $T\in\Th$, we denote by $\Fhi[T]\coloneq\Fh[T]\cap\Fhi$ the set of internal faces lying on the boundary $\partial T$ of $T$.

Our focus is on the $h$-convergence analysis, so we consider a sequence of refined meshes $(\Th)_{h\in{\cal H}}$, where ${\cal H}\subset\Real_+^*$ denotes a countable set of meshsizes having 0 as its unique accumulation point.
From this point on we assume, without necessarily recalling this fact at each occurrence, that the mesh sequence is regular, i.e., there exists a real number $\varrho>0$ such that, for all $h\in{\cal H}$ and all $T\in\Th$, $h_T/r_T\le\varrho$, with $r_T$ denoting the inradius of $T$.
This implies, in particular, that the diameter of one element is uniformly comparable to those of its faces.

To avoid the proliferation of generic constants, we will write $a\lesssim b$ to mean $a\le Cb$ with multiplicative constant $C>0$ independent of $h$ and, for local inequalities, of the mesh element or face, as well as on the problem data $\mu$, $\nu$, $\VEC{f}$, and $g$, and on the corresponding exact solution $(\VEC{u},p)$.
The notation $a\simeq b$ means $a\lesssim b\lesssim a$.
When useful, the dependencies of the hidden constant are further specified.

The construction underlying HHO methods hinges on projectors on local polynomial spaces.
Let $X$ denote an open bounded connect set of $\Real^n$ with $n\in\{1,2,3\}$ (in what follows, $X$ will typically represent a mesh element or face).
For a given integer $\ell\ge 0$, we denote by $\Poly{\ell}(X)$ the space spanned by the restriction to $X$ of $n$-variate, real-valued polynomials of total degree $\le\ell$.
The local $L^2$-orthogonal projector $\lproj[X]{\ell}:L^2(X)\to\Poly{\ell}(X)$ is defined as follows:
For any $v\in L^2(X)$, $\lproj[X]{\ell}v\in\Poly{\ell}(X)$ is the unique polynomial that satisfies
\begin{equation}\label{eq:lproj}
  (\lproj[X]{\ell}v-v,w)_X=0\qquad\forall w\in\Poly{\ell}(X).
\end{equation}
As a projector, $\lproj[X]{\ell}$ is linear and idempotent so that, in particular, it holds $\lproj[X]{\ell}v=v$ for all $v\in\Poly{\ell}(X)$.
The vector and tensor versions of the $L^2$-projector, both denoted by $\vlproj[X]{\ell}$, are obtained applying $\lproj[X]{\ell}$ component-wise.
The following boundedness property follows from \cite[Corollary 3.7]{Di-Pietro.Droniou:17}: For any $X$ mesh element or face, any $s\in\{0,\ldots,\ell+1\}$ and any function $v\in H^s(X)$, it holds that
\begin{equation}\label{eq:lproj:boundedness}
  \seminorm[H^s(X)]{\lproj[X]{\ell}v}\lesssim\seminorm[H^s(X)]{v},
\end{equation}
with hidden constant equal to 1 for $s=0$.
Optimal approximation properties for the $L^2$-orthogonal projector have also been proved in \cite{Di-Pietro.Droniou:17} in a very general setting.
For the present discussion, it will suffice to recall the following results, that are a special case of \cite[Lemmas 3.4 and 3.6]{Di-Pietro.Droniou:17}: Let an integer $s\in\{0,\ldots,\ell+1\}$ be given. Then, for any mesh element $T\in\Th$, any function $v\in H^s(T)$, and any exponent $m\in\{0,\ldots,s\}$, it holds that
\begin{equation}\label{eq:lproj:approx}
  \seminorm[H^m(T)]{v-\lproj[T]{\ell}v}\lesssim h_T^{s-m}\seminorm[H^s(T)]{v}.
\end{equation}
Moreover, if $s\ge 1$ and $m\le s-1$,
\begin{equation}\label{eq:lproj:approx.trace}
  \seminorm[{H^m(\Fh[T])}]{v-\lproj[T]{\ell}v}\lesssim h_T^{s-m-\frac12}\seminorm[H^s(T)]{v},
\end{equation}
where $H^m(\Fh[T])\coloneq\left\{v\in L^2(\partial T)\st v_{|F}\in H^m(F)\mbox{ for all }F\in\Fh[T]\right\}$ is the broken Sobolev space on $\Fh[T]$ and $\seminorm[{H^m(\Fh[T])}]{{\cdot}}$ the corresponding broken seminorm.

At the global level, we denote by $\Poly{\ell}(\Th)$ the space of broken polynomials on $\Th$ whose restriction to every mesh element $T\in\Th$ lies in $\Poly{\ell}(T)$.
The corresponding global $L^2$-orthogonal projector $\lproj{\ell}:L^2(\Omega)\to\Poly{\ell}(\Th)$ is such that, for all $v\in L^2(\Omega)$,
\begin{equation}\label{eq:lprojh}
  (\lproj{\ell}v)_{|T} \coloneq \lproj[T]{\ell}v_{|T}\qquad\forall T\in\Th.
\end{equation}
Also in this case, the vector version $\vlproj{\ell}:L^2(\Omega)^d\to\Poly{\ell}(\Th)^d$ is obtained applying $\lproj{\ell}$ component-wise.
The regularity requirements in the error estimates will be expressed in terms of the broken Sobolev spaces
$$
H^s(\Th)\coloneq\left\{
v\in L^2(\Omega)\st v_{|T}\in H^s(T)\quad\forall T\in\Th
\right\}.
$$


\section{Discrete problem}\label{sec:discrete}

In this section we formulate the discrete problem and state the main results of the analysis.

\subsection{Discrete unknowns}

Let an integer $k\ge 1$ be fixed and set
\begin{equation}\label{eq:l}
  l\coloneq\max(k-1,1).
\end{equation}
This choice for the polynomial degrees is motivated in Remark \ref{rem:k.l} below.
We define the following space of discrete velocity unknowns:
$$
\Uh\coloneq\left\{
\uvec{v}_h=((\VEC{v}_T)_{T\in\Th},(\VEC{v}_F)_{F\in\Fh})\st
\VEC{v}_T\in\Poly{l}(T)^d\quad\forall T\in\Th,\quad
\VEC{v}_F\in\Poly{k}(F)^d\quad\forall F\in\Fh
\right\}{.}
$$
For all $\uvec{v}_h\in\Uh$, $\VEC{v}_h$ (not underlined) denotes the function in $\Poly{l}(\Th)^d$ obtained patching element-based unknowns, that is to say,
\begin{equation}\label{eq:vh}
  (\VEC{v}_h)_{|T}\coloneq\VEC{v}_T\qquad\forall T\in\Th.
\end{equation}
The global interpolator $\Ih:H^1(\Omega)^d\to\Uh$ is such that, for all $\VEC{v}\in H^1(\Omega)^d$,
\begin{equation*}\label{eq:Ih}
  \Ih\VEC{v}\coloneq
  ((\vlproj[T]{l}\VEC{v}_{|T})_{T\in\Th},(\vlproj[F]{k}\VEC{v}_{|F})_{F\in\Fh}).
\end{equation*}
For any mesh element $T\in\Th$, the restrictions of $\Uh$ and $\uvec{v}_h\in\Uh$ to $T$ are respectively denoted by $\UT$ and $\uvec{v}_T=(\VEC{v}_T,(\VEC{v}_F)_{F\in\Fh[T]})$.
Similarly, the local interpolator $\IT:H^1(T)^d\to\UT$ is obtained restricting $\Ih$ to $T$, and is therefore such that, for all $\VEC{v}\in H^1(T)^d$,
\begin{equation}\label{eq:IT}
  \IT\VEC{v}\coloneq (\vlproj[T]{l}\VEC{v}, (\vlproj[F]{k}\VEC{v}_{|F})_{F\in\Fh[T]}).
\end{equation}
The spaces of discrete unknowns strongly accounting for the boundary condition \eqref{eq:bc} on the velocity and the zero-average constraint \eqref{eq:zero.mean.pressure} on the pressure are, respectively,
\begin{equation}\label{eq:UhD}
  \UhD\coloneq\left\{\uvec{v}_h\in\Uh\st\VEC{v}_F=\VEC{0}\quad\forall F\in\Fhb\right\},\qquad
  P_h^k\coloneq\Poly{k}(\Th)\cap P.
\end{equation}

\subsection{Stokes term}

Let an element $T\in\Th$ be fixed.
We define the local Stokes velocity reconstruction $\rST:\UT\to\Poly{k+1}(T)^d$ such that, for all $\uvec{v}_T\in\UT$,
\begin{subequations}\label{eq:rST}
  \begin{equation}\label{eq:rST:pde}
    (\GRADs\rST\uvec{v}_T,\GRADs\VEC{w})_T
    = -(\VEC{v}_T,\DIV\GRADs\VEC{w})_T + \sum_{F\in\Fh[T]}(\VEC{v}_F,\GRADs\VEC{w}\normal_{TF})_F
    \qquad\forall\VEC{w}\in\Poly{k+1}(T)^d.
  \end{equation}
  This equation defines $\rST\uvec{v}_T$ up to a rigid-body motion, which we prescribe by further imposing that
  \begin{equation}\label{eq:rST:closure}
    \int_T\rST\uvec{v}_T=\int_T\VEC{v}_T,\qquad
    \int_T\GRADss\rST\uvec{v}_T=\frac12\sum_{F\in\Fh[T]}\int_F\left(
    \VEC{v}_F\otimes\normal_{TF} - \normal_{TF}\otimes\VEC{v}_F
    \right),
  \end{equation}
\end{subequations}
where $\GRADss$ denotes the skew-symmetric part of the gradient operator and $\otimes$ the tensor product.
\begin{remark}[Link with the strain projector and approximation properties of the Stokes velocity reconstruction]
  Definition \eqref{eq:rST} can be justified observing that it holds, for all $\VEC{v}\in H^1(T)^d$,
  \begin{equation}\label{eq:rST.IT=elproj}
    \rST\IT\VEC{v} = \elproj{k+1}\VEC{v},
  \end{equation}
  where $\elproj{k+1}:H^1(T)^d\to\Poly{k+1}(T)^d$ is the strain projector defined by \eqref{eq:elproj} below, i.e., for all $\VEC{v}\in H^1(T)^d$, $\elproj{k+1}\VEC{v}$ is such that
  \begin{equation}\label{eq:elproj.k+1}
    \begin{gathered}
      (\GRADs\elproj{k+1}\VEC{v},\GRADs\VEC{w})_T
      = (\GRADs\VEC{v},\GRADs\VEC{w})_T\qquad\forall\VEC{w}\in\Poly{k+1}(T)^d,
      \\
      \int_T\elproj{k+1}\VEC{v}=\int_T\VEC{v},\qquad
      \int_T\GRADss\elproj{k+1}\VEC{v}=\int_T\GRADss\VEC{v}.
    \end{gathered}
  \end{equation}
  To prove \eqref{eq:rST.IT=elproj}, write \eqref{eq:rST} with $\IT\VEC{v}$ instead of $\uvec{v}_T$, use the definition \eqref{eq:lproj} of $\vlproj[T]{l}$ and $\vlproj[F]{k}$ to cancel these projectors from the right-hand sides of \eqref{eq:rST:pde} and \eqref{eq:rST:closure}, integrate by parts the right-hand sides of \eqref{eq:rST:pde} and of the second equation in \eqref{eq:rST:closure}, and compare the result with \eqref{eq:elproj.k+1}.
  In order to deduce from \eqref{eq:rST.IT=elproj} that, for any $T\in\Th$ with $\Th$ belonging to a regular mesh sequence, $\rST\IT\VEC{v}$ optimally approximates $\VEC{v}$ in $\Poly{k+1}(T)^d$, it suffices to apply Theorem \ref{thm:elproj:approx} and Corollary \ref{cor:elproj:approx.trace} below with $X=T$ and $\ell=k+1$ after observing that the multiplicative constants in \eqref{eq:elproj:approx} and \eqref{eq:elproj:approx.trace} do not depend on $h$ or $T$, but only on $\varrho$.
\end{remark}
The Stokes term is discretised by means of the bilinear form $\mathrm{a}_{\stokes,h}:\Uh\times\Uh\to\Real$ such that, for all $\uvec{w}_h,\uvec{v}_h\in\Uh$,
\begin{equation}\label{eq:aSh}
  \mathrm{a}_{\stokes,h}(\uvec{w}_h,\uvec{v}_h)
  \coloneq\sum_{T\in\Th}\mathrm{a}_{\stokes,T}(\uvec{w}_T,\uvec{v}_T)
\end{equation}
with local contribution such that
\begin{equation}\label{eq:aST}
  \mathrm{a}_{\stokes,T}(\uvec{w}_T,\uvec{v}_T)
  \coloneq 2\mu_T(\GRADs\rST\uvec{w}_T,\GRADs\rST\uvec{v}_T)_T
  + \mathrm{s}_{\stokes,T}(\uvec{w}_T,\uvec{v}_T).
\end{equation}
The first term in $\mathrm{a}_{\stokes,T}$ is the usual Galerkin contribution responsible for consistency, while the second is a stabilisation term satisfying the following assumption, which will be implicitly kept throughout the rest of the exposition.
\begin{assumption}[Stokes stabilisation bilinear form]\label{ass:sST}
  The Stokes stabilisation bilinear form $\mathrm{s}_{\stokes,T}:\UT\times\UT\to\Real$ enjoys the following properties:
  \begin{compactenum}[(S1)]
  \item \emph{Symmetry and positivity.} $\mathrm{s}_{\stokes,T}$ is symmetric and positive semidefinite;
  \item \emph{Stability and boundedness.} It holds, for all $\uvec{v}_T\in\UT$,
    \begin{equation}\label{eq:aST.stability}
      \mathrm{a}_{\stokes,T}(\uvec{v}_T,\uvec{v}_T)^{\frac12}
      \eqcolon\norm[\stokes,T]{\uvec{v}_T}\simeq (2\mu_T)^{\frac12}\norm[\strain,T]{\uvec{v}_T}
    \end{equation}
    with local discrete strain seminorm
    \begin{equation}\label{eq:norm.ET}
      \norm[\strain,T]{\uvec{v}_T}\coloneq\left(
      \norm[T]{\GRADs\VEC{v}_T}^2
      + \seminorm[1,\partial T]{\uvec{v}_T}^2
      \right)^{\frac12},\qquad
      \seminorm[1,\partial T]{\uvec{v}_T}
      \coloneq\left(
      \sum_{F\in\Fh[T]}\frac{1}{h_F}\norm[F]{\VEC{v}_F-\VEC{v}_T}^2
      \right)^{\frac12};
    \end{equation}
  \item \emph{Polynomial consistency.} For all $\VEC{w}\in\Poly{k+1}(T)^d$ and all $\uvec{v}_T\in\UT$, it holds that
    \begin{equation*}\label{eq:sST:poly.cons}
      \mathrm{s}_{\stokes,T}(\IT\VEC{w},\uvec{v}_T) = 0.
    \end{equation*}
  \end{compactenum}
\end{assumption}
\begin{remark}[Global stability and boundedness for the Stokes bilinear form]
  Raising \eqref{eq:aST.stability} to the power $2$, summing over $T\in\Th$, accounting for \eqref{eq:mu.nu:bounds}, and passing to the square root, we infer the following global uniform seminorm equivalence valid for all $\uvec{v}_h\in\Uh$:
  \begin{equation}\label{eq:aSh.stability}
    (2\underline{\mu})^{\frac12}\norm[\strain,h]{\uvec{v}_h}
    \lesssim\norm[\stokes,h]{\uvec{v}_h}
    \lesssim(2\overline{\mu})^{\frac12}\norm[\strain,h]{\uvec{v}_h}
  \end{equation}
  with
  \begin{equation}\label{eq:norm.S.strain.h}
    \norm[\stokes,h]{\uvec{v}_h}\coloneq\mathrm{a}_{\stokes,h}(\uvec{v}_h,\uvec{v}_h)^{\frac12}
    \mbox{ and }
    \norm[\strain,h]{\uvec{v}_h}\coloneq\left(\sum_{T\in\Th}\norm[\strain,T]{\uvec{v}_T}^2\right)^{\frac12}.
  \end{equation}
  Adapting the reasoning of \cite[Proposition 5]{Di-Pietro.Ern:15}, one can prove that the map $\norm[\strain,h]{{\cdot}}$ defines a norm on the space $\UhD$ with strongly enforced boundary conditions.
\end{remark}
The following stabilisation bilinear form, classical in HHO methods, fulfils Assumption \ref{ass:sST}:
\begin{equation}\label{eq:sST}
  \mathrm{s}_{\stokes,T}(\uvec{w}_T,\uvec{v}_T)_T
  \coloneq \sum_{F\in\Fh[T]}\frac{2\mu_T}{h_F} ((\dTF-\dT)\uvec{w}_T,(\dTF-\dT)\uvec{v}_T)_F.
\end{equation}
Here, the Stokes difference operators $\dT:\UT\to\Poly{l}(T)^d$ and, for all $F\in\Fh[T]$, $\dTF:\UT\to\Poly{k}(F)^d$ are such that, for all $\uvec{v}_T\in\UT$,
$$
\dT\uvec{v}_T\coloneq\vlproj[T]{l}(\rST\uvec{v}_T-\VEC{v}_T),\qquad
\dTF\uvec{v}_T\coloneq\vlproj[F]{k}(\rST\uvec{v}_T-\VEC{v}_F)\quad\forall F\in\Fh[T].
$$
\begin{remark}[Choice of the polynomial degrees for the discrete velocity unknowns]\label{rem:k.l}
  The assumption $k\ge 1$ and the choice \eqref{eq:l} of the degree for element-based discrete unknowns (which implies, in particular, $l=1$ when $k=1$) are required to prove condition (S2) for the bilinear form defined by \eqref{eq:sST}.
  The key point is to ensure that rigid-body motions and their traces are captured by element and face unknowns, respectively.
  For further insight into this point, we refer the reader to \cite[Lemma 4]{Di-Pietro.Ern:15}, where \eqref{eq:aST.stability} is proved for a variation of the stabilisation bilinear form \eqref{eq:sST} corresponding to the case $l=k$.
  Stability for $k=0$ could be recovered by penalising the jumps of the Stokes velocity reconstruction similarly to \cite{Burman.Hansbo:05}. %
  This modification would, however, introduce additional links among element-based velocity unknowns, so that the static condensation strategy discussed in Remark \ref{rem:static.cond} below would no longer be an interesting option. %
  Further details on this point are postponed to a future work.

  In passing, we notice that, when $\mu$ is constant, problem \eqref{eq:strong} can be simplified replacing the term $-\DIV(2\mu\GRADs\VEC{u})$ by $-\mu\LAPL\VEC{u}$.
  In this case, the discretisation of the Stokes term can go along the lines of \cite{Di-Pietro.Ern.ea:14,Cockburn.Di-Pietro.ea:16} with $k\ge 0$ and $l=\max(0,k-1)$.
\end{remark}

\subsection{Darcy term}

Let an element $T\in\Th$ be fixed, and denote by $\RTN(T)\coloneq\Poly{k}(T)^d+\VEC{x}\Poly{k}(T)$ the Raviart--Thomas--N\'ed\'elec space of degree $k$ on $T$.
We define the local Darcy velocity reconstruction $\rDT:\UT\to\RTN(T)$ such that, for all $\uvec{v}_T\in\UT$,
\begin{subequations}\label{eq:rDT}
  \begin{alignat}{3}\label{eq:rDT:T}
    && (\rDT\uvec{v}_T,\VEC{w})_T &= (\VEC{v}_T,\VEC{w})_T &\qquad& \forall \VEC{w}\in\Poly{k-1}(T)^d
    \\ \label{eq:rDT:F}
    &&
    (\rDT\uvec{v}_T\SCAL\normal_{TF},q)_F &= (\VEC{v}_F\SCAL\normal_{TF},q)_F &\qquad&\forall F\in\Fh[T]\,,\; \forall q\in\Poly{k}(F).
  \end{alignat}
\end{subequations}
Classically, the relations \eqref{eq:rDT} identify $\rDT\uvec{v}_T$ uniquely; see, e.g., \cite[Proposition 2.3.4]{Boffi.Brezzi.ea:13}.
For further use, we also define the global Darcy velocity reconstruction $\rDh:\Uh\to\RTN(\Th)$ with $\RTN(\Th)\coloneq\{\VEC{v}\in\Hdiv\st\VEC{v}_{|T}\in\RTN(T)\text{ for all }T\in\Th\}$ such that, for all $\uvec{v}_h\in\Uh$,
$$
(\rDh\uvec{v}_h)_{|T}\coloneq\rDT\uvec{v}_T\qquad\forall T\in\Th.
$$
Some remarks are in order.
\begin{remark}[Reformulation of \eqref{eq:rDT}]\label{rem:rDT}
  Conditions \eqref{eq:rDT:T} and \eqref{eq:rDT:F} are respectively equivalent to
  \begin{equation}\label{eq:rDT.bis}
    \text{$\vlproj[T]{k-1}(\rDT\uvec{v}_T)=\vlproj[T]{k-1}\VEC{v}_T$ and $(\rDT\uvec{v}_T)_{|F}\SCAL\normal_{TF}=\VEC{v}_F\SCAL\normal_{TF}$ for all $F\in\Fh[T]$.}
  \end{equation}
  In particular, accounting for \eqref{eq:l} and using the idempotency of $\vlproj[T]{k-1}$, the former condition implies $\vlproj[T]{l}(\rDT\uvec{v}_T)=\vlproj[T]{k-1}(\rDT\uvec{v}_T)=\VEC{v}_T$ when $k\ge 2$.
\end{remark}
\begin{remark}[Link with the Raviart--Thomas--N\'ed\'elec interpolator]
  A direct verification shows that, for all $T\in\Th$, the local Darcy velocity reconstruction composed with the local interpolator \eqref{eq:IT} gives the Raviart--Thomas--N\'ed\'elec interpolator, i.e., for all $\VEC{v}\in H^1(T)^d$
  \begin{equation}\label{eq:rDh.Ih=IRTN}
    \rDT\IT \VEC{v} = \IRTNT\VEC{v},
  \end{equation}
  where $\IRTNT:H^1(T)^d\to\RTN(T)$ is such that, for all $\VEC{v}\in H^1(T)^d$,
  \begin{subequations}\label{eq:IRTNT}
    \begin{alignat}{3}\label{eq:IRTNT:T}
      &&(\IRTNT\VEC{v},\VEC{w})_T &= (\VEC{v},\VEC{w})_T &\qquad&\forall\VEC{w}\in\Poly{k-1}(T)^d,
      \\ \label{eq:IRNTN:F}
      &&
      (\IRTNT\VEC{v}\SCAL\normal_{TF},q)_F &= (\VEC{v}\SCAL\normal_{TF},q)_F &\qquad&\forall F\in\Fh[T] \,,\;\forall q\in\Poly{k}(F).
    \end{alignat}
  \end{subequations}
\end{remark}
The Darcy term is discretised by means of the bilinear form $\mathrm{a}_{\darcy,h}:\Uh\times\Uh\to\Real$ such that, for all $\uvec{w}_h,\uvec{v}_h\in\Uh$,
\begin{equation}\label{eq:aDh}
  \mathrm{a}_{\darcy,h}(\uvec{w}_h,\uvec{v}_h)
  \coloneq\sum_{T\in\Th}\mathrm{a}_{\darcy,T}(\uvec{w}_T,\uvec{v}_T)
\end{equation}
with local contribution
$$
\mathrm{a}_{\darcy,T}(\uvec{w}_T,\uvec{v}_T)
\coloneq\nu_T(\rDT\uvec{w}_T,\rDT\uvec{v}_T)_T
+ \mathrm{s}_{\darcy,T}(\uvec{w}_T,\uvec{v}_T).
$$
Once again, the first term in the right-hand side of the above expression is responsible for consistency, while the second is the following stabilisation bilinear form, which plays a crucial role in the Darcy limit (see also Remark \ref{rem:darcy} on this subject):
\begin{equation}\label{eq:sDT}
  \mathrm{s}_{\darcy,T}(\uvec{w}_T,\uvec{v}_T)
  \coloneq \nu_T(\dDT\uvec{w}_T,\dDT\uvec{v}_T)_T
  + \sum_{F\in\Fhi[T]}\nu_T h_F(\dDTF\uvec{w}_T,\dDTF\uvec{v}_T)_F,
\end{equation}
with Darcy difference operators $\dDT:\UT\to\Poly{l}(T)^d$ and, for all $F\in\Fh[T]$, $\dDTF:\UT\to\Poly{k}(F)^d$ such that, for all $\uvec{v}_T\in\UT$,
\begin{equation}\label{eq:dDT.dDTF}
  \dDT\uvec{v}_T\coloneq\vlproj[T]{l}(\rDT\uvec{v}_T-\VEC{v}_T),\qquad
  \dDTF\uvec{v}_T\coloneq\vlproj[F]{k}(\rDT\uvec{v}_T-\VEC{v}_F)\quad\forall F\in\Fh[T].
\end{equation}
Recalling the characterisation \eqref{eq:rDT.bis} of the local Darcy velocity together with the definition \eqref{eq:dDT.dDTF} of the Darcy difference operator $\dDT$, it holds that
\begin{equation}\label{eq:dDT=0.k>=2}
  \dDT\uvec{v}_T=\VEC{0}\mbox{ if }k\ge 2.
\end{equation}
The role of the stabilisation term is illustrated by the following proposition.
\begin{proposition}[Darcy norm]\label{prop:norm.Dh}
  The function that maps every $\uvec{v}_h\in\Uh$ on
  \begin{equation}\label{eq:aDh.coercivity}
    \norm[\darcy,h]{\uvec{v}_h}\coloneq\left(\sum_{T\in\Th}\norm[\darcy,T]{\uvec{v}_T}^2\right)^{\frac12}
    \mbox{ where, for all $T\in\Th$, }
    \norm[\darcy,T]{\uvec{v}_T}\coloneq\mathrm{a}_{\darcy,T}(\uvec{v}_T,\uvec{v}_T)^{\frac12}
  \end{equation}
  is a norm on $\UhD$.
\end{proposition}
\begin{proof}
  The seminorm property being evident, it suffices to prove that, for all $\uvec{v}_h\in\UhD$, $\norm[\darcy,h]{\uvec{v}_h}=0$ implies $\uvec{v}_h=\uvec{0}$.
  Let $\uvec{v}_h\in\UhD$ be such that $\norm[\darcy,h]{\uvec{v}_h}=0$.
  Then, we have that
  \begin{alignat}{3}
    \forall T\in\Th&\qquad&\rDT\uvec{v}_T&=\VEC{0},\label{eq:rDT=0}
    \\
    \forall T\in\Th&\qquad&\dDT\uvec{v}_T&=\VEC{0},\label{eq:dDT=0}
    \\
    \forall T\in\Th&\qquad&\dDTF\uvec{v}_F&=\VEC{0} &\qquad&\forall F\in\Fhi[T].\label{eq:dDTF=0}
  \end{alignat}
  Plugging condition \eqref{eq:rDT=0} into \eqref{eq:dDT=0} and \eqref{eq:dDTF=0} we infer, respectively, that $\vlproj[T]{l}\VEC{v}_T=\VEC{v}_T=\VEC{0}$ for all $T\in\Th$ and $\vlproj[F]{k}\VEC{v}_F=\VEC{v}_F=\VEC{0}$ for all $F\in\Fhi$.
  On the other hand, by definition \eqref{eq:UhD} of $\UhD$, $\VEC{v}_F=\VEC{0}$ for all $F\in\Fhb$, which concludes the proof.
\end{proof}

\subsection{Velocity--pressure coupling}

The velocity--pressure coupling is realised by the bilinear form $\mathrm{b}_h:\Uh\times\Poly{k}(\Th)$ such that, for all $(\uvec{w}_h,q_h)\in\Uh\times\Poly{k}(\Th)$,
\begin{equation}\label{eq:bh}
  \mathrm{b}_h(\uvec{w}_h,q_h)
  \coloneq\sum_{T\in\Th}\left(
  (\VEC{w}_T,\GRAD q_T)_T - \sum_{F\in\Fh[T]}(\VEC{w}_F,q_T\normal_{TF})_F
  \right),
\end{equation}
where, for all $T\in\Th$, we have let, for the sake of brevity, $q_T\coloneq q_{h|T}$.
This choice is motivated by the following property.
\begin{proposition}[Consistency of the velocity--pressure coupling bilinear form]
  For all $\VEC{w}\in H^1(\Omega)^d$ and all $q_h\in\Poly{k}(\Th)$, it holds that
  \begin{equation}\label{eq:bh.consistency}
    \mathrm{b}_h(\Ih\VEC{w},q_h)=-(\DIV\VEC{w},q_h).
  \end{equation}
\end{proposition}
\begin{proof}
  Writing \eqref{eq:bh} for $\uvec{w}_h=\Ih\VEC{w}$, we obtain
  $$
  \begin{aligned}
    \mathrm{b}_h(\Ih\VEC{w},q_h)
    &= \sum_{T\in\Th}\left(
    (\vlproj[T]{l}\VEC{w},\GRAD q_T)_T - \sum_{F\in\Fh[T]}(\vlproj[F]{k}\VEC{w},q_T\normal_{TF})_F
    \right)
    \\
    &= \sum_{T\in\Th}\left(
    (\VEC{w},\GRAD q_T)_T - \sum_{F\in\Fh[T]}(\VEC{w},q_T\normal_{TF})_F
    \right)
    =-(\DIV\VEC{w},q_h),
  \end{aligned}
  $$
  where we have used the fact that, for all $T\in\Th$, $\GRAD q_T\in\Poly{k-1}(T)^d\subset\Poly{l}(T)^d$ (see \eqref{eq:l}) and, for all $F\in\Fh[T]$, $q_{T|F}\normal_{TF}\in\Poly{k}(F)^d$ together with \eqref{eq:lproj} to remove the projectors in the second line, and an element by element integration by parts to conclude.
\end{proof}%
The following proposition establishes a link between the divergence of the Darcy velocity reconstruction and the bilinear form $\mathrm{b}_h$.
As we will see in Remark \ref{rem:darcy}, this property plays a key role when extending the method to the Darcy problem.
\begin{proposition}[Link with the divergence of the Darcy velocity reconstruction]\label{prop:bh.div.rDh}
  For all $\uvec{v}_h\in\Uh$ and all $q_h\in\Poly{k}(\Th)$, it holds that
  \begin{equation}\label{eq:bh.div.rDh}
    \mathrm{b}_h(\uvec{v}_h,q_h) = -(\DIV\rDh\uvec{v}_h,q_h).
  \end{equation}
\end{proposition}
\begin{proof}
  We have that
  $$
  \begin{aligned}
    -(\DIV\rDh\uvec{v}_h,q_h)
    &= \sum_{T\in\Th}\left(
    (\rDT\uvec{v}_T,\GRAD q_T)_T - \sum_{F\in\Fh[T]}(\rDT\uvec{v}_T\SCAL\normal_{TF},q_T)_F
    \right)
    \\
    &= \sum_{T\in\Th}\left(
    (\VEC{v}_T,\GRAD q_T)_T - \sum_{F\in\Fh[T]}(\VEC{v}_F\SCAL\normal_{TF},q_T)_F
    \right)
    = \mathrm{b}_h(\uvec{v}_h,q_h),
  \end{aligned}
  $$
  where we have used an element by element integration by parts in the first line,
  the definition \eqref{eq:rDT} of the Darcy velocity in the second line after observing that, for any $T\in\Th$, $\GRAD q_T\in\Poly{k-1}(T)^d$ and $q_{T|F}\in\Poly{k}(F)$ for all $F\in\Fh[T]$,
  and recalled the definition \eqref{eq:bh} of $\mathrm{b}_h$ to conclude.
\end{proof}%

\subsection{Discrete problem and main results}

We define the global bilinear form $\mathrm{a}_h:\Uh\times\Uh\to\Real$ such that 
$$
\mathrm{a}_h\coloneq \mathrm{a}_{\stokes,h} + \mathrm{a}_{\darcy,h}
$$
with bilinear forms in the right-hand side respectively defined by \eqref{eq:aSh} and \eqref{eq:aDh}.
The discrete problem reads: Find $(\uvec{u}_h,p_h)\in\UhD\times P_h^k$ such that
\begin{subequations}\label{eq:discrete}
  \begin{alignat}{2}
    \label{eq:discrete:momentum}
    \mathrm{a}_h(\uvec{u}_h,\uvec{v}_h) + \mathrm{b}_h(\uvec{v}_h,p_h)
    &= (\VEC{f},\rDh\uvec{v}_h) &\qquad&\forall\uvec{v}_h\in\UhD,
    \\ \label{eq:discrete:mass}
    -\mathrm{b}_h(\uvec{u}_h,q_h) &= (g,q_h) &\qquad&\forall q_h\in P_h^k.
  \end{alignat}
\end{subequations}
\begin{remark}[Static condensation]\label{rem:static.cond}
  The size of the linear system corresponding to the discrete problem~\eqref{eq:discrete} can be significantly reduced by resorting to static condensation.
  Following the procedure hinted to in~\cite{Aghili.Boyaval.ea:15} and detailed in~\cite[Section~6.2]{Di-Pietro.Ern.ea:16}, it can be shown that the only globally coupled variables are the face unknowns for the velocity and the mean value of the pressure inside each mesh element.
  Hence, after statically condensing the other discrete unknowns, the size of the linear system matrix is
  \begin{equation}\label{eq:Ndof}
    N_{\rm dof} \coloneq d{k+d-1\choose k} \card{\Fhi} + \card{\Th}.
  \end{equation}
\end{remark}
We start by studying the well-posedness of problem \eqref{eq:discrete}.
We equip henceforth $\UhD$ with the norm such that, for all $\uvec{v}_h\in\UhD$,
\begin{equation}\label{eq:norm.Uh}
  \norm[\VEC{U},h]{\uvec{v}_h}\coloneq\left(
  \sum_{T\in\Th}\norm[\VEC{U},T]{\uvec{v}_T}^2
  \right)^{\frac12}\mbox{ where, for all $T\in\Th$, }
  \norm[\VEC{U},T]{\uvec{v}_T}\coloneq\left(
  \norm[\stokes,T]{\uvec{v}_T}^2 + \norm[\darcy,T]{\uvec{v}_T}^2
  \right)^{\frac12}
\end{equation}
with local Stokes and Darcy (semi)norms respectively defined by \eqref{eq:aST.stability} and \eqref{eq:aDh.coercivity}.
Given a linear functional $\mathfrak{f}$ on $\UhD$, its dual  norm is classically given by
\begin{equation}\label{eq:norm.U*h}
  \norm[\VEC{U}^*,h]{\mathfrak{f}}
  \coloneq\sup_{\uvec{v}_h\in\UhD\setminus\{\uvec{0}\}}
  \frac{\left|\langle\mathfrak{f},\uvec{v}_h\rangle\right|}{\norm[\VEC{U},h]{\uvec{v}_h}}.
\end{equation}
\begin{theorem}[Well-posedness]\label{thm:well-posedness}
  Problem \eqref{eq:discrete} is well-posed with a priori bound:
  \begin{equation}\label{eq:a-priori}
    \norm[\VEC{U},h]{\uvec{u}_h}
    + \beta\norm{p_h}
    \lesssim(2\underline{\mu})^{-\frac12}\norm{\VEC{f}}
    + \beta^{-1}\norm{g}
    \mbox{ with }\beta\coloneq\left( 2\overline{\mu} + \overline{\nu} \right)^{-\frac12}.
  \end{equation}
\end{theorem}
\begin{proof}
  See Section \ref{sec:proofs:well-posedness}.
\end{proof}
We next investigate the convergence of the method.
We measure the error as the difference between the discrete solution and the interpolate of the exact solution defined as
$$
(\hat{\uvec{u}}_h,\hat p_h)\coloneq (\Ih\VEC{u},\lproj{k} p)\in\UhD\times P_h^k.
$$
After noticing that, for any $q_h\in P_h^k$,
\begin{equation}\label{eq:exact.vel.press}
-\mathrm{b}_h(\hat{\uvec{u}}_h,q_h)
=(\DIV\VEC{u},q_h)=(g,q_h)
=-\mathrm{b}_h(\uvec{u}_h,q_h)
\end{equation}
owing to the consistency property \eqref{eq:bh.consistency} of $\mathrm{b}_h$ together with the continuous \eqref{eq:weak:mass} and discrete \eqref{eq:discrete:mass} mass conservation equations, it is a simple matter to check that the discretisation error
$$
(\uvec{e}_h,\epsilon_h)\coloneq (\uvec{u}_h - \hat{\uvec{u}}_h, p_h - \hat{p}_h)
$$
solves the following problem:
\begin{equation}\label{eq:error.eq}
  \begin{alignedat}{2}
    \mathrm{a}_h(\uvec{e}_h,\uvec{v}_h) + \mathrm{b}_h(\uvec{v}_h,\epsilon_h)
    &= \langle\err,\uvec{v}_h\rangle
    &\qquad&\forall\uvec{v}_h\in\UhD,
    \\
    -\mathrm{b}_h(\uvec{e}_h,q_h) &= 0
    &\qquad&\forall q_h\in P_h^k,
  \end{alignedat}
\end{equation}
where $\err$ is the linear functional on $\UhD$ representing the consistency error and such that, for all $\uvec{v}_h\in\UhD$,
\begin{equation}\label{eq:Rup}
  \langle\err,\uvec{v}_h\rangle
  \coloneq (\VEC{f},\rDh\uvec{v}_h)
  - \mathrm{a}_h(\hat{\uvec{u}}_h,\uvec{v}_h) - \mathrm{b}_h(\uvec{v}_h,\hat{p}_h).
\end{equation}

\begin{theorem}[Error estimates and convergence]\label{thm:err.est}
  Denote by $(\VEC{u},p)\in\VEC{U}\times P$ and by $(\uvec{u}_h,p_h)\in\UhD\times P_h^k$ the unique solutions to \eqref{eq:weak} and \eqref{eq:discrete}, respectively.
  Then, the following error estimate holds with $\beta$ defined by \eqref{eq:a-priori}:
  \begin{equation}\label{eq:err.est}
    \norm[\VEC{U},h]{\uvec{e}_h} + \beta\norm{\epsilon_h}
    \lesssim\norm[\VEC{U}^*,h]{\err}.
  \end{equation}
  Moreover, assuming the additional regularity $\VEC{u}\in H^{k+2}(\Th)^d$ and $p\in H^1(\Omega)$, it holds that
  \begin{multline}\label{eq:conv.rate}
    \norm[\VEC{U}^*,h]{\err}\lesssim
    \\
    \left[
      \sum_{T\in\Th}\left(
      (2\mu_T)\min(1,\Cf^{-1})h_T^{2(k+1)}\seminorm[H^{k+2}(T)^d]{\VEC{u}}^2
      + \alpha_\mu \nu_T\min(1,\Cf)h_T^{2(k+1)}\seminorm[H^{k+1}(T)^d]{\VEC{u}}^2
      \right)
      \right]^{\frac12}
  \end{multline}
  where, for all $T\in\Th$, we have introduced the local friction coefficient 
  \begin{equation}\label{eq:Cf}
    \Cf\coloneq\frac{\nu_T h_T^2}{2\mu_T}
  \end{equation}
  with the convention that $\Cf^{-1}\coloneq+\infty$ if $\nu_T=0$, and we have set
  \begin{equation}\label{eq:alpha.mu}
    \alpha_\mu\coloneq\begin{cases}
    \overline{\mu}/\underline{\mu} & \text{if $k=1$ and $\min_{T\in\Th}\Cf\le 1$},
    \\
    1 & \text{otherwise.}
    \end{cases}
  \end{equation}
\end{theorem}
\begin{proof}
  See Section \ref{sec:proofs:convergence}.
\end{proof}
Some remarks are in order.
\begin{remark}[Robustness of the error estimate]
  The error estimate \eqref{eq:err.est} is robust across the entire range of values $\Cf\in [0,+\infty)$ (and, as we will see in the next remark, $+\infty$ can also be included) thanks to the presence of the cutoff factors $\min(1,\Cf^{-1})$ and $\min(1,\Cf)$ that prevent the multiplicative constants in the right-hand side from exploding.
    Those mesh elements for which $\Cf<1$ are in the Stokes-dominated regime and, correspondingly, the first contribution inside the sum in \eqref{eq:conv.rate} dominates.
  On the other hand, those elements for which $\Cf>1$ are in the Darcy-dominated regime and, correspondingly, the second contribution dominates.
  Since the method is designed so that these contributions are equilibrated, convergence in $\mathcal{O}(h^{k+1})$ is attained irrespectively of the local regime.
  The specific forms of the Stokes and Darcy velocity reconstructions play a key role in attaining this goal; see the discussion in Remark \ref{rem:darcy} below.
  Comparing, e.g., with \cite[Theorem 3.2]{Konno.Stenberg:11}, where a Raviart--Thomas--N\'ed\'elec approximation of the velocity is used also in the Stokes term, we gain one order of convergence in the Stokes-dominated regime.
  Similar considerations hold for the Virtual Element method of \cite{Vacca:18}, see in particular the error estimate in Theorem 5.2 therein.
\end{remark}
\begin{remark}[Application to the Darcy problem]\label{rem:darcy}
  Assume $\underline{\nu}>0$.
  A close inspection of the proofs in Section \ref{sec:proofs} below reveals that the proposed method can be used also when $\mu=0$ formally setting $\Cf\coloneq+\infty$ for all $T\in\Th$.  
  In this case, denoting by $\gamma_{\normal}$ the normal trace operator on $\partial\Omega$, the velocity space becomes $\VEC{U}=\left\{\VEC{v}\in\VEC{H}({\rm div};\Omega)\st \gamma_{\normal}(\VEC{v})=0\mbox{ on }\partial\Omega\right\}$, and \eqref{eq:weak} coincides with the mixed formulation of the Darcy problem.
  In particular, the well-posedness results of Theorem \ref{thm:well-posedness} remain valid replacing the term $(2\underline{\mu})^{-1}\norm{\VEC{f}}$ by $\underline{\nu}^{-1}\norm{\VEC{f}}$ in \eqref{eq:a-priori}, and so is the case for the error estimates of Theorem \ref{thm:err.est} under the regularity $\VEC{u}\in H^1(\Omega)^d\cap H^{k+1}(\Th)^d$.

  The key point to achieve well-posedness when $\mu=0$ is the introduction of the stabilisation term \eqref{eq:sDT} in the local Darcy bilinear form.
  Thanks to this term, we can control the discrete unknowns that are not controlled by the $L^2$-norm of the Darcy velocity reconstruction, namely the tangential velocity unknowns on interfaces and the linear component of the element unknowns when $k=1$; see Proposition \ref{prop:norm.Dh}.
  The tangential components of velocity unknowns on boundary faces, on the other hand, are set to zero in the definition \eqref{eq:UhD} of the space $\UhD$, and do not appear in the formulation of the method when $\mu=0$.
  This means that they are discarded, coherently with the fact that we cannot enforce their value when $\mu=0$.
  It is precisely for this reason that the boundary term in \eqref{eq:sDT} is only taken on interfaces.

  The key point to retain convergence in $h^{k+1}$ when $\mu=0$ is the specific form \eqref{eq:rDT} of the Darcy velocity reconstruction, and its use both in the Darcy contribution and in the source term in \eqref{eq:discrete:momentum}.
  The role of this choice is to make the term $\term_4$ in the proof of Theorem \ref{thm:err.est} vanish (the corresponding crucial property is stated in Proposition \ref{prop:bh.div.rDh}).
  More trivial discretisations of the Darcy term (obtained, e.g., by taking for all $T\in\Th$ the element unknowns in $\Poly{k}(T)^d$ and setting $\rDT\uvec{v}_T=\VEC{v}_T$) would reduce by one the order of convergence of the method.
  Using a discretisation of the Darcy contribution inspired by the Mixed High-Order method of \cite{Di-Pietro.Ern:17}, on the other hand, would reduce by one the convergence rate for $\mu\neq 0$.
  As a matter of fact, the convergence in $h^{k+1}$ for this choice is intimately linked to the fact that $\nu\VEC{u}$ is a gradient (which is true for the Darcy problem but not for the Brinkman problem).

  We conclude this remark by noticing that the method for the Darcy problem can also be extended to treat the case $k=l=0$. This point is numerically demonstrated in Section \ref{sec:numerical.examples}.
\end{remark}
\begin{remark}[Pressure-robustness]
  It is also interesting to notice that the right-hand side of the error estimate \eqref{eq:err.est} does not depend on the pressure.
  The key to that property is the exact formula \eqref{eq:exact.vel.press} that relates the velocity--pressure coupling applied to the approximate velocity $\uvec{u}_h$ and the interpolant $\hat{\uvec{u}}_h=\Ih\VEC{u}$ of the exact velocity.
  As pointed out in \cite{Di-Pietro.Ern.ea:16} and references therein, this means that the proposed method is robust with respect to source terms $\VEC{f}$ with a large irrotational part.
\end{remark}


\section{Numerical examples}\label{sec:numerical.examples}

In this section we present some numerical examples.

\subsection{Convergence for the Darcy, Brinkman, and Stokes problems with constant coefficients}
\label{sec:numerical.examples:analytical}

We start by assessing the convergence rates predicted by Theorem \ref{thm:err.est} in various regimes.
Set $\Omega\coloneq(0,2)\times(-1,1)$, and define the global friction coefficient $\Cf[\Omega]\coloneq\frac{\nu}{\mu}$, corresponding to a unit global reference length.
We consider the family of solution parametrised by $\Cf[\Omega]\in[0,+\infty]$ such that, setting $\chi_\stokes(\xi)\coloneq \exp(\xi)^{-1}$ for any $\xi\in\Real^+$ and $\chi_\stokes(+\infty)\coloneq 0$, it holds for any $\VEC{x}\in\Omega$,
\begin{equation}\label{eq:analytical.u.p}
  \VEC{u}(\VEC{x}) =
  \chi_\stokes\left(\Cf[\Omega]\right) \VEC{u}_\stokes(\VEC{x})  
  + (1-\chi_\stokes)\left(\Cf[\Omega]\right) \VEC{u}_\darcy(\VEC{x}),\qquad
  p(\VEC{x})\coloneq\cos x_1\sin x_2-p_0,
\end{equation}
where $p_0\in\Real$ 
is such that the zero average condition on $p$ is verified and, defining the stream function $\psi(\VEC{x})\coloneq -\sin x_1\cos x_2$, we have set
$$
\VEC{u}_\darcy(\VEC{x})\coloneq
\begin{cases}
  -\nu^{-1}\GRAD p(\VEC{x}) & \text{if $\nu\neq 0$}, \\
  \VEC{0} & \text{otherwise},
\end{cases}\qquad
\VEC{u}_\stokes(\VEC{x})\coloneq
\curl\psi(\VEC{x}).
$$
The boundary condition on $\VEC{u}$ if $\Cf[\Omega]<+\infty$ or $\VEC{u}\cdot\normal$ if $\Cf[\Omega]=+\infty$, as well as the source terms $\VEC{f}$ and $g$, are chosen coherently with \eqref{eq:analytical.u.p}.
It can be easily checked that $\VEC{u}_\darcy$ and $\VEC{u}_\stokes$ are the limit solutions in the Darcy and Stokes case corresponding, respectively, to $\Cf[\Omega]=+\infty$ ($\chi_\stokes=0$) and $\Cf[\Omega]=0$ ($\chi_\stokes=1$).

We consider a refined sequence $(\Th[h_i])_{0\le i\le 4}$ of triangular meshes in which the meshsize is halved at each refinement, that is to say, $h_{i+1}=h_i/2$ for $0\le i\le 3$; see Figure \ref{fig:meshes}.
The tests were run on a 2016 MacBook Pro equipped with an Intel Core i7 CPU clocked at 2.7GHz and 16Gb of RAM, and the implementation was based on the \texttt{SpaFEDte} platform.
The linear systems were solved using the sparse LU solver from the \texttt{Eigen} library; see http://eigen.tuxfamily.org.
We consider the values $(\mu,\nu)\in\{(0,1),(1,1),(1,0)\}$ for the coefficients.
The corresponding solutions are represented in Figures \ref{fig:regular:darcy}--\ref{fig:regular:stokes}, respectively.
The results for polynomial degrees $k$ up to 4 are collected in Tables \ref{tab:test:regular:darcy}--\ref{tab:test:regular:stokes}, which display:
the number of degrees of freedom $N_{\rm dof}$ after static condensation (see \eqref{eq:Ndof}),
the number $N_{\rm nz}$ of nonzero entries in the statically condensed matrix,
the energy-norm error $\norm[\VEC{U},h]{\uvec{e}_h}$ on the velocity,
the $L^2$-error $\norm{\VEC{e}_h}$ on the velocity,
the $L^2$-error $\norm{\epsilon_h}$ on the pressure,
as well as the assembly time $\tau_{\rm ass}$ and the resolution time $\tau_{\rm sol}$.
Denoting by $e_i$ the error in a given norm at the refinement iteration $i$, the corresponding estimated order of convergence (EOC) is obtained according to the following formula:
$$
{\rm EOC} = \frac{\log e_i - \log e_{i+1}}{\log 2}.
$$

The expected orders of convergence are observed in all the cases, and the method behaves robustly also in the limit cases corresponding to the Stokes and Darcy problem.
As for the $L^2$-norm of the velocity, it converges as $h^{k+1}$ in the Darcy case (see the fifth and sixth columns of Table \ref{tab:test:regular:darcy}) and as $h^{k+2}$ in the Brinkman and Stokes cases (see the fifth and sixth columns of Tables \ref{tab:test:regular:brinkman} and \ref{tab:test:regular:stokes}).
This behaviour is expected, as for the Darcy problem the $L^2$-norm of the velocity coincides with the energy norm, and no superconvergent behaviour can be triggered.
For the Stokes problem, on the other hand, superconvergence in the $L^2$-norm for HHO methods has been proved in, e.g., \cite[Theorem 4.5]{Aghili.Boyaval.ea:15} and \cite[Theorem 7]{Di-Pietro.Ern.ea:16}, and similar arguments can lead to analogous estimates in the Brinkman case.
For the Brinkman and Stokes problems, an inspection of the last lines of Tables \ref{tab:test:regular:brinkman} and \ref{tab:test:regular:stokes} reveals that numerical precision is approached on the finest mesh for $k=4$ and, correspondingly, the order of convergence deteriorates (see the starred values in the tables).

From the rightmost columns of Tables \ref{tab:test:regular:darcy}--\ref{tab:test:regular:stokes}, it can be noticed that the assembly time becomes negligible with respect to the resolution time as finer and finer meshes are considered.
This behaviour had already been observed in other HHO implementations (see, e.g., the numerical results in \cite{Di-Pietro.Specogna:16}). 

\begin{figure}\centering
  \includegraphics[width=3.5cm]{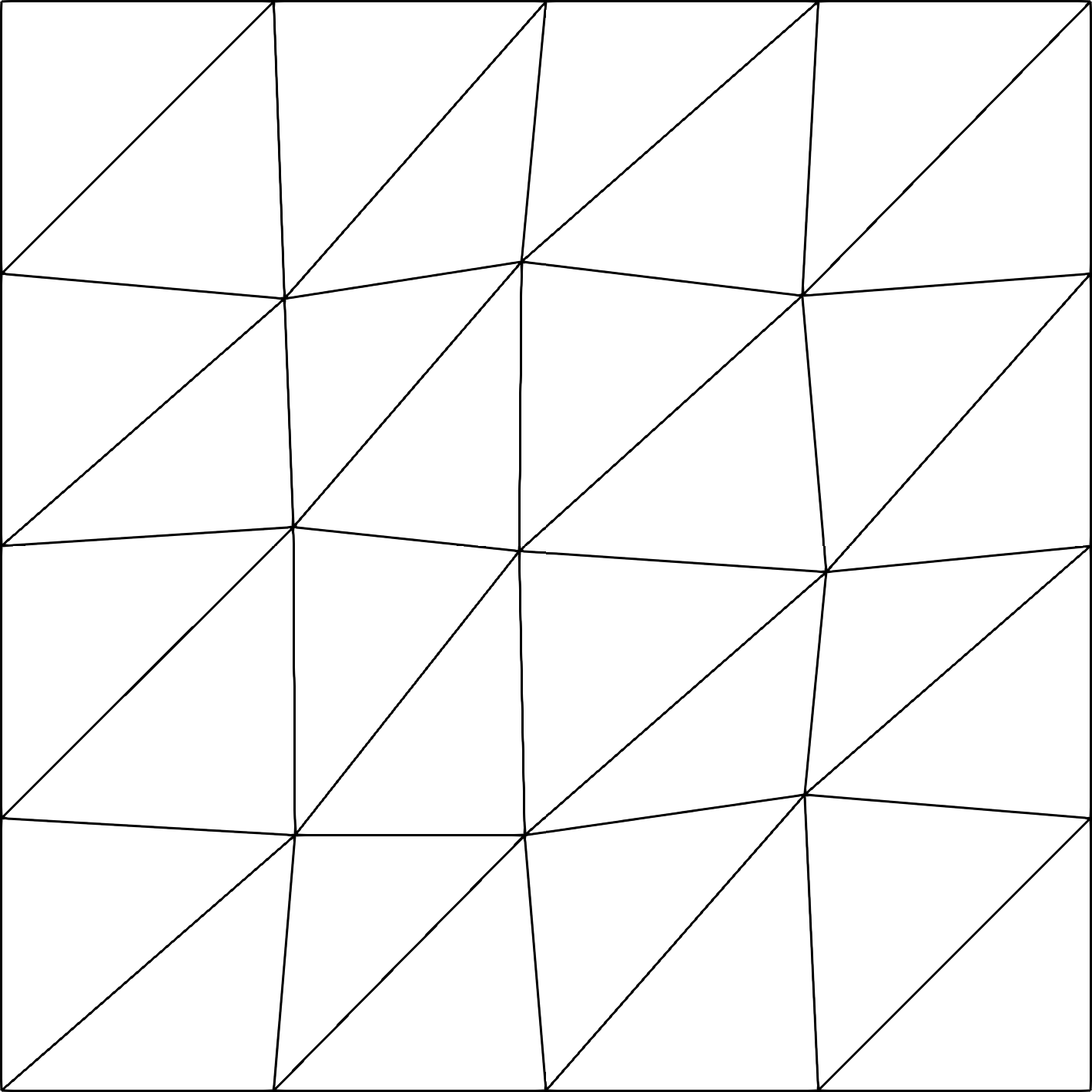}
  \hspace{0.25cm}  
  \includegraphics[width=3.5cm]{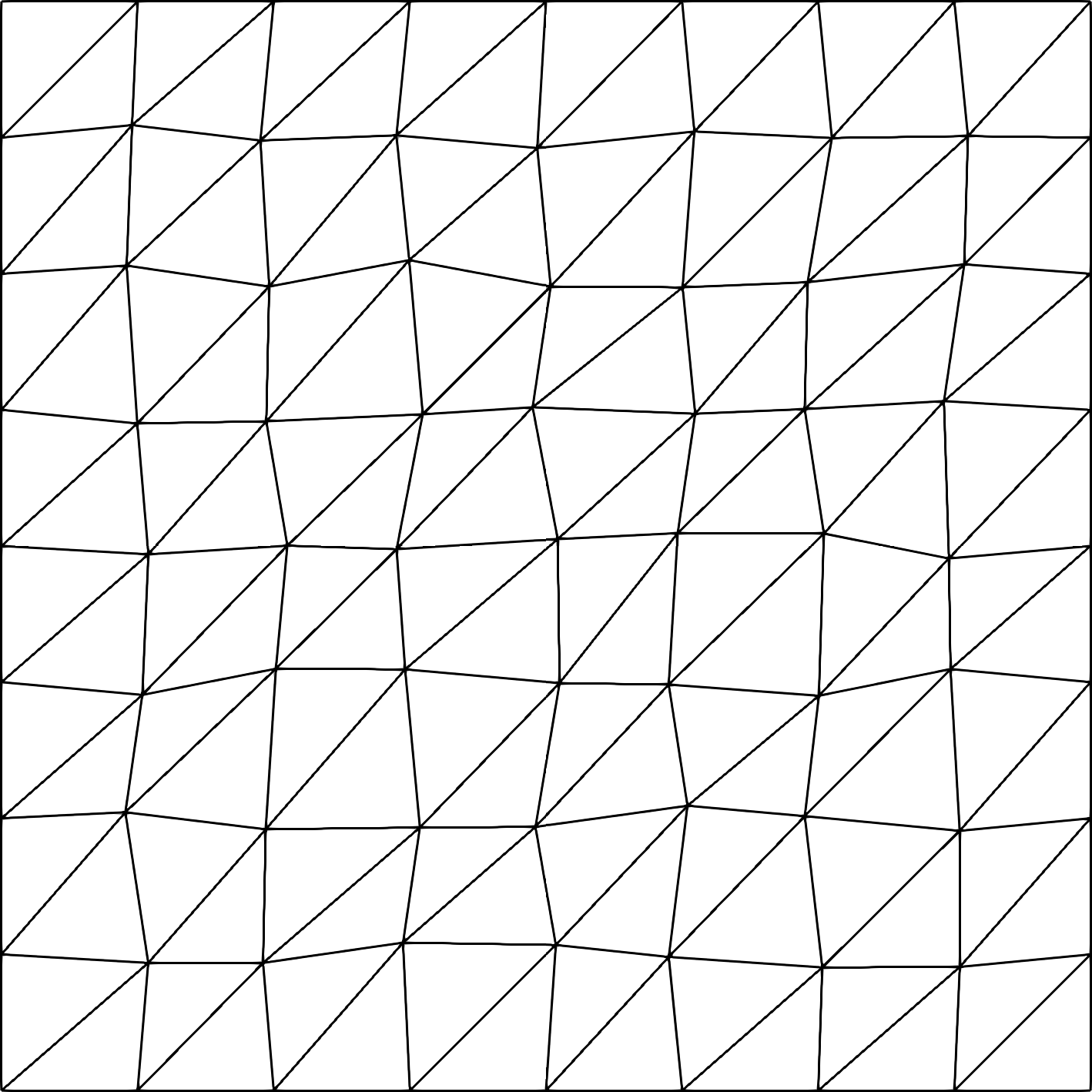}
  \hspace{0.25cm}
  \includegraphics[width=3.5cm]{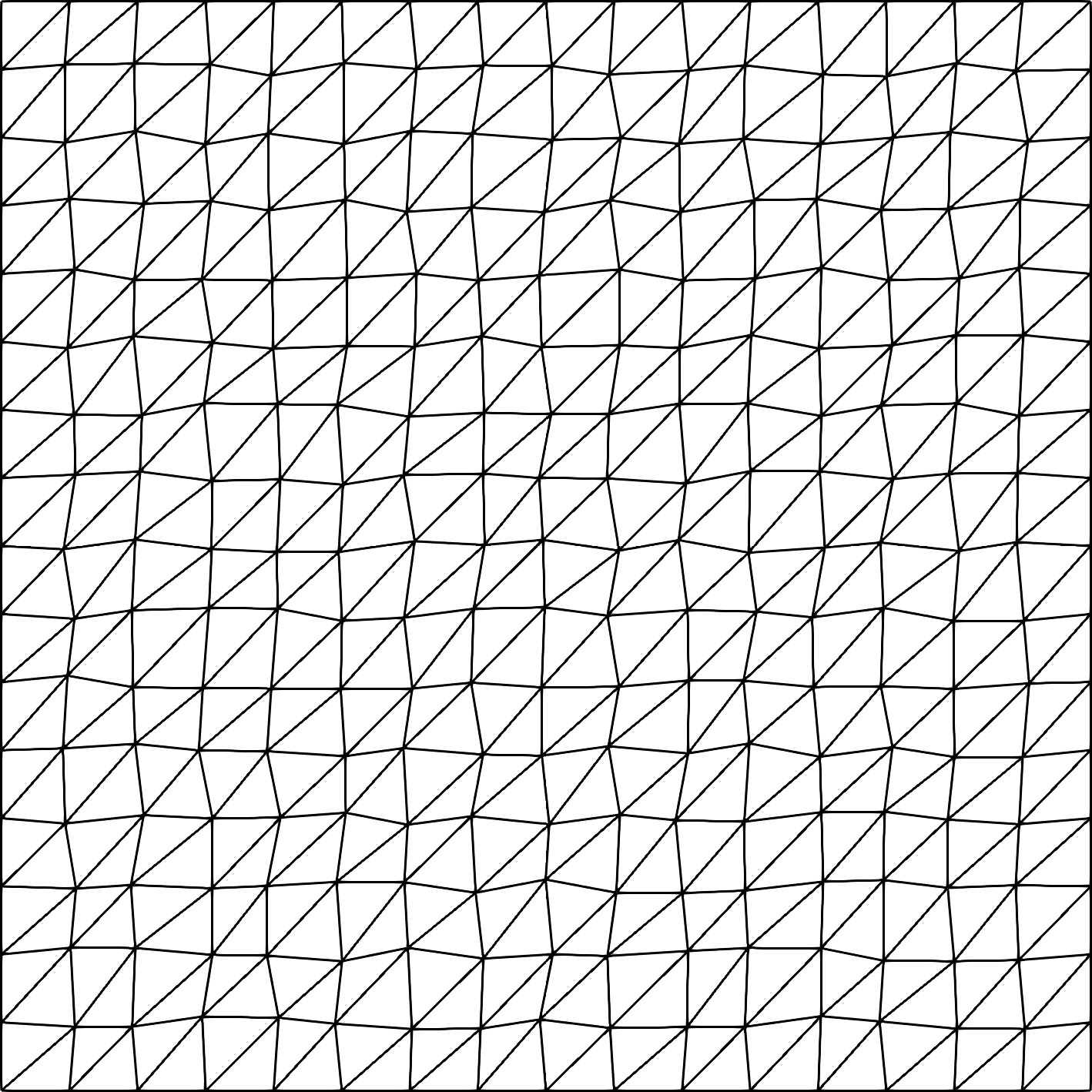}
  \caption{First three meshes of the sequence used for the numerical test of Section \ref{sec:numerical.examples}.\label{fig:meshes}}
\end{figure}

\begin{figure}\centering
  \includegraphics[height=5.50cm]{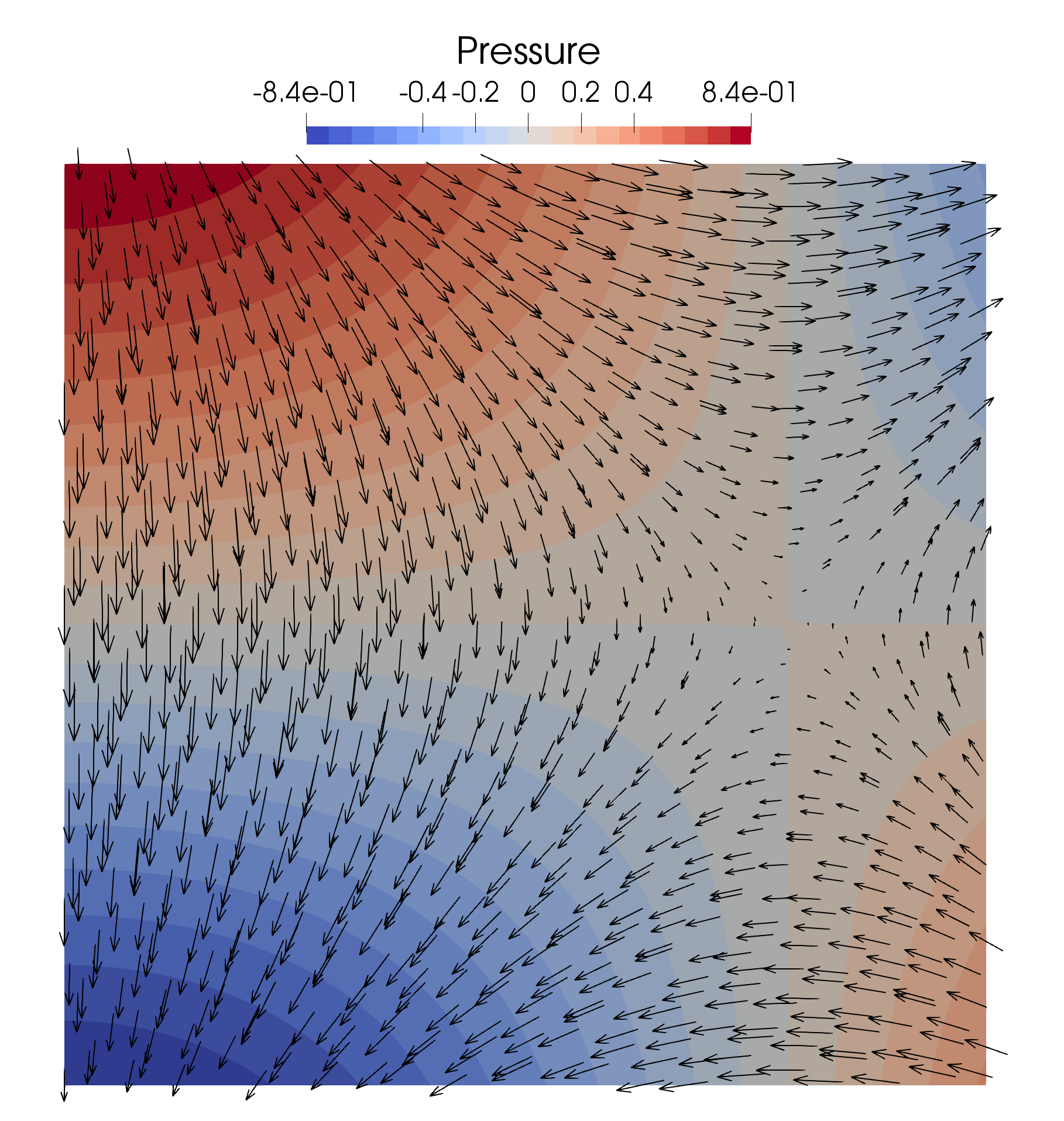}
  \caption{Solution \eqref{eq:analytical.u.p} with $(\mu,\nu)=(0,1)$ (Darcy problem). Arrows indicate orientation and magnitude of the velocity at a given point.\label{fig:regular:darcy}}
  \captionof{table}{Convergence results for the Darcy problem; see Remark \ref{rem:darcy}.\label{tab:test:regular:darcy}}
  \begin{footnotesize}
    \begin{tabular}{ccccccccccc}
      \toprule
      $N_{\rm dof}$  & $N_{\rm nz}$ & $\norm[\VEC{U},h]{\uvec{e}_h}$ & EOC & $\norm{\VEC{e}_h}$ & EOC & $\norm{\epsilon_h}$ & EOC & $\tau_{\rm ass}$ & $\tau_{\rm sol}$ \\
      \midrule
      \multicolumn{10}{c}{$k=0$} \\
      \midrule \\
      113        & 1072       & 1.69e-01   & --         & 1.69e-01   & --         & 1.39e-01   & --         & 2.26e-03   & 9.68e-04   \\ 
      481        & 4944       & 8.84e-02   & 0.94       & 8.84e-02   & 0.94       & 4.27e-02   & 1.70       & 1.19e-02   & 5.34e-03   \\ 
      1985       & 21136      & 4.47e-02   & 0.98       & 4.47e-02   & 0.98       & 1.18e-02   & 1.86       & 3.34e-02   & 5.83e-02   \\ 
      8065       & 87312      & 2.22e-02   & 1.01       & 2.22e-02   & 1.01       & 3.69e-03   & 1.67       & 1.12e-01   & 1.02e+00   \\ 
      32513      & 354832     & 1.09e-02   & 1.03       & 1.09e-02   & 1.03       & 1.45e-03   & 1.35       & 3.94e-01   & 3.39e+01   \\ 
      \midrule
      \multicolumn{10}{c}{$k=1$} \\
      \midrule \\
      193        & 3456       & 1.33e-02   & --         & 3.89e-03   & --         & 5.15e-03   & --         & 4.24e-03   & 1.71e-03   \\ 
      833        & 16192      & 2.65e-03   & 2.32       & 7.73e-04   & 2.33       & 1.01e-03   & 2.36       & 1.98e-02   & 1.91e-02   \\ 
      3457       & 69696      & 6.55e-04   & 2.02       & 1.90e-04   & 2.03       & 2.27e-04   & 2.15       & 6.16e-02   & 1.35e-01   \\ 
      14081      & 288832     & 1.66e-04   & 1.98       & 4.80e-05   & 1.98       & 5.53e-05   & 2.03       & 2.05e-01   & 1.94e+00   \\ 
      56833      & 1175616    & 4.32e-05   & 1.94       & 1.25e-05   & 1.94       & 1.37e-05   & 2.01       & 7.70e-01   & 6.49e+01   \\ 
      \midrule
      \multicolumn{10}{c}{$k=2$} \\
      \midrule \\
      273        & 7216       & 4.84e-03   & --         & 1.25e-03   & --         & 2.48e-04   & --         & 7.61e-03   & 2.57e-03   \\ 
      1185       & 34000      & 7.55e-04   & 2.68       & 1.94e-04   & 2.68       & 2.94e-05   & 3.08       & 3.64e-02   & 4.46e-02   \\ 
      4929       & 146704     & 1.00e-04   & 2.91       & 2.59e-05   & 2.90       & 3.76e-06   & 2.97       & 1.23e-01   & 2.39e-01   \\ 
      20097      & 608656     & 1.29e-05   & 2.95       & 3.36e-06   & 2.95       & 4.77e-07   & 2.98       & 4.02e-01   & 3.84e+00   \\ 
      81153      & 2478736    & 1.64e-06   & 2.98       & 4.25e-07   & 2.98       & 5.94e-08   & 3.00       & 1.55e+00   & 8.75e+01   \\ 
      \midrule
      \multicolumn{10}{c}{$k=3$} \\
      \midrule \\
      353        & 12352      & 1.33e-04   & --         & 2.53e-05   & --         & 1.29e-05   & --         & 1.98e-02   & 3.63e-03   \\ 
      1537       & 58368      & 8.03e-06   & 4.05       & 1.48e-06   & 4.09       & 8.93e-07   & 3.85       & 6.52e-02   & 4.05e-02   \\ 
      6401       & 252160     & 5.17e-07   & 3.96       & 9.48e-08   & 3.97       & 5.64e-08   & 3.99       & 2.17e-01   & 5.34e-01   \\ 
      26113      & 1046784    & 3.28e-08   & 3.98       & 5.98e-09   & 3.99       & 3.56e-09   & 3.98       & 8.46e-01   & 7.84e+00   \\ 
      105473     & 4264192    & 2.08e-09   & 3.98       & 3.80e-10   & 3.98       & 2.22e-10   & 4.01       & 3.39e+00   & 1.27e+02   \\ 
      \midrule
      \multicolumn{10}{c}{$k=4$} \\
      \midrule \\
      433        & 18864      & 1.34e-05   & --         & 2.31e-06   & --         & 5.47e-07   & --         & 3.36e-02   & 5.35e-03   \\ 
      1889       & 89296      & 5.24e-07   & 4.68       & 8.94e-08   & 4.69       & 1.79e-08   & 4.93       & 1.30e-01   & 4.59e-02   \\ 
      7873       & 386064     & 1.70e-08   & 4.95       & 2.91e-09   & 4.94       & 5.91e-10   & 4.92       & 4.04e-01   & 7.72e-01   \\ 
      32129      & 1603216    & 5.51e-10   & 4.95       & 9.45e-11   & 4.95       & 1.88e-11   & 4.97       & 1.48e+00   & 1.02e+01   \\ 
      129793     & 6531984    & 1.92e-11   & 4.84       & 3.26e-12   & 4.86       & 5.84e-13   & 5.01       & 6.11e+00   & 1.70e+02   \\ 
      \bottomrule
    \end{tabular}
  \end{footnotesize}
\end{figure}

\begin{figure}\centering  
  \includegraphics[height=5.50cm]{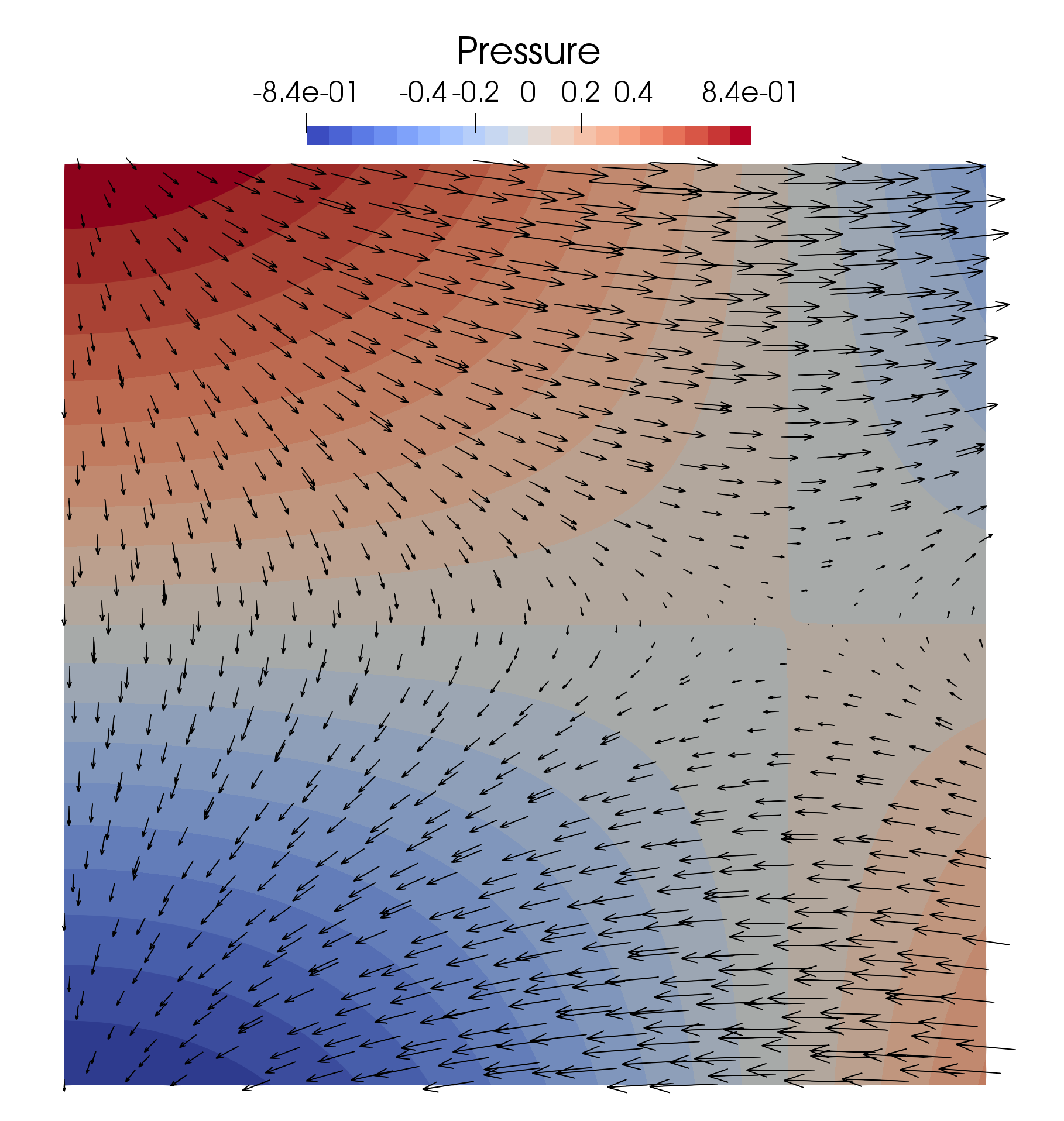}
  \caption{Solution \eqref{eq:analytical.u.p} with $(\mu,\nu)=(1,1)$ (Brinkman problem). Arrows indicate orientation and magnitude of the velocity at a given point.\label{fig:regular:brinkman}}
  \captionof{table}{Convergence results for the Brinkman problem.\label{tab:test:regular:brinkman}}
  \begin{footnotesize}
    \begin{tabular}{cccccccccc}
      \toprule
      $N_{\rm dof}$  & $N_{\rm nz}$ & $\norm[\VEC{U},h]{\uvec{e}_h}$ & EOC & $\norm{\VEC{e}_h}$ & EOC & $\norm{\epsilon_h}$ & EOC & $\tau_{\rm ass}$ & $\tau_{\rm sol}$ \\
      \midrule
      \multicolumn{10}{c}{$k=1$} \\
      \midrule \\
      193        & 3456       & 6.48e-02   & --         & 3.51e-03   & --         & 3.40e-02   & --         & 4.86e-03   & 1.87e-03   \\ 
      833        & 16192      & 2.78e-02   & 1.22       & 7.40e-04   & 2.24       & 9.34e-03   & 1.86       & 1.65e-02   & 2.05e-02   \\ 
      3457       & 69696      & 8.93e-03   & 1.64       & 1.18e-04   & 2.65       & 2.60e-03   & 1.84       & 6.32e-02   & 1.19e-01   \\ 
      14081      & 288832     & 2.43e-03   & 1.88       & 1.62e-05   & 2.87       & 6.84e-04   & 1.93       & 2.20e-01   & 1.69e+00   \\ 
      56833      & 1175616    & 6.30e-04   & 1.95       & 2.10e-06   & 2.95       & 1.75e-04   & 1.97       & 8.13e-01   & 4.38e+01   \\
      \midrule
      \multicolumn{10}{c}{$k=2$} \\
      \midrule \\
      273        & 7216       & 3.72e-03   & --         & 1.21e-04   & --         & 1.74e-03   & --         & 8.64e-03   & 2.76e-03   \\ 
      1185       & 34000      & 7.56e-04   & 2.30       & 1.24e-05   & 3.28       & 1.98e-04   & 3.13       & 3.56e-02   & 3.12e-02   \\ 
      4929       & 146704     & 1.13e-04   & 2.74       & 9.35e-07   & 3.73       & 2.29e-05   & 3.12       & 1.28e-01   & 1.87e-01   \\ 
      20097      & 608656     & 1.52e-05   & 2.89       & 6.30e-08   & 3.89       & 2.70e-06   & 3.08       & 4.23e-01   & 2.97e+00   \\ 
      81153      & 2478736    & 1.96e-06   & 2.95       & 4.08e-09   & 3.95       & 3.27e-07   & 3.04       & 1.71e+00   & 5.92e+01   \\
      \midrule
      \multicolumn{10}{c}{$k=3$} \\
      \midrule \\
      353        & 12352      & 2.44e-04   & --         & 6.48e-06   & --         & 1.41e-04   & --         & 1.74e-02   & 3.93e-03   \\ 
      1537       & 58368      & 1.99e-05   & 3.62       & 2.68e-07   & 4.60       & 9.32e-06   & 3.92       & 7.41e-02   & 4.50e-02   \\ 
      6401       & 252160     & 1.27e-06   & 3.97       & 8.50e-09   & 4.98       & 5.65e-07   & 4.04       & 2.53e-01   & 4.28e-01   \\ 
      26113      & 1046784    & 8.26e-08   & 3.94       & 2.79e-10   & 4.93       & 3.58e-08   & 3.98       & 9.11e-01   & 5.58e+00   \\ 
      105473     & 4264192    & 5.19e-09   & 3.99       & 8.78e-12   & 4.99       & 2.23e-09   & 4.00       & 3.67e+00   & 8.72e+01   \\ 
      \midrule
      \multicolumn{10}{c}{$k=4$} \\
      \midrule \\
      433        & 18864      & 1.10e-05   & --         & 2.61e-07   & --         & 6.84e-06   & --         & 3.13e-02   & 1.15e-02   \\ 
      1889       & 89296      & 3.98e-07   & 4.78       & 4.75e-09   & 5.78       & 2.14e-07   & 5.00       & 1.39e-01   & 4.46e-02   \\ 
      7873       & 386064     & 1.33e-08   & 4.90       & 7.93e-11   & 5.90       & 6.83e-09   & 4.97       & 4.80e-01   & 8.31e-01   \\ 
      32129      & 1603216    & 4.26e-10   & 4.96       & 1.26e-12   & 5.97       & 2.13e-10   & 5.00       & 1.79e+00   & 8.10e+00   \\ 
      129793     & 6531984    & 1.08e-10   & $1.98^*$   & 1.80e-13   & $2.81^*$   & 2.87e-11   & $2.89^*$   & 7.09e+00   & 1.33e+02   \\ 
      \bottomrule
    \end{tabular}
  \end{footnotesize}
\end{figure}

\begin{figure}\centering
  \includegraphics[height=5.50cm]{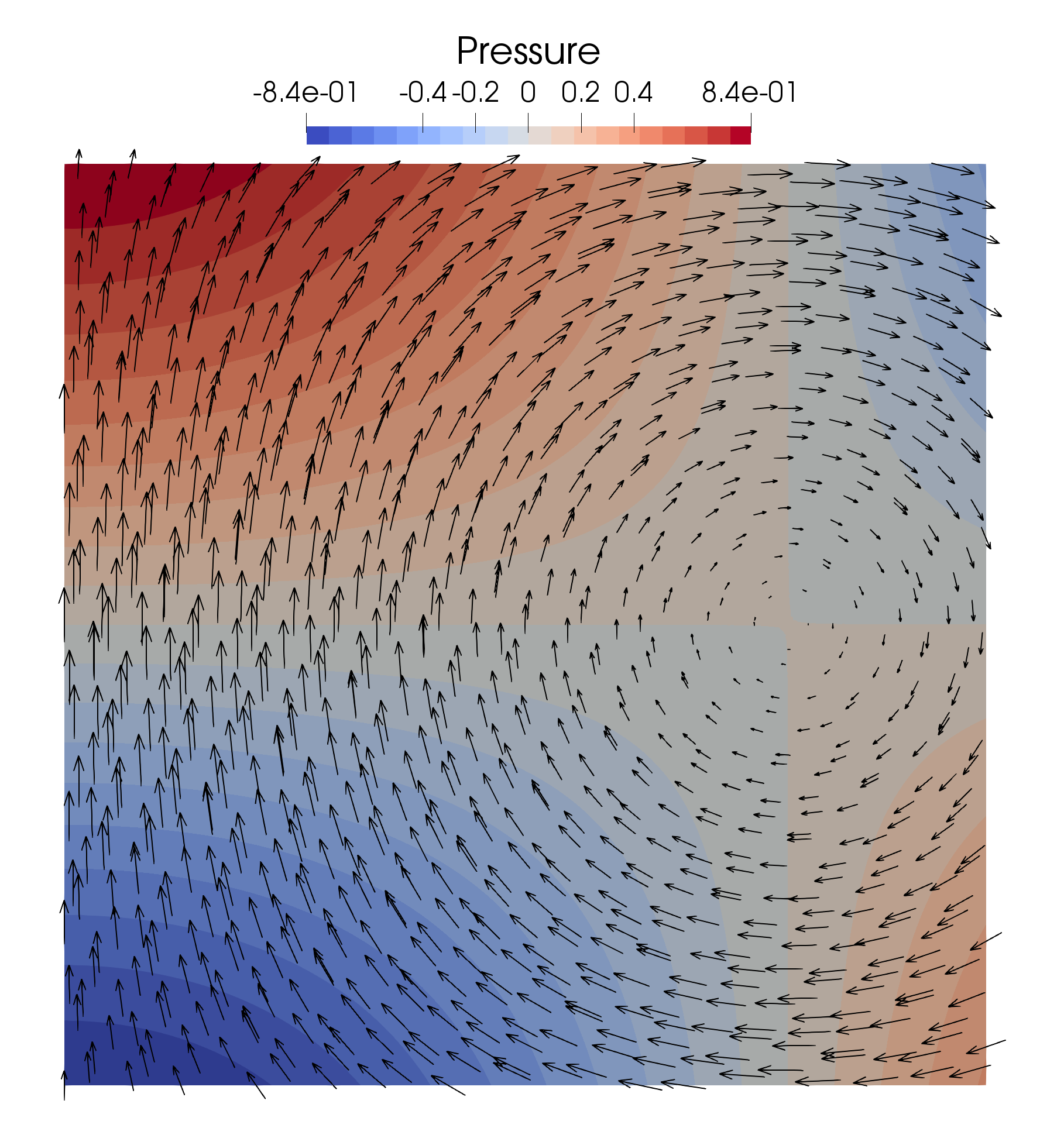}
  \caption{Solution \eqref{eq:analytical.u.p} with $(\mu,\nu)=(1,0)$ (Stokes problem). Arrows indicate orientation and magnitude of the velocity at a given point.\label{fig:regular:stokes}}
  \captionof{table}{Convergence results for the Stokes problem.\label{tab:test:regular:stokes}}
  \begin{footnotesize}
    \begin{tabular}{cccccccccc}
      \toprule
      $N_{\rm dof}$  & $N_{\rm nz}$ & $\norm[\VEC{U},h]{\uvec{e}_h}$ & EOC & $\norm{\VEC{e}_h}$ & EOC & $\norm{\epsilon_h}$ & EOC & $\tau_{\rm ass}$ & $\tau_{\rm sol}$ \\
      \midrule
      \multicolumn{10}{c}{$k=1$} \\
      \midrule \\
      193        & 3456       & 1.10e-02   & --         & 6.07e-04   & --         & 1.82e-02   & --         & 6.74e-03   & 2.36e-03   \\ 
      833        & 16192      & 3.79e-03   & 1.54       & 1.09e-04   & 2.48       & 5.06e-03   & 1.85       & 1.61e-02   & 2.31e-02   \\ 
      3457       & 69696      & 1.04e-03   & 1.86       & 1.52e-05   & 2.84       & 1.32e-03   & 1.94       & 7.64e-02   & 1.33e-01   \\ 
      14081      & 288832     & 2.71e-04   & 1.94       & 1.99e-06   & 2.93       & 3.37e-04   & 1.96       & 2.32e-01   & 1.68e+00   \\ 
      56833      & 1175616    & 6.98e-05   & 1.96       & 2.56e-07   & 2.96       & 8.53e-05   & 1.98       & 8.35e-01   & 4.41e+01   \\ 
      \midrule
      \multicolumn{10}{c}{$k=2$} \\
      \midrule \\
      273        & 7216       & 1.38e-03   & --         & 4.97e-05   & --         & 1.70e-03   & --         & 9.99e-03   & 2.82e-03   \\ 
      1185       & 34000      & 1.95e-04   & 2.83       & 3.47e-06   & 3.84       & 2.39e-04   & 2.83       & 4.15e-02   & 3.44e-02   \\ 
      4929       & 146704     & 2.74e-05   & 2.83       & 2.39e-07   & 3.86       & 3.06e-05   & 2.96       & 2.38e-01   & 2.09e-01   \\ 
      20097      & 608656     & 3.58e-06   & 2.94       & 1.55e-08   & 3.94       & 3.90e-06   & 2.97       & 4.52e-01   & 3.11e+00   \\ 
      81153      & 2478736    & 4.50e-07   & 2.99       & 9.77e-10   & 3.99       & 4.90e-07   & 2.99       & 1.74e+00   & 6.17e+01   \\
      \midrule
      \multicolumn{10}{c}{$k=3$} \\
      \midrule \\
      353        & 12352      & 1.17e-04   & --         & 3.38e-06   & --         & 1.51e-04   & --         & 1.78e-02   & 4.03e-03   \\ 
      1537       & 58368      & 8.48e-06   & 3.79       & 1.26e-07   & 4.74       & 1.07e-05   & 3.83       & 7.66e-02   & 4.63e-02   \\ 
      6401       & 252160     & 5.43e-07   & 3.96       & 4.01e-09   & 4.98       & 6.70e-07   & 3.99       & 2.58e-01   & 4.51e-01   \\ 
      26113      & 1046784    & 3.45e-08   & 3.98       & 1.28e-10   & 4.97       & 4.24e-08   & 3.98       & 9.33e-01   & 5.87e+00   \\ 
      105473     & 4264192    & 2.18e-09   & 3.99       & 4.04e-12   & 4.99       & 2.66e-09   & 3.99       & 3.63e+00   & 9.27e+01   \\
      \midrule
      \multicolumn{10}{c}{$k=4$} \\
      \midrule \\
      433        & 18864      & 7.92e-06   & --         & 2.20e-07   & --         & 7.41e-06   & --         & 3.28e-02   & 1.26e-02   \\ 
      1889       & 89296      & 2.45e-07   & 5.02       & 3.39e-09   & 6.02       & 2.39e-07   & 4.96       & 1.38e-01   & 4.53e-02   \\ 
      7873       & 386064     & 7.87e-09   & 4.96       & 5.43e-11   & 5.96       & 7.68e-09   & 4.96       & 4.54e-01   & 8.91e-01   \\ 
      32129      & 1603216    & 2.46e-10   & 5.00       & 8.43e-13   & 6.01       & 2.41e-10   & 4.99       & 1.92e+00   & 9.18e+00   \\ 
      129793     & 6531984    & 1.18e-10   & $1.07^*$   & 1.93e-13   & $2.13^*$   & 3.43e-11   & $2.81^*$   & 7.61e+00   & 1.37e+02   \\ 
      \bottomrule
    \end{tabular}
  \end{footnotesize}
\end{figure}

\subsection{Convergence for the Darcy problem with spatially varying permeability}
\label{sec:numerical.examples:darcy}

It is clear from the analysis that the key feature to handle all the possible regimes is the robustness in the singular limit corresponding to the Darcy problem.
In this section we further investigate the numerical performance of the proposed method in this case considering:
\begin{inparaenum}[(a)] 
\item a smooth spatially varying coefficient $\nu$ over $\Omega$, 
\item a piecewise constant discontinuous coefficient $\nu$ over $\Omega$. 
\end{inparaenum}

As for case (a), setting $\Omega=(0,3\pi)\times(0,2\pi)$, we consider the exact solution originally proposed in \cite{Philips:86} which, for a fixed value of the parameter $\alpha\in(0,1]$, corresponds to
\begin{equation}\label{eq:philip}
  \VEC{u}(\VEC{x})=
  \begin{pmatrix}
  -1 - \alpha\sin(x_1)\cos(x_2) \\
  \alpha\cos x_1\sin x_2
  \end{pmatrix}
\end{equation}
with 
$$
\mu=0,\qquad
\nu(\VEC{x})=( 1 + 2\alpha\sin x_1\cos x_2 + \alpha^2\cos^2 x_2)^{-1}.
$$
In what follows, we take $(1-\alpha)^2=10^{-3}$ which implies a variation of $\nu$ spanning three orders of magnitude: indeed the permeability ranges from  $(1-\alpha)^{-2}$ to $(1+\alpha)^{-2}$ over $\Omega$.
An analytical expression for the pressure is not available in this case.
The numerical solution computed with a $k=4$ degree discretization on the finest grid is represented in Figure \ref{fig:numerical.examples:darcy:heterogeneous:solution}.

As for case (b), setting $\Omega=(-1,1) \times (-1,1)$, we rely on the exact solution originally proposed in \cite{Kellogg:74} where, for a given value of the permeability jump $\frac{\nu_{1}}{\nu_{2}}$ 
over the horizontal and vertical center-lines, the authors provide a means to compute the parameters $ \gamma, \rho, \sigma$ such that 
the exact pressure reads 
\begin{equation}
\label{eq:kellog} 
p(r,\theta)= r^\gamma s(\theta), 
\end{equation}
with $\tan(\theta) = \frac{y}{x}, r = \sqrt{x^2 + y^2}$ and
\begin{equation*} 
  s(\theta)= \begin{cases}
  \cos((\pi/2 - \sigma) \gamma) \; \cos((\theta-\pi/2+\rho)\gamma)& \rm{if} \; 0 \leq \theta \leq \pi/2, \\
  \cos( \rho \gamma) \; \cos((\theta-\pi+\sigma)\gamma)&             \rm{if} \; \pi/2 \leq \theta \leq pi, \\
  \cos(\sigma \gamma) \; \cos((\theta-\pi-\rho)\gamma)&              \rm{if} \; \pi \leq \theta \leq 3\pi/2, \\
  \cos((\pi/2 - \rho) \gamma) \; \cos((\theta-3\pi/2-\sigma)\gamma)&  \rm{if} \; 3\pi/2 \leq \theta \leq 2\pi. \\
  \end{cases}
\end{equation*} 
In what follows, we take $\frac{\nu_{1}}{\nu_{2}}=100$, leading to $\gamma = 0.1269020697222, \, \rho=\pi/4$ and $\sigma=-11.5926215980874$.
Neumann boundary conditions are enforced according to the exact solution. 
Computational meshes are compatible with the distribution of $\nu$ over $\Omega$, 
that is $\nu_1 = 100$ in the first and third quadrant and $\nu_2 = 1$ in the second and fourth quadrant, in order to ensure that permeability jumps do not occur inside mesh elements.
The solution is singular at the origin, and its regularity depends on the parameter $\gamma$, namely $p \in H^{1+\gamma}(\Omega)$. 
Accordingly, the expected convergence rate for the velocity and the pressure in $L^2$-norm is $\gamma$ and $2 \gamma$, respectively.
The numerical solution computed with $k=4$ on the finest grid is represented in Figure \ref{fig:numerical.examples:darcy:discontinuous:solution}.

The convergence results on a sequence of refined triangular meshes for polynomial degrees $k\in\{0,\ldots,4\}$ are collected in Table \ref{tab:test:philip} and Table \ref{tab:test:kellog} for case (a) and (b), respectively.
The columns have the same meaning as in the previous section, except for the fact that, for the case (a), the pressure errors and the corresponding estimated orders of convergence are not displayed.
For case (a), it can be seen that the predicted asymptotic orders of convergence are matched or exceeded.
Notice that we had to increase the degree of exactness of the quadrature rule in this case to account for the fact that the coefficient $\nu$ varies inside the elements.
This is crucial for obtaining the expected convergence rates for $k > 1$.
For case (b), pressure and velocity convergence rates are slightly suboptimal at the highest polynomial degree, a consequence of the well known Runge phenomenon, but 
in line with results obtained in \cite{Stephansen:07} by means of a SWIP dG discretization. Remarkably, both pressure and velocity errors decrease when increasing the polynomial degree.

\begin{figure}\centering
  \includegraphics[height=5.50cm]{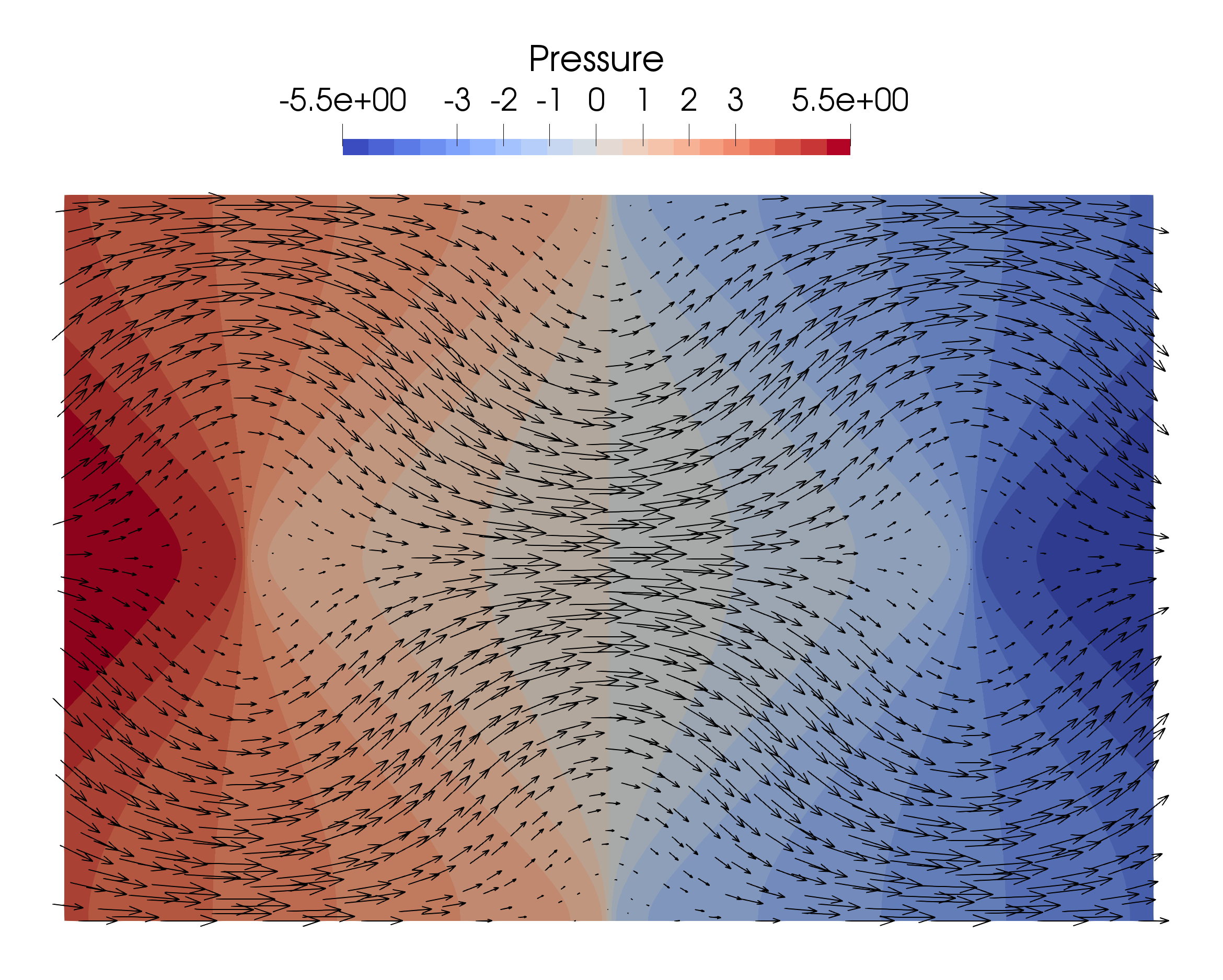}
  \caption{Solution \eqref{eq:philip} (Darcy problem with variable permeability). Arrows indicate orientation and magnitude of the velocity at a given point.\label{fig:numerical.examples:darcy:heterogeneous:solution}}
  \captionof{table}{Convergence results for the Darcy problem with variable permeability. \label{tab:test:philip}}
  \begin{footnotesize}
    \begin{tabular}{cccccccc}
      \toprule
      $N_{\rm dof}$  & $N_{\rm nz}$ & $\norm[\VEC{U},h]{\uvec{e}_h}$ & EOC & $\norm{\VEC{e}_h}$ & EOC & $\tau_{\rm ass}$ & $\tau_{\rm sol}$ \\
      \midrule
      \multicolumn{8}{c}{$k=0$} \\
      \midrule \\
      173        & 1688       & 1.54e+01   & --         & 3.68e+00   & --       & 4.21e-03   & 1.61e-03   \\ 
      729        & 7584       & 7.59e+00   & 1.03       & 2.10e+00   & 0.81     & 1.99e-02   & 1.10e-02   \\ 
      2993       & 32048      & 2.30e+00   & 1.72       & 1.08e+00   & 0.95     & 8.56e-02   & 1.64e-01   \\ 
      12129      & 131664     & 8.04e-01   & 1.51       & 5.02e-01   & 1.11     & 2.00e-01   & 2.90e+00   \\ 
      48833      & 533648     & 3.13e-01   & 1.36       & 2.23e-01   & 1.17     & 7.38e-01   & 1.28e+02   \\
      \midrule
      \multicolumn{8}{c}{$k=1$} \\
      \midrule \\
      297        & 5472       & 4.79e+01   & --         & 2.42e+00   & --       & 8.49e-03   & 3.23e-03   \\ 
      1265       & 24896      & 1.06e+01   & 2.18       & 6.64e-01   & 1.87     & 4.68e-02   & 3.01e-02   \\ 
      5217       & 105792     & 2.26e+00   & 2.23       & 1.27e-01   & 2.39     & 1.36e-01   & 2.57e-01   \\ 
      21185      & 435776     & 3.79e-01   & 2.58       & 1.59e-02   & 2.99     & 4.29e-01   & 5.28e+00   \\ 
      85377      & 1768512    & 4.58e-02   & 3.05       & 2.15e-03   & 2.89     & 1.73e+00   & 1.85e+02   \\ 
      \midrule
      \multicolumn{8}{c}{$k=2$} \\
      \midrule \\
      421        & 11448      & 1.66e+01   & --         & 5.69e-01   & --       & 2.39e-02   & 5.40e-03   \\ 
      1801       & 52320      & 8.73e-01   & 4.25       & 5.60e-02   & 3.34     & 9.45e-02   & 6.97e-02   \\ 
      7441       & 222768     & 6.72e-02   & 3.70       & 5.29e-03   & 3.40     & 2.82e-01   & 5.16e-01   \\ 
      30241      & 918480     & 4.68e-03   & 3.84       & 5.85e-04   & 3.18     & 1.04e+00   & 8.94e+00   \\ 
      121921     & 3729168    & 4.15e-04   & 3.50       & 6.88e-05   & 3.09     & 4.13e+00   & 2.65e+02   \\
      \midrule
      \multicolumn{8}{c}{$k=3$} \\
      \midrule \\
      545        & 19616      & 3.23e+00   & --         & 9.54e-02   & --       & 5.24e-02   & 8.26e-03   \\ 
      2337       & 89856      & 1.93e-01   & 4.06       & 5.31e-03   & 4.17     & 1.59e-01   & 7.70e-02   \\ 
      9665       & 382976     & 9.29e-03   & 4.38       & 2.49e-04   & 4.42     & 5.54e-01   & 1.00e+00   \\ 
      39297      & 1579776    & 4.17e-04   & 4.48       & 1.21e-05   & 4.36     & 2.14e+00   & 1.56e+01   \\ 
      158465     & 6415616    & 1.67e-05   & 4.65       & 6.55e-07   & 4.21     & 8.35e+00   & 3.59e+02   \\ 
      \midrule
      \multicolumn{8}{c}{$k=4$} \\
      \midrule \\
      669        & 29976      & 9.15e-01   & --         & 1.49e-02   & --       & 1.16e-01   & 9.69e-03   \\ 
      2873       & 137504     & 3.33e-02   & 4.78       & 4.92e-04   & 4.92     & 3.08e-01   & 1.09e-01   \\ 
      11889      & 586416     & 2.13e-04   & 7.29       & 1.05e-05   & 5.55     & 1.19e+00   & 1.65e+00   \\ 
      48353      & 2419664    & 1.28e-05   & 4.06       & 3.34e-07   & 4.97     & 4.91e+00   & 2.68e+01   \\ 
      195009     & 9827856    & 1.20e-07   & 6.74       & 1.00e-08   & 5.06     & 1.87e+01   & 4.89e+02   \\ 
      \bottomrule
    \end{tabular}
  \end{footnotesize}
\end{figure}

\begin{figure}\centering
  \includegraphics[height=5.50cm]{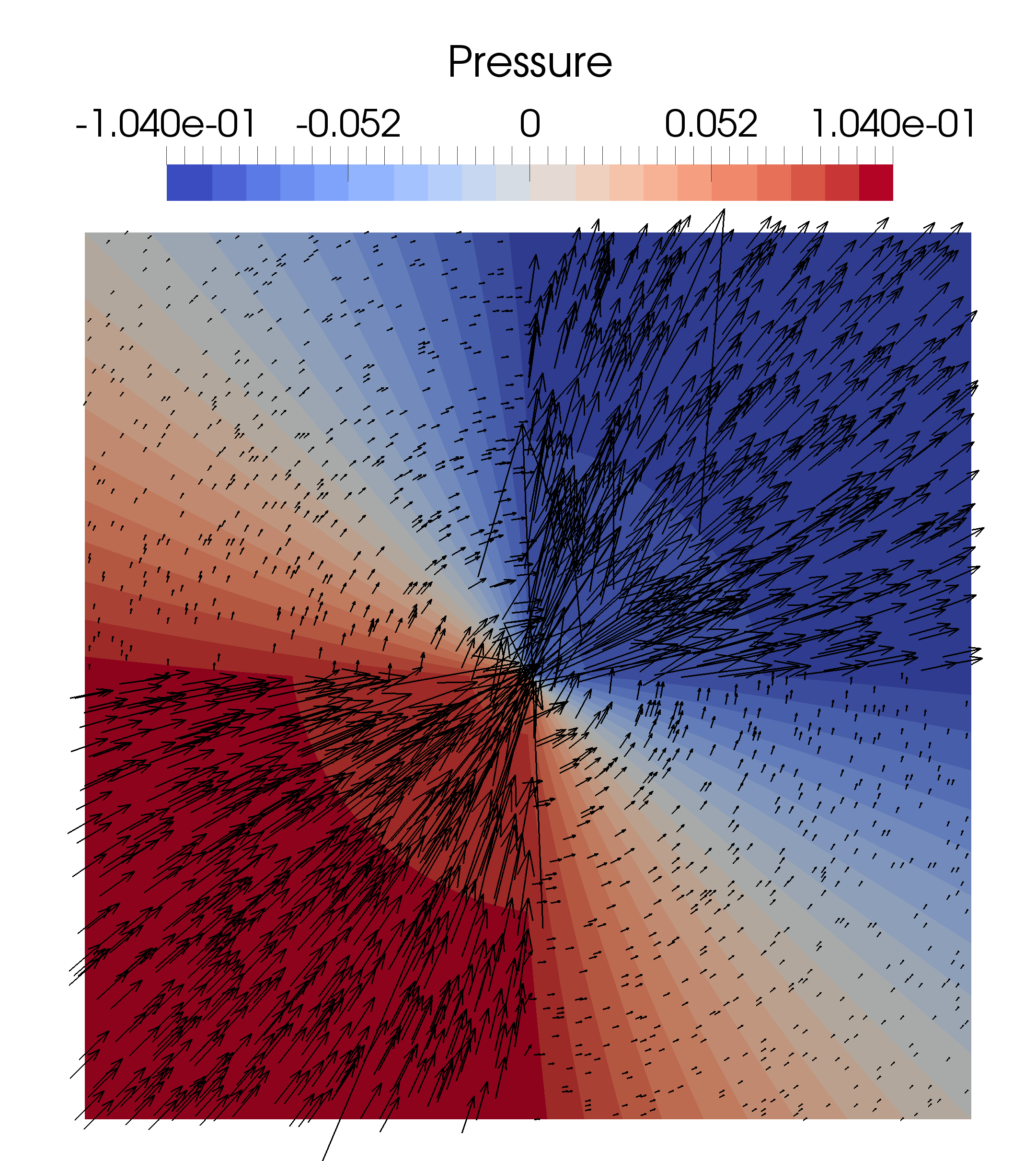}
  \caption{Solution \eqref{eq:kellog} (Darcy problem with discontinuous permeability). 
  Arrows indicate orientation and magnitude of the velocity at a given point.\label{fig:numerical.examples:darcy:discontinuous:solution}}
  \captionof{table}{Convergence results for the Darcy problem with variable permeability. \label{tab:test:kellog}}
  \begin{footnotesize}
    \begin{tabular}{cccccccc}
      \toprule
      $N_{\rm dof}$  & $N_{\rm nz}$ & $\norm[\VEC{U},h]{\uvec{e}_h}$ & EOC & $\norm{\VEC{e}_h}$ & EOC & $\norm{\epsilon_h}$ & EOC  \\
      \midrule
      \multicolumn{8}{c}{$k=0$} \\
      \midrule \\
      113        & 1440       & 1.12e+00   & --         & 2.62e-01   & --         & 4.30e-02   & --   \\ 
      481        & 5696       & 1.03e+00   & 0.13       & 2.44e-01   & 0.11       & 2.58e-02   & 0.74 \\ 
      1985       & 22656      & 9.36e-01   & 0.13       & 2.22e-01   & 0.13       & 1.49e-02   & 0.79 \\ 
      8065       & 90368      & 8.55e-01   & 0.13       & 2.04e-01   & 0.12       & 9.51e-03   & 0.65 \\ 
      32513      & 360960     & 7.81e-01   & 0.13       & 1.89e-01   & 0.11       & 7.29e-03   & 0.38 \\ 
      \midrule
      \multicolumn{8}{c}{$k=1$} \\
      \midrule \\
      193        & 4800       & 1.37e+00   & --         & 2.52e-01   & --         & 1.24e-02   & --   \\ 
      833        & 18944      & 1.24e+00   & 0.13       & 2.35e-01   & 0.10       & 8.97e-03   & 0.46 \\ 
      3457       & 75264      & 1.14e+00   & 0.13       & 2.19e-01   & 0.10       & 6.84e-03   & 0.39 \\ 
      14081      & 300032     & 1.05e+00   & 0.12       & 2.05e-01   & 0.10       & 5.69e-03   & 0.26 \\ 
      56833      & 1198080    & 9.67e-01   & 0.12       & 1.93e-01   & 0.09       & 4.90e-03   & 0.22 \\ 
      \midrule
      \multicolumn{8}{c}{$k=2$} \\
      \midrule \\
      273        & 10144      & 1.62e+00   & --         & 2.05e-01   & --         & 9.69e-03   & --   \\ 
      1185       & 40000      & 1.49e+00   & 0.12       & 1.92e-01   & 0.09       & 7.05e-03   & 0.46 \\ 
      4929       & 158848     & 1.38e+00   & 0.11       & 1.81e-01   & 0.09       & 5.75e-03   & 0.29 \\ 
      20097      & 633088     & 1.27e+00   & 0.11       & 1.71e-01   & 0.08       & 4.92e-03   & 0.23 \\ 
      81153      & 2527744    & 1.18e+00   & 0.11       & 1.62e-01   & 0.08       & 4.25e-03   & 0.21 \\ 
      \midrule
      \multicolumn{8}{c}{$k=3$} \\
      \midrule \\
      353        & 17472      & 1.84e+00   & --         & 1.79e-01   & --         & 7.00e-03   & --   \\ 
      1537       & 68864      & 1.70e+00   & 0.11       & 1.69e-01   & 0.08       & 5.82e-03   & 0.27 \\ 
      6401       & 273408     & 1.57e+00   & 0.11       & 1.61e-01   & 0.07       & 5.03e-03   & 0.21 \\ 
      26113      & 1089536    & 1.45e+00   & 0.11       & 1.53e-01   & 0.07       & 4.36e-03   & 0.20 \\ 
      105473     & 4349952    & 1.35e+00   & 0.11       & 1.46e-01   & 0.07       & 3.78e-03   & 0.21 \\ 
      \midrule
      \multicolumn{8}{c}{$k=4$} \\
      \midrule \\
      433        & 26784      & 2.05e+00   & --         & 1.62e-01   & --         & 6.24e-03   & --   \\ 
      1889       & 105536     & 1.90e+00   & 0.11       & 1.55e-01   & 0.07       & 5.33e-03   & 0.23 \\ 
      7873       & 418944     & 1.77e+00   & 0.11       & 1.48e-01   & 0.07       & 4.62e-03   & 0.21 \\ 
      32129      & 1669376    & 1.64e+00   & 0.11       & 1.41e-01   & 0.07       & 4.01e-03   & 0.21 \\ 
      129793     & 6664704    & 1.52e+00   & 0.11       & 1.35e-01   & 0.07       & 3.46e-03   & 0.21 \\ 
      \bottomrule
    \end{tabular}
  \end{footnotesize}
\end{figure}


\section{Proofs}\label{sec:proofs}

This section collects the proofs of Theorems \ref{thm:well-posedness} and \ref{thm:err.est} preceded by the required intermediate results.

\subsection{Comparison of local seminorms}

  In this section we prove a technical proposition that contains comparison results for the local Stokes and Darcy seminorms.
  \begin{proposition}[Comparison of the local Darcy and Stokes seminorms]
    Let a mesh element $T\in\Th$ be fixed.
  Recalling the definition \eqref{eq:norm.ET} of the boundary seminorm $\seminorm[1,\partial T]{{\cdot}}$, it holds for all $\uvec{v}_T\in\UT$ that
  \begin{equation}\label{eq:rDT-vT}
    \norm[T]{\rDT\uvec{v}_T - \VEC{v}_T}
    \lesssim h_T \seminorm[1,\partial T]{\uvec{v}_T}
  \end{equation}
  and
  \begin{equation}\label{eq:seminorm.DT<=seminorm.1pT}
    \seminorm[\darcy,T]{\uvec{v}_T}^2
    \lesssim(2\mu_T)\Cf\seminorm[1,\partial T]{\uvec{v}_T}^2
    \mbox{ where }
    \seminorm[\darcy,T]{\uvec{v}_T}\coloneq\mathrm{s}_{\darcy,T}(\uvec{v}_T,\uvec{v}_T)^{\frac12}.
  \end{equation}
  Moreover,
  \begin{equation}\label{eq:norm.ET<=norm.DT}
    (2\mu_T)\norm[\strain,T]{\uvec{v}_T}^2\lesssim\Cf^{-1}\norm[\darcy,T]{\uvec{v}_T}^2\mbox{ if }\nu_T>0.
  \end{equation}
\end{proposition}
\begin{proof}
  (i) \emph{Proof of \eqref{eq:rDT-vT}.}
   We apply the estimate
   $
  \norm[T]{\VEC{\varphi}_T}
  \lesssim\left(
  \norm[T]{\vlproj[T]{k-1}\VEC{\varphi}_T}^2
  + \sum_{F\in\Fh[T]}h_F\norm[F]{\VEC{\varphi}_T\SCAL\normal_{TF}}^2
  \right)^{\frac12}
  $
  valid for any function $\VEC{\varphi}_T\in\RTN(T)$ to $\VEC{\varphi}_T=\rDT\uvec{v}_T-\VEC{v}_T$ to infer
  $$
  \begin{aligned}
    \norm[T]{\rDT\uvec{v}_T-\VEC{v}_T}
    &\lesssim\left(
    \sum_{F\in\Fh[T]}h_F\norm[F]{(\rDT\uvec{v}_T-\VEC{v}_T)\SCAL\normal_{TF}}^2
    \right)^{\frac12}
    \\
    &=\left(
    \sum_{F\in\Fh[T]}h_F\norm[F]{(\VEC{v}_F-\VEC{v}_T)\SCAL\normal_{TF}}^2
    \right)^{\frac12}
    \le h_T\seminorm[1,\partial T]{\uvec{v}_T},
  \end{aligned}
  $$
  where we have used the characterisation \eqref{eq:rDT.bis} of the local Darcy velocity reconstruction in the first two passages and the fact that $h_F\le h_T$ for all $F\in\Fh[T]$ to conclude.
  \medskip

  (ii) \emph{Proof of \eqref{eq:seminorm.DT<=seminorm.1pT}.}
  The volumetric term in $\seminorm[\darcy,T]{{\cdot}}$ is zero if $k\ge 2$ (see \eqref{eq:dDT=0.k>=2}).
  If $k=1$, on the other hand, we can write
  \begin{equation}\label{eq:norm.dDT<=seminorm.1pT.k=1}
    \nu_T\norm[T]{\dDT[1]\uvec{v}_T}^2
    =\nu_T\norm[T]{\vlproj[T]{1}(\rDT[1]\uvec{v}_T-\VEC{v}_T)}^2
    \le\nu_T\norm[T]{\rDT[1]\uvec{v}_T-\VEC{v}_T}^2
    \lesssim (2\mu_T)\Cf\seminorm[1,\partial T]{\uvec{v}_T}^2,
  \end{equation}
  where we have used the definition \eqref{eq:dDT.dDTF} of $\dDT[1]$ in the equality, the boundedness of $\vlproj[T]{1}$ expressed by \eqref{eq:lproj:boundedness} with $X=T$, $\ell=1$, and $s=0$ in the first bound, and \eqref{eq:rDT-vT} together with the definition \eqref{eq:Cf} of the local friction coefficient to conclude.
  
  On the other hand, for all $F\in\Fh[T]$ we have that
  $$
  \norm[F]{\dDTF\uvec{v}_T}
  = \norm[F]{\vlproj[F]{k}(\rDT\uvec{v}_T-\VEC{v}_F)}
  \le\norm[F]{\rDT\uvec{v}_T-\VEC{v}_F}
  \lesssim h_F^{-\frac12}\norm[T]{\rDT\uvec{v}_T-\VEC{v}_T} + \norm[F]{\VEC{v}_F-\VEC{v}_T},
  $$
  where we have used the definition \eqref{eq:dDT.dDTF} of $\dDTF$ in the equality, invoked the boundedness of $\vlproj[F]{k}$ expressed by \eqref{eq:lproj:boundedness} with $X=F$, $\ell=k$, and $s=0$ in the first bound, inserted $\pm\VEC{v}_T$ into the norm and used the triangle inequality together with a discrete trace inequality (see, e.g., \cite[Lemma 1.46]{Di-Pietro.Ern:12}) to conclude.
  Squaring the above inequality, multiplying it by $\nu_Th_F$, summing over $F\in\Fh[T]$, and recalling the definition \eqref{eq:Cf} of the local friction coefficient after observing that $h_F\le h_T$, we obtain
  $$
  \sum_{F\in\Fhi[T]}\nu_T h_F\norm[F]{\dDTF\uvec{v}_T}^2
  \lesssim\nu_T\norm[T]{\rDT\uvec{v}_T-\VEC{v}_T}^2
  + (2\mu_T)\Cf\seminorm[1,\partial T]{\uvec{v}_T}^2.
  $$
  Using \eqref{eq:rDT-vT} and again \eqref{eq:Cf} to bound the first term in the right-hand side, we arrive at
  \begin{equation}\label{eq:norm.dDTF<=seminorm.1pT}
    \sum_{F\in\Fhi[T]}\nu_T h_F\norm[F]{\dDTF\uvec{v}_T}^2
    \lesssim (2\mu_T)\Cf\seminorm[1,\partial T]{\uvec{v}_T}^2.
  \end{equation}
  Estimate \eqref{eq:norm.dDTF<=seminorm.1pT}, added to \eqref{eq:norm.dDT<=seminorm.1pT.k=1} in the case $k=1$, concludes the proof of \eqref{eq:seminorm.DT<=seminorm.1pT}.
  \medskip

  (iii) \emph{Proof of \eqref{eq:norm.ET<=norm.DT}.}
  For any $F\in\Fh[T]$, recalling the definition \eqref{eq:dDT.dDTF} of the Darcy difference operators, we have that
  $$
  \norm[F]{\VEC{v}_F-\VEC{v}_T}
  \le
  \norm[F]{\dDTF\uvec{v}_T}
  + \norm[F]{\dDT\uvec{v}_T}
  + \norm[F]{\vlproj[F]{k}\rDT\uvec{v}_T-\vlproj[T]{l}\rDT\uvec{v}_T},
  $$
  where we have inserted $\pm(\vlproj[F]{k}\rDT\uvec{v}_T-\vlproj[T]{l}\rDT\uvec{v}_T)$ into the norm and used the triangle inequality to conclude.
  Using $l\le k$ (see \eqref{eq:l}) and the idempotency, linearity, and boundedness of $\vlproj[F]{k}$ to write $\norm[F]{\vlproj[F]{k}\rDT\uvec{v}_T-\vlproj[T]{l}\rDT\uvec{v}_T}=\norm[F]{\vlproj[F]{k}(\rDT\uvec{v}_T-\vlproj[T]{l}\rDT\uvec{v}_T)}\le\norm[F]{\rDT\uvec{v}_T-\vlproj[T]{l}\rDT\uvec{v}_T}$ together with a standard discrete trace inequality, we can go on writing
  $$
  \begin{aligned}
    \norm[F]{\VEC{v}_F-\VEC{v}_T}
    &\lesssim\norm[F]{\dDTF\uvec{v}_T}
    + h_F^{-\frac12}\norm[T]{\dDT\uvec{v}_T}
    + h_F^{-\frac12}\norm[T]{\rDT\uvec{v}_T-\vlproj[T]{l}\rDT\uvec{v}_T}
    \\
    &\lesssim\norm[F]{\dDTF\uvec{v}_T}
    + h_F^{-\frac12}\norm[T]{\dDT\uvec{v}_T}
    + h_F^{-\frac12}\norm[T]{\rDT\uvec{v}_T},
  \end{aligned}
  $$
  where the conclusion follows using the triangle inequality in the last term together with the boundedness of $\vlproj[T]{l}$ expressed by \eqref{eq:lproj:boundedness} with $X=T$, $\ell=l$, and $s=0$.
  Raising the above inequality to the square, multiplying by $(2\mu_T)h_F^{-1}$ both sides, summing over $F\in\Fh[T]$, and using the uniform equivalence $h_F\simeq h_T$ followed by the definition \eqref{eq:Cf} of the local friction coefficient, we conclude that
  \begin{equation}\label{eq:seminorm.1pT<=norm.DT}
    (2\mu_T)\seminorm[1,\partial T]{\uvec{v}_T}^2
    \lesssim\Cf^{-1}\norm[\darcy,T]{\uvec{v}_T}^2.
  \end{equation}
  We next write
  \begin{align}
      (2\mu_T)^{\frac12}\norm[T]{\GRADs\VEC{v}_T}
      &\lesssim (2\mu_T)^{\frac12}h_T^{-1}\norm[T]{\VEC{v}_T}
      \nonumber\\
      &\le (2\mu_T)^{\frac12}h_T^{-1}\norm[T]{\rDT\uvec{v}_T-\VEC{v}_T} + \Cf^{-\frac12}\norm[T]{\nu_T^{\frac12}\rDT\uvec{v}_T}
      \nonumber\\
      &\le (2\mu_T)^{\frac12}\seminorm[1,\partial T]{\uvec{v}_T} + \Cf^{-\frac12}\norm[T]{\nu_T^{\frac12}\rDT\uvec{v}_T}
      \lesssim\Cf^{-\frac12}\norm[\darcy,T]{\uvec{v}_T},
		\label{eq:norm.Gs<=norm.DT}
  \end{align}
  where we have used a standard discrete inverse inequality (see, e.g., \cite[Section 1.7]{Ern.Guermond:04}) in the first line, we have inserted $\pm\rDT\uvec{v}_T$ into the norm and used a triangle inequality together with the definition \eqref{eq:Cf} of the local friction coefficient to pass to the second line, and we have used \eqref{eq:rDT-vT} to pass to the third line and \eqref{eq:seminorm.1pT<=norm.DT} to conclude.
  Squaring \eqref{eq:norm.Gs<=norm.DT}, summing it to \eqref{eq:seminorm.1pT<=norm.DT}, and recalling the definition \eqref{eq:norm.ET} of the local strain seminorm, the conclusion follows.
\end{proof}%

\subsection{Well-posedness}\label{sec:proofs:well-posedness}

Well-posedness classically hinges on the following inf--sup condition on $\mathrm{b}_h$.
\begin{lemma}[Stability of the velocity--pressure coupling]
  For all $q_h\in P_h^k$, it holds with $\beta$ defined by \eqref{eq:a-priori}
  \begin{equation}\label{eq:bh.stability}
    \beta\norm{q_h}\lesssim\sup_{\uvec{v}_h\in\UhD\setminus\{\uvec{0}\}}
    \frac{\mathrm{b}_h(\uvec{v}_h,q_h)}{\norm[\VEC{U},h]{\uvec{v}_h}}.
  \end{equation}
\end{lemma}
\begin{proof}
  (i) \emph{Boundedness of $\Ih$.}
  We start by proving the following boundedness property for $\Ih$:
  For all $\VEC{v}\in H_0^1(\Omega)^d$,
  \begin{equation}\label{eq:norm.Uh.boundedness.Ih}
    \norm[\VEC{U},h]{\Ih\VEC{v}}\lesssim\beta^{-1}\norm[H^1(\Omega)^d]{\VEC{v}}.
  \end{equation}
  A straightforward adaptation of the arguments of \cite[Proposition 7.1]{Di-Pietro.Droniou:17} gives for the Stokes norm of $\Ih\VEC{v}$:
  \begin{equation}\label{eq:norm.Sh.boundedness.Ih} 
    \norm[\stokes,h]{\IT\VEC{v}}^2\lesssim(2\overline{\mu})\seminorm[H^1(\Omega)^d]{\VEC{v}}^2.
  \end{equation}
  We next estimate the Darcy norm of $\Ih\VEC{v}$.
  Recalling the definitions \eqref{eq:aDh.coercivity} and \eqref{eq:seminorm.DT<=seminorm.1pT} of the local Darcy norm $\norm[\darcy,T]{{\cdot}}$ and seminorm $\seminorm[\darcy,T]{{\cdot}}$, and replacing $\rDT\IT\VEC{v}$ by $\IRTNT\VEC{v}$ for each $T\in\Th$ (see \eqref{eq:rDh.Ih=IRTN}), we can write
  $$
  \norm[\darcy,h]{\Ih\VEC{v}}^2
  =\sum_{T\in\Th}\left(
  \nu_T\norm[T]{\IRTNT\VEC{v}}^2
  + \seminorm[\darcy,T]{\IT\VEC{v}}^2    
  \right).
  $$
  For any mesh element $T\in\Th$, using the seminorm comparison \eqref{eq:seminorm.DT<=seminorm.1pT} and recalling the definition \eqref{eq:Cf} of the local friction coefficient together with the bound $\seminorm[1,\partial T]{\IT\VEC{v}}\lesssim\seminorm[H^1(T)^d]{\VEC{v}}$ (see again \cite[Proposition 7.1]{Di-Pietro.Droniou:17}), it is readily inferred that
  $$
  \seminorm[\darcy,T]{\IT\VEC{v}}
  \lesssim\nu_T^{\frac12}h_T\seminorm[1,\partial T]{\IT\VEC{v}}
  \lesssim\overline{\nu}^{\frac12} d_\Omega\seminorm[H^1(T)^d]{\VEC{v}}
  \lesssim\overline{\nu}^{\frac12}\seminorm[H^1(T)^d]{\VEC{v}},
  $$
  with $d_\Omega$ denoting the diameter of $\Omega$.
  Hence, denoting by $\IRTN:H^1(\Omega)^d\to\RTN(\Th)$ the global Raviart--Thomas--N\'ed\'elec interpolator whose restriction to each mesh element $T\in\Th$ coincides with $\IRTNT$ (see \eqref{eq:IRTNT}), we arrive at
  \begin{equation}\label{eq:norm.Dh.boundedness.Ih}
    \norm[\darcy,h]{\Ih\VEC{v}}^2\lesssim\overline{\nu}\left(
    \norm{\IRTN\VEC{v}}^2 + \seminorm[H^1(\Omega)^d]{\VEC{v}}^2
    \right)
    \lesssim\overline{\nu}\norm[H^1(\Omega)^d]{\VEC{v}}^2,
  \end{equation}
  where we have used standard boundedness properties of $\IRTN$ (see, e.g., \cite[Lemma 4.4]{Gatica:14}) to conclude.
  Summing \eqref{eq:norm.Sh.boundedness.Ih} and \eqref{eq:norm.Dh.boundedness.Ih}, recalling the definition \eqref{eq:norm.Uh} of $\norm[\VEC{U},h]{{\cdot}}$ with $\uvec{v}_h=\Ih\VEC{v}$, and passing to the square root, \eqref{eq:norm.Uh.boundedness.Ih} follows.
  \medskip\\
  (ii) \emph{Conclusion.}
  Let $q_h\in P_h$. From the surjectivity of the continuous divergence operator from $\VEC{U}$ to $P$ (see, e.g., \cite[Section 2.2]{Girault.Raviart:86}), we infer the existence of $\VEC{v}_{q_h}\in\VEC{U}$ such that $-\DIV\VEC{v}_{q_h}=q_h$ and $\norm[H^1(\Omega)^d]{\VEC{v}_{q_h}}\lesssim\norm{q_h}$, with hidden constant depending only on $\Omega$.
  Using the above fact, we can write
  $$ 
    \norm{q_h}^2
    = -(\DIV\VEC{v}_{q_h},q_h)     
    = \mathrm{b}_h(\Ih\VEC{v}_{q_h},q_h),
  $$ 
  where we have used the consistency property \eqref{eq:bh.consistency} of $\mathrm{b}_h$ with $\VEC{w}=\VEC{v}_{q_h}$ to conclude.
  Hence, denoting by $\$$ the supremum in the right-hand side of \eqref{eq:bh.stability} and using \eqref{eq:norm.Uh.boundedness.Ih}, it holds that
  $$
  \norm{q_h}^2
  \le\$\norm[\VEC{U},h]{\Ih\VEC{v}_{q_h}}
  \lesssim\$\beta^{-1}\norm[H^1(\Omega)^d]{\VEC{v}_{q_h}}
  \lesssim\$\beta^{-1}\norm{q_h}.
  $$
  This concludes the proof.
\end{proof}
In the proof of Theorem \ref{thm:well-posedness} below, we will need the following global discrete Korn inequality that descends from \cite[Proposition 20]{Botti.Di-Pietro.ea:17} (based, in turn, on the results of \cite{Brenner:04}):
For all $\uvec{v}_h\in\UhD$, recalling the definition \eqref{eq:vh} of $\VEC{v}_h\in\Poly{l}(\Th)^d$, it holds that
\begin{equation}\label{eq:korn.h}
  \norm{\VEC{v}_h}
  + \norm[1,h]{\uvec{v}_h}\lesssim\norm[\strain,h]{\uvec{v}_h}
  \mbox{ where }
  \norm[1,h]{\uvec{v}_h}^2\coloneq\sum_{T\in\Th}\left(
  \norm[T]{\GRAD\VEC{v}_T}^2 + \seminorm[1,\partial T]{\uvec{v}_T}^2
  \right).
\end{equation}
We are now ready to prove well-posedness.
\begin{proof}[Proof of Theorem \ref{thm:well-posedness}]
  The bilinear form $\mathrm{a}_h$ is coercive and bounded (with coercivity and boundedness constants equal to $1$) on $\UhD$ equipped with the $\norm[\VEC{U},h]{{\cdot}}$-norm (see \eqref{eq:norm.S.strain.h}, and \eqref{eq:aDh.coercivity}), and the bilinear form $\mathrm{b}_h$ is inf--sup stable on this space (see \eqref{eq:bh.stability}).
  Hence, using the fact that $\UhD$ is finite-dimensional, we can apply \cite[Theorem 2.34]{Ern.Guermond:04} to infer that problem \eqref{eq:discrete} is well-posed and that the following a priori bound holds:
  \begin{equation}\label{eq:a-priori.basic}
    \norm[\VEC{U},h]{\uvec{u}_h}
    + \beta\norm{p_h}
    \lesssim \norm[\VEC{U}^*,h]{\mathfrak{f}} + \beta^{-1}\norm{g},
  \end{equation}
  where we have denoted by $\mathfrak{f}$ the linear functional on $\UhD$ such that, for all $\uvec{v}_h\in\UhD$, $\langle\mathfrak{f},\uvec{v}_h\rangle \coloneq (\VEC{f},\rDh\uvec{v}_h)$.
  We next observe that, for all $T\in\Th$,
  $$
  \begin{aligned}
    (\VEC{f},\rDT\uvec{v}_T)_T
    &=(\VEC{f},\VEC{v}_T)_T + (\VEC{f},\rDT\uvec{v}_T-\VEC{v}_T)_T
    \\
    &\le\norm[T]{\VEC{f}}\left(
    \norm[T]{\VEC{v}_T} + \norm[T]{\rDT\uvec{v}_T-\VEC{v}_T}
    \right)
    \\
    &\lesssim\norm[T]{\VEC{f}}\left(
    \norm[T]{\VEC{v}_T} + h_T\seminorm[1,\partial T]{\uvec{v}_T}
    \right),
  \end{aligned}
  $$
  where we have used the Cauchy--Schwarz inequality to pass to the second line and \eqref{eq:rDT-vT} to conclude.
  Using the previous bound, a discrete Cauchy--Schwarz inequality on the sum over $T\in\Th$, and the fact that, for all $T\in\Th$, $h_T\le d_\Omega$ with $d_\Omega$ denoting the diameter of $\Omega$, we arrive at
  $$
  |(\VEC{f},\rDh\uvec{v}_h)|
  \lesssim\norm{\VEC{f}}\left(
  \norm{\VEC{v}_h} + \norm[\strain,h]{\uvec{v}_h}
  \right)
  \lesssim\norm{\VEC{f}}\norm[\strain,h]{\uvec{v}_h}
  \lesssim (2\underline{\mu})^{-\frac12}\norm{\VEC{f}}\norm[\VEC{U},h]{\uvec{v}_h},
  $$
  where we have used the discrete Korn inequality \eqref{eq:korn.h} in the second bound and, to conclude, we have invoked the uniform global norm equivalence \eqref{eq:aSh.stability} together with \eqref{eq:norm.Uh} to write $\norm[\strain,h]{\uvec{v}_h}\lesssim(2\underline{\mu})^{-\frac12}\norm[\stokes,h]{\uvec{v}_h}\lesssim (2\underline{\mu})^{-\frac12}\norm[\VEC{U},h]{\uvec{v}_h}$.
  Plug this bound into the definition \eqref{eq:norm.U*h} of $\norm[\VEC{U}^*,h]{\mathfrak{f}}$ to infer \eqref{eq:a-priori}.
\end{proof}

\subsection{Convergence}\label{sec:proofs:convergence}

This section contains the proof of Theorem \ref{thm:err.est} preceded by two intermediate lemmas containing consistency results for the Stokes and Darcy bilinear forms

\subsubsection{Consistency of the Stokes bilinear form}

\begin{lemma}[Consistency of the Stokes bilinear form]
  It holds for all $\VEC{w}\in H^1(\Omega)^d\cap H^{k+2}(\Th)^d$ such that $\DIV(2\mu\GRADs\VEC{w})\in L^2(\Omega)^d$,
  \begin{equation}\label{eq:aSh.consistency}
    \sup_{\uvec{v}_h\in\UhD,\norm[\VEC{U},h]{\uvec{v}_h}=1}\left|
    \Err[\stokes,h](\VEC{w};\uvec{v}_h)
    \right|
    \lesssim\left(
    \sum_{T\in\Th}(2\mu_T) \min(1,\Cf^{-1}) h_T^{2(k+1)}\seminorm[H^{k+2}(T)^d]{\VEC{w}}^2
    \right)^{\frac12}
  \end{equation}
  with Stokes consistency error
  \begin{equation}\label{eq:err.S}
    \Err[\stokes,h](\VEC{w};\uvec{v}_h)\coloneq -(\DIV(2\mu\GRADs\VEC{w}),\VEC{v}_h) - \mathrm{a}_{\stokes,h}(\Ih\VEC{w},\uvec{v}_h).
  \end{equation}
\end{lemma}
\begin{proof}
  We proceed to bound the consistency error for a generic $\uvec{v}_h\in\UhD$.
  For the sake of brevity, throughout the proof we let, for all $T\in\Th$, $\cvec{w}_T\coloneq\rST\IT\VEC{w}=\elproj[T]{k+1}\VEC{w}$ (see \eqref{eq:rST.IT=elproj}).

  We start by noting the following consistency property for the stabilisation term valid under Assumption \ref{ass:sST}, whose proof follows using the arguments of \cite[Proposition 3.1]{Di-Pietro.Tittarelli:18}:
  For all $T\in\Th$,
  \begin{equation}\label{eq:sST.approx}
    \seminorm[\stokes,T]{\IT\VEC{w}}
    \lesssim (2\mu_T)^{\frac12}h_T^{k+1}\seminorm[H^{k+2}(T)^d]{\VEC{w}}\mbox{ where, for all $\uvec{v}_T\in\UT$, } \seminorm[\stokes,T]{\uvec{v}_T}\coloneq\mathrm{s}_{\stokes,T}(\uvec{v}_T,\uvec{v}_T)^{\frac12}.
  \end{equation}

  We next find a more convenient reformulation of the terms composing the consistency error.
  Integrating by parts element by element, and using the continuity of the normal component of $2\mu\GRADs\VEC{w}$ across interfaces together with the strongly enforced boundary conditions to insert $\VEC{v}_F$ into the boundary term, we have that
  \begin{equation}\label{eq:aSh.consistency:1}
    -(\DIV(2\mu\GRADs\VEC{w}),\VEC{v}_h)
    = \sum_{T\in\Th}\left(
    2\mu_T(\GRADs\VEC{w},\GRADs\VEC{v}_T)_T
    + \sum_{F\in\Fh[T]}2\mu_T(\GRADs\VEC{w}\normal_{TF},\VEC{v}_F-\VEC{v}_T)_F
    \right).
  \end{equation}
  On the other hand, plugging \eqref{eq:aST} into \eqref{eq:aSh}, and expanding, for all $T\in\Th$, the consistency term involving $\rST\uvec{v}_T$ according to its definition \eqref{eq:rST:pde} with $\VEC{w}=\cvec{w}_T$, it is inferred that
  \begin{multline}\label{eq:aSh.consistency:2}
    \mathrm{a}_{\stokes,h}(\Ih\VEC{w},\uvec{v}_h)
    =
    \\
    \sum_{T\in\Th}\left(
    2\mu_T(\GRADs\cvec{w}_T,\GRADs\VEC{v}_T)_T
    + \sum_{F\in\Fh[T]}2\mu_T(\GRADs\cvec{w}_T\normal_{TF},\VEC{v}_F-\VEC{v}_T)_F
    + \mathrm{s}_{\stokes,T}(\IT\VEC{w},\uvec{v}_T)_T
    \right).
  \end{multline}
  Subtracting \eqref{eq:aSh.consistency:2} from \eqref{eq:aSh.consistency:1}, taking absolute values, and using the definition \eqref{eq:elproj.k+1} of the strain projector to cancel the first terms in parentheses, we get
  \begin{equation}\label{eq:err.Sh.development}
    \begin{aligned}
      |\Err[\stokes,h](\VEC{w};\uvec{v}_h)|
      &=\left|
      \sum_{T\in\Th}\left(
      \sum_{F\in\Fh[T]}2\mu_T(\GRADs(\VEC{w}-\cvec{w}_T)\normal_{TF},\VEC{v}_F-\VEC{v}_T)_F
      - \mathrm{s}_{\stokes,T}(\IT\VEC{w},\uvec{v}_T)_T
      \right)
      \right|
      \\
      &\le
      \sum_{T\in\Th}\left(
      (2\mu_T) h_T^{\frac12}\norm[\partial T]{\GRADs(\VEC{w}-\cvec{w}_T)}\seminorm[1,\partial T]{\uvec{v}_T}
      + \seminorm[\stokes,T]{\IT\VEC{w}}\seminorm[\stokes,T]{\uvec{v}_T}
      \right)
      \\
      &\lesssim
      \sum_{T\in\Th}\left(
      (2\mu_T)^{\frac12} h_T^{k+1}\seminorm[H^{k+2}(T)^d]{\VEC{w}}~
      (2\mu_T)^{\frac12}\norm[\strain,T]{\uvec{v}_T}
      \right),
    \end{aligned}
  \end{equation}
  where we have used Cauchy--Schwarz inequalities to pass to the second line, and we have concluded using the approximation properties \eqref{eq:elproj:approx.trace} of the strain projector and \eqref{eq:sST.approx} of the Stokes stabilisation bilinear form, as well as the uniform local seminorm equivalence \eqref{eq:aST.stability}.
  Let now an element $T\in\Th$ be fixed.
  We distinguish two cases according to the value of the local friction coefficient.
  If $\Cf^{-1}\le 1$ (which means, in particular, that $\nu_T>0$), by virtue of \eqref{eq:norm.ET<=norm.DT} we can write
  $$
  (2\mu_T)^{\frac12}\norm[\strain,T]{\uvec{v}_T}
  \lesssim\Cf^{-\frac12}\norm[\darcy,T]{\uvec{v}_T}
  =\min(1,\Cf^{-1})^{\frac12}\norm[\darcy,T]{\uvec{v}_T}
  \le\min(1,\Cf^{-1})^{\frac12}\norm[\VEC{U},T]{\uvec{v}_T}.
  $$
  If, on the other hand, $\Cf^{-1}>1$ (with $\Cf^{-1}=+\infty$ corresponding to $\nu_T=0$), recalling the uniform local seminorm equivalence \eqref{eq:aST.stability}, it holds
  $$
  (2\mu_T)^{\frac12}\norm[\strain,T]{\uvec{v}_T}
  \lesssim\norm[\stokes,T]{\uvec{v}_T}
  = \min(1,\Cf^{-1})^{\frac12}\norm[\stokes,T]{\uvec{v}_T}
  \le\min(1,\Cf^{-1})^{\frac12}\norm[\VEC{U},T]{\uvec{v}_T}.
  $$
  Plugging the above estimates into \eqref{eq:err.Sh.development} and using a discrete Cauchy--Schwarz inequality on the sum over $T\in\Th$, the conclusion follows.
\end{proof}

\subsubsection{Consistency of the Darcy bilinear form}

To prove Lemma \ref{lem:aDh.consistency} below, we will need the following optimal approximation properties (see, e.g., \cite[Lemma 3.17]{Gatica:14} for a proof):
For all $T\in\Th$ and all $m\in\{0,\ldots,k+1\}$,
\begin{equation}\label{eq:IRTNT.approx}
  \seminorm[H^m(T)^d]{\IRTNT\VEC{w}-\VEC{w}}
  \lesssim h_T^{k+1-m}\seminorm[H^{k+1}(T)^d]{\VEC{w}}.
\end{equation}
Using the trace inequality \eqref{eq:trace.T} below with $X=T$ in conjunction with \eqref{eq:IRTNT.approx} for $m=0$ and $m=1$, we additionally infer that it holds, for all $F\in\Fh[T]$,
\begin{equation}\label{eq:IRTNT.approx.trace}
  \norm[F]{\IRTNT\VEC{w}-\VEC{w}}\lesssim h_T^{k+\frac12}\seminorm[H^{k+1}(T)^d]{\VEC{w}}.
\end{equation}

\begin{lemma}[Consistency of the Darcy bilinear form]\label{lem:aDh.consistency}
  For all $\VEC{w}\in H^1(\Omega)^d\cap H^{k+1}(\Th)^d$, it holds that
  \begin{equation}\label{eq:aDh.consistency}
    \sup_{\uvec{v}_h\in\UhD,\norm[\VEC{U},h]{\uvec{v}_h}=1}\left|
    \Err[\darcy,h](\VEC{w};\uvec{v}_h)
    \right|\lesssim\left(
    \sum_{T\in\Th}\nu_T\min(1,\Cf) h_T^{2(k+1)}\seminorm[H^{k+1}(T)^d]{\VEC{w}}^2
    \right)^{\frac12}
  \end{equation}
  with Darcy consistency error
  \begin{equation}\label{eq:err.D}
    \Err[\darcy,h](\VEC{w};\uvec{v}_h)\coloneq(\nu\VEC{w},\rDh\uvec{v}_h) - \mathrm{a}_{\darcy,h}(\Ih\VEC{w},\uvec{v}_h).
  \end{equation}
\end{lemma}
\begin{proof}
  We decompose the consistency error as follows:
  \begin{equation}\label{eq:EDh.decomposition}
    \Err[\darcy,h](\VEC{w};\uvec{v}_h)
    \eqcolon\term_1 + \term_2
    =\sum_{T\in\Th}\term_1(T) + \sum_{T\in\Th}\term_2(T)
  \end{equation}
  where, for all $T\in\Th$, we have set
  $$
  \term_1(T)\coloneq\nu_T(\VEC{w} - \rDT\IT\VEC{w},\rDT\uvec{v}_T)_T,\qquad
  \term_2(T)\coloneq\mathrm{s}_{\darcy,T}(\IT\VEC{w},\uvec{v}_T).
  $$
  \\
  (i) \emph{Estimate of $\term_1$.}
  For any $T\in\Th$ such that $\Cf>1$, using \eqref{eq:rDh.Ih=IRTN} to replace $\rDT\IT\VEC{w}$ by $\IRTNT\VEC{w}$, we can write
  \begin{equation}\label{eq:aDh.consistency:T1(T):Cf>1}
    \begin{aligned}
      |\term_1(T)|
      &= |\nu_T(\VEC{w}-\IRTNT\VEC{w},\rDT\uvec{v}_T)_T|
      \\
      &\le\nu_T^{\frac12}\norm[T]{\VEC{w}-\IRTNT\VEC{w}}\norm[T]{\nu_T^{\frac12}\rDT\uvec{v}_T}
      \\
      &\lesssim\nu_T^{\frac12} \min(1,\Cf)^{\frac12} h_T^{k+1}\seminorm[H^{k+1}(T)^d]{\VEC{w}}\norm[\darcy,T]{\uvec{v}_T}
      \\
      &\le\nu_T^{\frac12} \min(1,\Cf)^{\frac12} h_T^{k+1}\seminorm[H^{k+1}(T)^d]{\VEC{w}}\norm[\VEC{U},T]{\uvec{v}_T},
    \end{aligned}
  \end{equation}
  where we have used the Cauchy--Schwarz inequality to pass to the second line, the approximation properties \eqref{eq:IRTNT.approx} with $m=0$ and the fact that $1=\min(1,\Cf)$ since $\Cf>1$ together with the definition \eqref{eq:aDh.coercivity} of the $\norm[\darcy,T]{{\cdot}}$-norm to pass to the third line, and the definition \eqref{eq:norm.Uh} of the $\norm[\VEC{U},T]{{\cdot}}$-norm to conclude.
  
  If $\Cf\le1$ and $k\ge 2$, on the other hand, we can write
  \begin{equation}\label{eq:aDh.consistency:T1(T):Cf<=1.k>=2}
    \begin{aligned}
      |\term_1(T)|
      &=|\nu_T (\VEC{w}-\IRTNT\VEC{w},\rDT\uvec{v}_T-\VEC{v}_T)_T|
      \\
      &\le\nu_T\norm[T]{\VEC{w}-\IRTNT\VEC{w}}\norm[T]{\rDT\uvec{v}_T-\VEC{v}_T}
      \\
      &\lesssim\nu_T^{\frac12} h_T^{k+1}\seminorm[H^{k+1}(T)^d]{\VEC{w}}~\nu_T^{\frac12}h_T\seminorm[1,\partial T]{\uvec{v}_T}
      \\
      &\lesssim\nu_T^{\frac12} \Cf^{\frac12} h_T^{k+1}\seminorm[H^{k+1}(T)^d]{\VEC{w}}\norm[\stokes,T]{\uvec{v}_T},
      \\
      &=\nu_T^{\frac12} \min(1,\Cf)^{\frac12} h_T^{k+1}\seminorm[H^{k+1}(T)^d]{\VEC{w}}\norm[\VEC{U},T]{\uvec{v}_T},
    \end{aligned}
  \end{equation}
  where we have used the definition \eqref{eq:IRTNT:T} of $\IRTNT$ to insert $\VEC{v}_T\in\Poly{k-1}(T)^d$ in the first line, the Cauchy--Schwarz inequality to pass to the second line, the approximation properties \eqref{eq:IRTNT.approx} of $\IRTNT$ with $m=0$ together with the bound \eqref{eq:rDT-vT} to pass to the third line, the definition \eqref{eq:Cf} of the local friction coefficient together with the uniform local seminorm equivalence \eqref{eq:aST.stability} to pass to the fourth line, and the fact $\Cf=\min(1,\Cf)$ since $\Cf\le 1$ together with the definition \eqref{eq:norm.Uh} of the $\norm[\VEC{U},T]{{\cdot}}$-norm to conclude.

  The case $\Cf\le 1$ and $k=1$ has to be treated separately because the reasoning in the first line of \eqref{eq:aDh.consistency:T1(T):Cf<=1.k>=2} breaks down as we cannot insert $\VEC{v}_T\in\Poly{1}(T)^d$ inside the second argument of the $L^2$-product.
  To overcome this difficulty, we insert instead $\vlproj[T]{0}\VEC{v}_T$, write $\rDT[1]\uvec{v}_T-\vlproj[T]{0}\VEC{v}_T=\rDT[1]\uvec{v}_T-\VEC{v}_T+\VEC{v}_T-\vlproj[T]{0}\VEC{v}_T$ and, proceeding in a similar manner as before with the additional use of a local Poincar\'e inequality corresponding to \eqref{eq:lproj:approx} with $\ell=0$, $s=1$, and $m=0$, we arrive at the following bound:
  \begin{equation}\label{eq:aDh.consistency:T1(T):Cf<=1.k=1}
    |\term_1(T)|
    \lesssim
    \nu_T^{\frac12} \min(1,\Cf)^{\frac12} h_T^2\seminorm[H^2(T)^d]{\VEC{w}}
    (2\mu_T{})^{\frac12}\left(
    \norm[T]{\GRAD\VEC{v}_T}^2 + \seminorm[1,\partial T]{\VEC{v}_T}^2
    \right)^{\frac12}.
  \end{equation}

  Gathering \eqref{eq:aDh.consistency:T1(T):Cf>1}--\eqref{eq:aDh.consistency:T1(T):Cf<=1.k=1}, using a discrete Cauchy--Schwarz inequality on the sum over $T$, invoking the discrete Korn inequality \eqref{eq:korn.h} followed by the first inequality in \eqref{eq:aSh.stability} for the case $\Cf\le 1$ and $k=1$ to further bound (introducing $\alpha_\mu$) the term resulting from the last factor in the right-hand side of \eqref{eq:aDh.consistency:T1(T):Cf<=1.k=1}, and using the fact that, by definition, $\alpha_\mu\ge 1$, we arrive at
  \begin{equation}\label{eq:aDh.consistency:T1}
    |\term_1|
    \lesssim\left(
    \sum_{T\in\Th}\alpha_\mu\nu_T\min(1,\Cf) h_T^{2(k+1)}\seminorm[H^{k+1}(T)^d]{\VEC{w}}^2
    \right)^{\frac12}\norm[\VEC{U},h]{\uvec{v}_h}.
  \end{equation}
  \medskip\\
  (ii) \emph{Estimate of $\term_2$.}
  To bound the second elementary contribution in the right-hand side of \eqref{eq:EDh.decomposition}, we preliminarily estimate the $L^2$-norms of the Darcy difference operators \eqref{eq:dDT.dDTF} when their argument is $\IT\VEC{w}$.
  Clearly, $\dDT\IT\VEC{w}$ vanishes when $k\ge 2$ (see \eqref{eq:dDT=0.k>=2}).
  When $k=l=1$, on the other hand, we can write, accounting for \eqref{eq:rDh.Ih=IRTN},
  $$
  \norm[T]{\dDT[1]\IT[1]\VEC{w}}
  =\norm[T]{\vlproj[T]{1}(\IRTNT[1]\VEC{w} - \VEC{w})}
  \le\norm[T]{\IRTNT[1]\VEC{w} - \VEC{w}}
  \lesssim h_T^2\seminorm[H^2(T)^d]{\VEC{w}},
  $$
  where we have used the boundedness of $\vlproj[T]{l}$ expressed by \eqref{eq:lproj:boundedness} with $X=T$, $\ell=k$, and $s=0$ in the first bound and the approximation properties \eqref{eq:IRTNT.approx} with $k=1$ and $m=0$ to conclude.
  On the other hand, for any $F\in\Fhi[T]$ we can write
  $$
  \norm[F]{\dDTF\IT\VEC{w}}
  =\norm[F]{\vlproj[F]{k}(\IRTNT\VEC{w}-\VEC{w})}
  \le\norm[F]{\IRTNT\VEC{w}-\VEC{w}}
  \lesssim h_T^{k+\frac12}\seminorm[H^{k+1}(T)^d]{\VEC{w}},
  $$
  where we have used \eqref{eq:rDh.Ih=IRTN}, the boundedness of $\vlproj[F]{k}$ expressed by \eqref{eq:lproj:boundedness} with $X=F$, $\ell=k$, and $s=0$, and \eqref{eq:IRTNT.approx.trace}.
  Using the Cauchy--Schwarz inequality together with the above bounds yields
  $$
  |\term_2(T)|
  \le\seminorm[\darcy,T]{\IT\VEC{w}}\seminorm[\darcy,T]{\uvec{v}_T}
  \lesssim\nu_T^{\frac12}h_T^{k+1}\seminorm[H^{k+1}(T)^d]{\VEC{w}}\seminorm[\darcy,T]{\uvec{v}_T}
  $$
  with $\seminorm[\darcy,T]{{\cdot}}$-seminorm defined by \eqref{eq:seminorm.DT<=seminorm.1pT}.
  If $\Cf>1$, observing that $\seminorm[\darcy,T]{\uvec{v}_T}\le\norm[\darcy,T]{\uvec{v}_T}\le\norm[\VEC{U},T]{\uvec{v}_T}$ (see \eqref{eq:aDh.coercivity}, \eqref{eq:seminorm.DT<=seminorm.1pT}, and \eqref{eq:norm.Uh}), this immediately yields
  \begin{equation*}\label{eq:aDh.consistency:T2(T):Cf>1}
    |\term_2(T)|
    \lesssim\nu_T^{\frac12}\min(1,\Cf)^{\frac12}h_T^{k+1}\seminorm[H^{k+1}(T)^d]{\VEC{w}}\norm[\VEC{U},T]{\uvec{v}_T}.
  \end{equation*}
  If $\Cf\le 1$, on the other hand, recalling \eqref{eq:seminorm.DT<=seminorm.1pT} we can go on writing
  \begin{equation*}\label{eq:aDh.consistency:T2(T):Cf<=1}
    \begin{aligned}
      |\term_2(T)|
      &\lesssim\nu_T^{\frac12}\Cf^{\frac12} h_T^{k+1}\seminorm[H^{k+1}(T)^d]{\VEC{w}}~(2\mu_T)^{\frac12}\seminorm[1,\partial T]{\uvec{v}_T}
      \\
      &\le\nu_T^{\frac12}\min(1,\Cf)^{\frac12} h_T^{k+1}\seminorm[H^{k+1}(T)^d]{\VEC{w}}\norm[\stokes,T]{\uvec{v}_T}
      \\
      &\lesssim\nu_T^{\frac12}\min(1,\Cf)^{\frac12} h_T^{k+1}\seminorm[H^{k+1}(T)^d]{\VEC{w}}\norm[\VEC{U},T]{\uvec{v}_T}
    \end{aligned}
  \end{equation*}
  where we have used the fact that $\Cf\le 1$ to replace $\Cf^{\frac12}$ by $\min(1,\Cf)^{\frac12}$ together with the uniform local seminorm equivalence \eqref{eq:aST.stability} in the second line, and the definition \eqref{eq:norm.Uh} of the $\norm[\VEC{U},T]{{\cdot}}$-norm to conclude.
  Gathering the above bounds and using a discrete Cauchy--Schwarz inequality on the sum over $T\in\Th$, we arrive at
  \begin{equation}\label{eq:aDh.consistency:T2}
    |\term_2|
    \lesssim\left(
    \sum_{T\in\Th}\nu_T\min(1,\Cf) h_T^{2(k+1)}\seminorm[H^{k+1}(T)^d]{\VEC{w}}^2
    \right)^{\frac12}\norm[\VEC{U},h]{\uvec{v}_h}.
  \end{equation}
  \medskip\\
  (iii) \emph{Conclusion.}
  Take the absolute value of \eqref{eq:EDh.decomposition}, use \eqref{eq:aDh.consistency:T1} and \eqref{eq:aDh.consistency:T2} (after recalling that, by \eqref{eq:alpha.mu}, $\alpha_\mu\ge 1$) to estimate the terms in the right-hand side, and use the resulting bound to estimate the supremum in \eqref{eq:aDh.consistency}.
\end{proof}

\begin{remark}[Cut-off factors]
  The proof and usage of the local cut-off factors are inspired by \cite{Di-Pietro.Droniou.ea:15}, in which an advection--diffusion--reaction model is considered. In this reference, a seamless treatment of both advection--dominated and diffusion--dominated regimes is carried out by introducing two discrete norms (advective and diffusive); the consistency errors due to the advective and diffusive terms are estimated using either one of the norm, depending on the locally dominating regime as determined by the value of a local P\'eclet number. Here, we defined Stokes and Darcy norms, and similarly estimated the consistency errors of the Stokes and Darcy terms using either norm, depending on the locally dominating regime as determined by the friction coefficients $\Cf$.
\end{remark}

\subsubsection{Error estimate and convergence}\label{sec:proofs:convergence:proof}

We are now ready to prove the main convergence result.
\begin{proof}[Proof of Theorem \ref{thm:err.est}]
  (i) \emph{Error estimates.}
  The a priori error estimate \eqref{eq:err.est} is inferred proceeding as in the proof of the a priori bound \eqref{eq:a-priori.basic} in Theorem \ref{thm:well-posedness}, this time for the error equation \eqref{eq:error.eq}.
  \medskip\\
  (ii) \emph{Convergence rate.}
  To estimate the convergence rate, we bound the dual norm of the consistency error $\err$.
  Using its definition \eqref{eq:Rup} together with the fact that \eqref{eq:strong:momentum} is satisfied almost everywhere in $\Omega$ by the weak solution $(\VEC{u},p)$ of \eqref{eq:weak}, it is inferred for all $\uvec{v}_h\in\UhD$ that
  \begin{multline}\label{eq:err.decomposition}
    \langle\err,\uvec{v}_h\rangle
    =
    \\
    (\DIV(2\mu\GRADs\VEC{u}), \VEC{v}_h - \rDh\uvec{v}_h)
    + \Err[\stokes,h](\VEC{u};\uvec{v}_h)
    + \Err[\darcy,h](\VEC{u};\uvec{v}_h)
    + \left[(\GRAD p,\rDh\uvec{v}_h) -  \mathrm{b}_h(\uvec{v}_h,\hat{p}_h)\right]
  \end{multline}
  with Stokes and Darcy errors defined by \eqref{eq:err.S} and \eqref{eq:err.D}, respectively.
  Denote by $\term_1,\ldots,\term_4$ the terms in the right-hand side.
  Recalling \eqref{eq:rDT:T}, for the first term we can write
  $$
  \term_1
  = \sum_{T\in\Th}
  2\mu_T(\DIV\GRADs\VEC{u} - \vlproj[T]{k-1}(\DIV\GRADs\VEC{u}),\VEC{v}_T-\rDT\uvec{v}_T)_T.
  $$
  Hence, using Cauchy--Schwarz inequalities followed by the optimal approximation properties \eqref{eq:lproj:approx} of $\vlproj[T]{k-1}$ with $\ell=s=k-1$ and $m=0$ together with \eqref{eq:rDT-vT}, we have that        
  \begin{equation}\label{eq:T1.basic}
    |\term_1|
    \lesssim
    \sum_{T\in\Th} (2\mu_T)^{\frac12}h_T^{k+1}\seminorm[H^{k+2}(T)^d]{\VEC{u}}
    ~(2\mu_T)^{\frac12}\seminorm[1,\partial T]{\uvec{v}_T}.
  \end{equation}
  When $\Cf^{-1}> 1$, we can write using the uniform local seminorm equivalence \eqref{eq:aST.stability} and recalling the definition \eqref{eq:norm.Uh} of the $\norm[\VEC{U},T]{{\cdot}}$-norm,
  $$
  (2\mu_T)^{\frac12}\seminorm[1,\partial T]{\uvec{v}_T}
  \lesssim\min(1,\Cf^{-1})^{\frac12}\norm[\stokes,T]{\uvec{v}_T}
  \le\min(1,\Cf^{-1})^{\frac12}\norm[\VEC{U},T]{\uvec{v}_T}.
  $$
  On the other hand, when $\Cf^{-1}\le1$ (so that, in particular, $\nu_T>0$), \eqref{eq:norm.ET<=norm.DT} gives
  $$
  (2\mu_T)^{\frac12}\seminorm[1,\partial T]{\uvec{v}_T}
  \lesssim\Cf^{-\frac12}\norm[\darcy,T]{\uvec{v}_T}
  =\min(1,\Cf^{-1})^{\frac12}\norm[\darcy,T]{\uvec{v}_T}
  \le\min(1,\Cf^{-1})^{\frac12}\norm[\VEC{U},T]{\uvec{v}_T}.
  $$
  Plugging the above bounds into \eqref{eq:T1.basic} and using a discrete Cauchy--Schwarz inequality on the sum over $T\in\Th$ finally yields
  \begin{equation}\label{eq:T1}
    |\term_1|
    \lesssim\left(
    \sum_{T\in\Th}(2\mu_T)\min(1,\Cf^{-1})h_T^{2(k+1)}\seminorm[H^{k+2}(T)^d]{\VEC{u}}^2
    \right)^{\frac12}\norm[\VEC{U},h]{\uvec{v}_h}.
  \end{equation}
  The second term is readily estimated using the consistency properties \eqref{eq:aSh.consistency} of the Stokes bilinear form:
  \begin{equation}\label{eq:T2}
    |\term_2|
    \lesssim\left(
    \sum_{T\in\Th} (2\mu_T)\min(1,\Cf^{-1})h_T^{2(k+1)}\seminorm[H^{k+2}(T)^d]{\VEC{u}}^2
    \right)^{\frac12}\norm[\VEC{U},h]{\uvec{v}_h}.
  \end{equation}
  Similarly, the consistency properties \eqref{eq:aDh.consistency} of the Darcy bilinear form yield
  \begin{equation}\label{eq:T3}
    |\term_3|
    \lesssim\left(
    \sum_{T\in\Th}\nu_T\min(1,\Cf) h_T^{2(k+1)}\seminorm[H^{k+1}(T)^d]{\VEC{u}}^2
    \right)^{\frac12}\norm[\VEC{U},h]{\uvec{v}_h}.
  \end{equation}
  In order to estimate the fourth term, we observe that
    $$
    (\GRAD p,\rDh\uvec{v}_h)
    = -(\DIV\rDh\uvec{v}_h,p)
    = -(\DIV\rDh\uvec{v}_h,\hat{p}_h)      
    = \mathrm{b}_h(\uvec{v}_h,\hat{p}_h),
    $$
    where we have used a global integration by parts in the first passage,
    standard properties of Raviart--Thomas--N\'ed\'elec functions to infer $\DIV\rDh\uvec{v}_h\in\Poly{k}(\Th)$ (see, e.g., \cite[Proposition 2.3.3]{Boffi.Brezzi.ea:13}) and replace $p$ by $\hat{p}_h$ in the second passage,
    and property \eqref{eq:bh.div.rDh} to conclude.
    This readily implies
  \begin{equation}\label{eq:T4}
    \term_4 = 0.
  \end{equation}
  
  Plugging the bounds \eqref{eq:T1}--\eqref{eq:T4} into \eqref{eq:err.decomposition}, using the resulting inequality to estimate the supremum in \eqref{eq:norm.U*h}, and recalling \eqref{eq:err.est} yields \eqref{eq:conv.rate}.
\end{proof}


\appendix

\section{Uniform Korn inequalities with application to the study of optimal approximation properties for the strain projector}\label{sec:appen}

This appendix contains two technical results whose interest goes beyond the application to the Brinkman problem: the proof of uniform Korn inequalities on star-shaped polytopal domains, and their application to the study of optimal approximation properties for the strain projector on polynomial spaces.

\subsection{Uniform local second Korn inequality}\label{app:korn}

The goal of this section is to establish a uniform second Korn inequality for mesh elements. We consider generic regular polytopal elements, which is more general than the setting considered in Section \ref{sec:setting}; note that, even if some parts of the proofs can be simplified when the elements are triangular/tetrahedral, most of the difficulty remains even for such elements.
The main result of this section is stated in the following proposition. Besides the second Korn inequality, we also state a uniform Ne\v{c}as inequality which is strongly related to the Korn inequality.

\begin{proposition}\label{prop:unif.korn}
Let $\varrho>0$. There exists $C$ depending only on $d$ and $\varrho$ such that, for any polytope $T\subset \Real^d$, of diameter $h_T$, that is star-shaped with respect to every point in a ball of radius $\varrho h_T$,
\begin{equation}\label{eq:necas}
\norm[T]{v-\lproj[T]{0}v}\le C \norm[H^{-1}(T)^d]{\GRAD v}\qquad\forall v\in L^2(T),
\end{equation}
and
\begin{equation}\label{eq:korn}
  \norm[T]{\GRAD \VEC{u}-\vlproj[T]{0}(\GRADss\VEC{u})}\le
  C \norm[T]{\GRADs\VEC{u}}\qquad\forall \VEC{u}\in H^1(T)^d.
\end{equation}
\end{proposition}

Before proving this proposition, we establish the following lemma that gives the existence of a uniform atlas for all mesh elements star-shaped with respect to every point in a ball of radius comparable to their diameter. In this lemma, for given positive number $\zeta$ and unit vector $\VEC{r}$, we let $B_\zeta$ be the open ball in $\Real^d$ centred at the origin and of radius $\zeta$, and we define the semi-infinite cylinder
$$
M(\VEC{r},\zeta)\coloneq\{
\VEC{x}\coloneq\VEC{x}^\bot+z\VEC{r}
\st\text{$\VEC{x}^\bot\in B_\zeta$ is orthogonal to $\VEC{r}$ and $z\in [0,+\infty)$}
  \}.
  $$
  Figure \ref{fig:element_sshaped} provides an illustration of this definition, as well as of other notations used in the proof of the lemma.

\begin{lemma}[Uniform atlas for star-shaped elements]\label{lem:atlas}
Let $\varrho>0$. There exists a finite number of unit vectors $(\VEC{r}_1,\ldots,\VEC{r}_m)$ in $\Real^d$ and a real number $L>0$, all depending only on $d$ and $\varrho$, such that $B_1\subset\bigcup_{l=1}^m M(\VEC{r}_l,\varrho/2)$ and, if $T$ is a polytope of $\Real^d$ contained in $B_1$ and star-shaped with respect to every point in $B_{\varrho}$, for any $l=1,\ldots,m$,
\[
T\cap M(\VEC{r}_l,\varrho/2)=\left\{
\VEC{x}=(x_1,\ldots,x_d)\in M(\VEC{r}_l,\varrho/2)\,:\,x_d\le \varphi_l(x_1,\ldots,x_{d-1})
\right\},
\]
where the system of orthogonal coordinates $(x_1,\ldots,x_d)$ is chosen such that $x_d$ is the coordinate along $\VEC{r}_l$, $H_d=\{\VEC{x}\in\Real^d\,:\,x_d=0\}$ is the horizontal hyperplane in this system of coordinates, and $\varphi_l:B_{\varrho/2}\cap H_d\to \Real$ is a Lipschitz-continuous function with Lipschitz constant bounded by $L$.
\end{lemma}

\begin{figure}
\begin{center}
\begin{tabular}{c@{\quad}c}
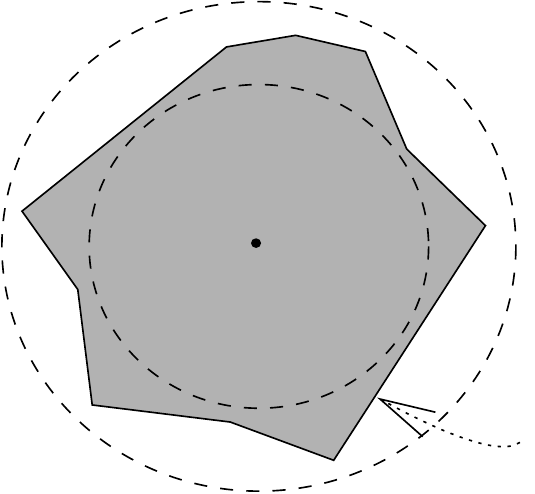&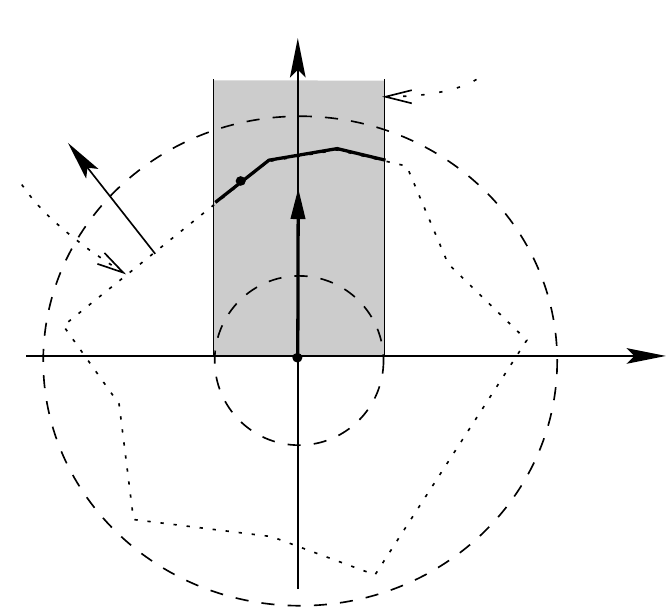
\end{tabular}
\caption{Illustration of the proof of Lemma \ref{lem:atlas}}
\label{fig:element_sshaped}
\end{center}\end{figure}

\begin{proof} In the following, $a\lesssim b$ means that $a\le Cb$ with $C$ depending only on $d$ and $\varrho$. We first notice that, since $B_1$ is determined by $d$, there is a fixed number $m$ of unit vectors $(\VEC{r}_1,\ldots,\VEC{r}_m)$, depending only on $d$ and $\varrho$, such that $B_1\subset\bigcup_{l=1}^m M(\VEC{r}_l,\varrho/2)$. The proof is complete by showing that, in each $M(\VEC{r}_l,\varrho/2)$ and in the coordinates associated with $\VEC{r}_l$ as in the lemma, $T$ is the hypograph of a Lipschitz function $\varphi_l$, with a controlled Lipschitz constant. From thereon we drop the index $l$ for legibility.

Since the boundary of $T$ is made of the faces $F\in\Fh[T]$, this function $\varphi$ is piecewise affine on each affine part corresponding to a face $F$, and it holds that
\[
\normal_{TF}=\frac{(-\GRAD_{d-1}\varphi,1)}{\sqrt{1+|\GRAD_{d-1}\varphi|^2}},
\]
where $\GRAD_{d-1}\varphi$ is the gradient in $H_d$ of $\varphi$ with respect to its $(d-1)$ variables.
Hence, $\normal_{TF}\SCAL\VEC{r}=(1+|\GRAD_{d-1}\varphi|^2)^{-\frac12}$ and if we can prove that 
\begin{equation}\label{proof:atlas:1}
1\lesssim\normal_{TF}\SCAL\VEC{r}\qquad\forall F\in\Fh[T],
\end{equation}
then we have $|\GRAD_{d-1}\varphi|\lesssim 1$, that is, the uniform control of the Lipschitz constant of $\varphi$.

To prove \eqref{proof:atlas:1}, let $F\in\Fh[T]$, $\VEC{a}\in F\cap M(\VEC{r},\varrho/2)$ and let us translate the fact that $T$ is star-shaped with respect to every point in $B_{\varrho}$. Working as in the proof of \cite[Lemma B.1]{Droniou.Eymard.ea:17}, we see that this assumption forces $B_{\varrho}$ to be fully on one side of the hyperplane spanned by $F$, which translates into
\begin{equation}\label{proof:atlas:2}
(\VEC{a}-\VEC{x})\SCAL\normal_{TF}\ge 0\quad\forall \VEC{x}\in B_{\varrho}.
\end{equation}
On the other, hand since $\VEC{a}\in M(\VEC{r},\varrho/2)$, we have $\VEC{a}=\VEC{a}^\bot+z \VEC{r}$ with $z>0$ and $\VEC{a}^\bot$ orthogonal to $\VEC{r}$ and in $B_{\varrho/2}$. 
Apply \eqref{proof:atlas:2} to $\VEC{x}=\VEC{a}^\bot + (\varrho/2)\normal_{TF}$, which
belongs to $B_{\varrho}$ since $|\VEC{a}^\bot|\le \varrho/2$.
Noticing that $\VEC{a}-\VEC{x}=z\VEC{r}-(\varrho /2)\normal_{TF}$, this yields
\begin{equation}\label{proof:atlas:3}
0\le z \VEC{r}\SCAL\normal_{TF} - \frac{\varrho}{2}.
\end{equation}
Since $\VEC{a}\in B_{1}$ and $\VEC{r}$ is a unit vector, we have $0<z\le 1$
and \eqref{proof:atlas:3} therefore gives $\VEC{r}\SCAL\normal_{TF}=z^{-1}(z\VEC{r}\SCAL\normal_{TF})\ge z^{-1} \varrho /2 \ge \varrho/2$. The proof of \eqref{proof:atlas:1} is complete. \end{proof}

We are now in a position to prove the uniform Ne\v{c}as and second Korn inequalities.

\begin{proof}[Proof of Proposition \ref{prop:unif.korn}]

  The reasoning of \cite{Tiero:99} shows that if the Ne\v{c}as inequality \eqref{eq:necas} holds with a certain $C$, then the second Korn inequality \eqref{eq:korn} holds with the constant $\sqrt{1+2C^2}$. Hence, we only have to prove that the mesh elements $T$ considered in the proposition satisfy \eqref{eq:necas} with a constant $C$ that depends only on $d$ and $\varrho$. 
This is achieved proceeding in two steps: first, we scale the problem in order to reduce the proof to the case of a polytopal set contained in the unit ball and star-shaped with respect to $B_\varrho$; second, we prove the sought result in the latter case.

\medskip

(i) \emph{Scaling.} Since the inequality is obviously invariant by translation, we can assume that $T$ is star-shaped with respect to every point in $B_{\varrho h_T}$. We then scale $T$ so that its diameter is equal to $1$. Precisely, define $\sca{T}=T/h_T$ and, for $f\in L^2(T)$, set $\sca{f}\in L^2(\sca{T})$ such that $\sca{f}(\sca{x})=f(h_T \sca{x})$ for all $\sca{x}\in\sca{T}$. Then $h_{\sca{T}}=1$ and $\sca{T}$ is star-shaped with respect to every point in $B_\varrho$. Moreover, by the change of variable $\sca{T}\ni\sca{x}\mapsto x=h_T\sca{x}\in T$, 
\begin{equation}\label{scaling.int}
\int_T f=h_T^d \int_{\sca{T}}\sca{f}
\end{equation}
and, if $f\in H^1(T)$, 
\begin{equation}\label{scaling.grad}
\GRAD\sca{f}=h_T\sca{\GRAD f}.
\end{equation}
These properties show that, for any $v\in L^2(T)$,
\begin{equation}\label{scaling.1}
\norm[T]{v-\lproj[T]{0}v}=h_T^{d/2}\norm[\sca{T}]{\sca{v}-\lproj[\sca{T}]{0}\sca{v}},
\end{equation}
and that
\begin{equation}
  \begin{aligned}
    \norm[H^{-1}(T)^d]{\GRAD v} &=\sup_{\VEC{\psi}\in H^1_0(T)^d}\frac{\int_T
      v\DIV\VEC{\psi}}{\norm[H^1_0(T)^d]{\VEC{\psi}}} =\sup_{\VEC{\psi}\in
      H^1_0(T)^d}\frac{h_T^d\int_{\sca{T}}
      \sca{v}~\sca{\DIV\VEC{\psi}}}{\norm[T]{\GRAD\VEC{\psi}}}
		\\ 
		&=\sup_{\VEC{\psi}\in
      H^1_0(T)^d}\frac{h_T^d\int_{\sca{T}}
      \sca{v}h_T^{-1}\DIV\sca{\VEC{\psi}}}{h_T^{d/2}\norm[\sca{T}]{\sca{\GRAD\VEC{\psi}}}}
    =\sup_{\sca{\VEC{\psi}}\in
      H^1_0(\sca{T})^d}\frac{h_T^dh_T^{-1}\int_{\sca{T}}
      \sca{v}\DIV\sca{\VEC{\psi}}}{h_T^{d/2}h_T^{-1}\norm[\sca{T}]{\GRAD\sca{\VEC{\psi}}}}
    =h_T^{d/2} \norm[H^{-1}(\sca{T})]{\GRAD \sca{v}},
    \label{scaling.2}
  \end{aligned}
\end{equation}
where we have used, in sequence, the definition of the norm in $H^{-1}(T)^d$, 
the change of variable \eqref{scaling.int} with $f=v\DIV\VEC{\psi}$ together with the
definition $\norm[H^1_0(T)^d]{{\cdot}}\coloneq\norm[T]{\GRAD {\cdot}}$ of
the norm $H^1_0(T)^d$, \eqref{scaling.grad} with $f=$ components of $\VEC{\psi}$ and
the change of variables \eqref{scaling.int} with $f=|\GRAD\VEC{\psi}|$,
once again the relation \eqref{scaling.grad} with $f=$ components of $\VEC{\psi}$,
and finally the definition of $\norm[H^{-1}(\sca{T})^d]{\GRAD\sca{v}}$.
If we prove \eqref{eq:necas} for all polytopal sets $\sca{T}$ of diameter $1$ and star-shaped with respect to every point in $B_\varrho$, with $C$ depending only on $d$ and $\varrho$, the relations \eqref{scaling.1}--\eqref{scaling.2} show that \eqref{eq:necas} also holds for $T$ with the same constant. To simplify the notations, in the following we drop the hat and we simply write $T$ and $v$ for $\sca{T}$ and $\sca{v}$. In other words, we have reduced the proof to the case where $T$ is a polytopal set contained in $B_1$ and star-shaped with respect to every point in $B_\varrho$.

\medskip

(ii) \emph{Proof of \eqref{eq:necas} and \eqref{eq:korn} in the scaled case.}
\cite[Theorem IV.1.1]{Boyer.Fabrie:13} establishes the existence of $C_T$ such that
\begin{equation}\label{est:Hm1.0}
\norm[T]{w}\le C_T\left(\norm[H^{-1}(T)]{w}+\norm[H^{-1}(T)]{\GRAD w}\right)\qquad\forall w\in L^2(T).
\end{equation}
The proof \cite[Theorem IV.1.1]{Boyer.Fabrie:13} gives a clear dependency on the constant $C_T$ in terms of an atlas of $\partial T$. For any $T$ contained in $B_1$ and star-shaped with respect to every point in $B_\varrho$, Lemma \ref{lem:atlas} provides an atlas of $\partial T$, whose open covering and domains and Lipschitz constant of the maps depend only on $d$ and $\varrho$. Using this atlas in the proof of \cite[Theorem IV.1.1]{Boyer.Fabrie:13}, we see that \eqref{est:Hm1.0} holds with $C_T=C_0$ depending only on $d$ and $\varrho$. Applying this inequality to $w=v-\lproj[T]{0}v$ (that has a zero integral over $T$), the Ne\v{c}as estimate \eqref{eq:necas} follows if we show that
\begin{equation}\label{est:Hm1.1}
\norm[H^{-1}(T)]{w}\le C_1 \norm[H^{-1}(T)^d]{\GRAD w}\qquad\forall w\in L^2(T)\mbox{ such that }\int_T w=0
\end{equation}
with $C_1$ depending only on $d$ and $\varrho$. This estimate is established in \cite[Proposition IV.1.7]{Boyer.Fabrie:13}, but with a proof that does not show the independence of $C_1$ with respect to the domain $T$. We adapt here this proof to show that \eqref{est:Hm1.1} holds with a constant that is uniform with respect to the mesh element $T$.

The proof proceeds by contradiction. Assume that \eqref{est:Hm1.1} does not hold uniformly with respect to $T$. Then there is a sequence $(T_n,w_n)_{n\in\Natural}$ such that $T_n$ is contained in $B_1$ and is star-shaped with respect to every point in $B_\varrho$, $w_n\in L^2(T_n)$ has a zero average over $T_n$, and
\begin{equation}\label{est:Hm1.2}
\norm[H^{-1}(T_n)]{w_n}\ge n\norm[H^{-1}(T_n)^d]{\GRAD w_n}.
\end{equation}
Replacing $w_n$ with $w_n/\norm[H^{-1}(T_n)]{w_n}$, we can also assume that
\begin{equation}\label{est:Hm1.2.1}
\norm[H^{-1}(T_n)]{w_n}=1.
\end{equation}
Let $\widetilde{w}_n$ be the extension of $w_n$ to $B_1$ by $0$ outside $T_n$. By \eqref{est:Hm1.0} (in which we recall that $C_T=C_0$ only depends on $\varrho$ and $d$), $\norm[T_n]{w_n}$ is bounded, and so $\widetilde{w}_n$ is bounded in $L^2(B_1)$.
Hence, $L^2(B_1)$ being compactly embedded in $H^{-1}(B_1)$, upon extracting a subsequence we can assume the existence of $w\in L^2(B_1)$ such that
\begin{equation}\label{est:Hm1.cv}
\widetilde{w}_n\to w\mbox{ weakly in $L^2(B_1)$ and strongly in $H^{-1}(B_1)$ as $n\to\infty$}.
\end{equation}
The weak convergence in $L^2(B_1)$ together with the relation $0=\int_{T_n}w_n=\int_{B_1}\widetilde{w}_n$ shows that
\begin{equation}\label{est:Hm1.average}
\int_{B_1}w=0.
\end{equation}

Considering the uniform atlas of $\partial T_n$ (independent of $n$) given by Lemma \ref{lem:atlas}, we see that the corresponding maps $(\varphi_{l,n})_{l=1,\ldots,m}$ are uniformly Lipschitz, with a constant not depending on $n$. Hence, upon extracting another subsequence, we can assume that these maps converge uniformly to some Lipschitz functions $(\varphi_l)_{l=1,\ldots,m}$. These Lipschitz functions define a Lipschitz open set $U$ and, by uniform convergence of the maps, the following two properties hold: 
\begin{enumerate}[(i)]
\item the characteristic function $\mathbf{1}_{T_n}$ of $T_n$ converges strongly in $L^2(B_1)$ towards the characteristic function $\mathbf{1}_U$ of $U$, and
\item for any $\VEC{\psi}\in C^\infty_c(U)^d$ there is an $N(\VEC{\psi})\in\Natural$ such that $\mathrm{supp}(\VEC{\psi})\subset T_n$ for all $n\ge N(\VEC{\psi})$.
\end{enumerate}
We exploit Property (i) by writing $\widetilde{w}_n=\mathbf{1}_{T_n}\widetilde{w}_n$ (since $\widetilde{w}_n$ is equal to zero outside $T_n$), and by passing to the $L^2$-weak limit in the left-hand side and the weak/strong distributional limit in the right-hand side, to see that $w=\mathbf{1}_U w$. In particular, this shows that $w=0$ outside $U$ and, together with \eqref{est:Hm1.average}, that
\begin{equation}\label{est:Hm1.average.2}
\int_{U}w=0.
\end{equation}
Consider now Property (ii) of $(T_n)_{n\in\Natural}$. Fixing $\VEC{\psi}\in C^\infty_c(U)^d$, for any $n\ge N(\VEC{\psi})$ we can write
\begin{align*}
\left|\int_{B_1}\widetilde{w}_n\DIV\VEC{\psi}\right|=\left|\int_{T_n}w_n\DIV\VEC{\psi}\right|={}&\left|-\langle \GRAD w_n,\VEC{\psi}\rangle_{H^{-1}(T_n),H^1_0(T_n)}\right|\\
\le{}& \norm[H^{-1}(T_n)^d]{\GRAD w_n}\norm[H^1_0(T_n)^d]{\VEC{\psi}}
\le \frac{1}{n}\norm[H^1_0(U)^d]{\VEC{\psi}}
\end{align*}
where the first line follows from the definitions of $\widetilde{w}_n$ and $\GRAD w_n$ together with the fact that $\VEC{\psi}\in C^\infty_c(T_n)^d$ (since $\mathrm{supp}(\VEC{\psi})\subset T_n$), and the second line is a consequence of \eqref{est:Hm1.2}--\eqref{est:Hm1.2.1} and of the fact that $\VEC{\psi}$ has a compact support in $U$.
Combined with the weak convergence in \eqref{est:Hm1.cv} this shows that
\[
\int_{B_1} w\DIV\VEC{\psi}=0.
\]
Since this is true for any $\VEC{\psi}\in C^\infty_c(U)^d$, this proves that $\GRAD w=0$ in $\mathcal D'(U)^d$. By construction, $U$ is connected and thus $w$ is constant over $U$. Invoking \eqref{est:Hm1.average.2}, we deduce that $w=0$ on $U$ and thus, since $w=0$ outside $U$, that $w=0$ on $B_1$. The strong convergence in \eqref{est:Hm1.cv} therefore shows that
\begin{equation}\label{est:Hm1.cvzero}
\widetilde{w}_n\to 0\mbox{ strongly in $H^{-1}(B_1)$ as $n\to+\infty$.}
\end{equation}
To conclude the proof, recall \eqref{est:Hm1.2.1} and notice that any function $\varphi\in H^1_0(T_n)$ can be considered, after extension by $0$ outside $T_n$, as a function in $H^1_0(B_1)$ with $\norm[H^1_0(T_n)]{\varphi}=\norm[H^1_0(B_1)]{\varphi}$. Hence, by definition of the norms in $H^{-1}(B_1)$ and $H^{-1}(T_n)$,
\[
\norm[H^{-1}(B_1)]{\widetilde{w}_n}=\sup_{\varphi\in H^1_0(B_1)}\frac{\int_{B_1}\widetilde{w}_n\varphi}{\norm[H^1_0(B_1)]{\varphi}}\\
\ge\sup_{\varphi\in H^1_0(T_n)}\frac{\int_{T_n}w_n\varphi}{\norm[H^1_0(T_n)]{\varphi}}=\norm[H^{-1}(T_n)]{w_n}=1.
\]
On the other hand, property \eqref{est:Hm1.cvzero} shows that the left-hand side goes to $0$ as $n\to+\infty$, which establishes the sought contradiction. \end{proof}

\begin{remark}[Second Korn inequality in $L^{q}$]
  Following \cite[Remark IV.1.1]{Boyer.Fabrie:13}, we could as well establish a uniform local second Korn inequality in $L^q$ spaces, with $1<q<+\infty$, rather than in the $L^2$ space.
\end{remark}

\subsection{Approximation properties of the strain projector}\label{app:elproj.approx}

Let $T$ be a polytopal open connected set of $\Real^d$ and $\ell\ge 1$ be a given integer.
The strain projector $\elproj[T]{\ell}:H^1(T)^d\to\Poly{\ell}(T)^d$ is such that, for any $\VEC{v}\in H^1(T)^d$,
\begin{subequations}\label{eq:elproj}
  \begin{gather}\label{eq:elproj:pde}
    (\GRADs\elproj[T]{\ell}\VEC{v},\GRADs\VEC{w})_T = (\GRADs\VEC{v},\GRADs\VEC{w}_T)_T
    \qquad\forall\VEC{w}\in\Poly{\ell}(T)^d,
    \\ \label{eq:elproj:closure}
    \int_T\elproj[T]{\ell}\VEC{v}=\int_T\VEC{v},\qquad
    \int_T\GRADss\elproj[T]{\ell}\VEC{v}=\int_T\GRADss\VEC{v}.
  \end{gather}
\end{subequations}
By the Riesz representation theorem in $\GRADs\Poly{\ell}(T)^d$ for the inner product of $L^2(T)^{d\times d}$, relation \eqref{eq:elproj:pde} defines a unique element $\GRADs\elproj[T]{\ell}\VEC{v}$, and thus a unique polynomial $\elproj[T]{\ell}\VEC{v}$ after accounting for the additional conditions in \eqref{eq:elproj:closure} (which prescribe a rigid-body motion).
%
\begin{theorem}[Approximation properties of the strain projector]\label{thm:elproj:approx}
  Let $\varrho>0$ and take $T\subset \Real^d$ a polytopal set, of diameter $h_T$, that is star-shaped with respect to every point of a ball of radius $\varrho h_T$.
  Let two integers $\ell\ge 1$ and $s\in\{1,\ldots,\ell+1\}$ be given.
  Then, there exists a real number $C>0$ depending only on $d$, $\varrho$, $\ell$, and $s$ such that, for all $m\in\{0,\ldots,s\}$ and all $\VEC{v}\in H^s(T)^d$,
  \begin{equation}\label{eq:elproj:approx}
    \seminorm[H^m(T)^d]{\VEC{v} - \elproj[T]{\ell}\VEC{v}}
    \le C h_T^{s-m}\seminorm[H^s(T)^d]{\VEC{v}}.
  \end{equation}
\end{theorem}
\begin{proof}
  We apply the abstract results of \cite[Lemma 2.1]{Di-Pietro.Droniou:17*1}, which readily extend to the vector case.
  This requires to prove that it holds, for all $\VEC{v}\in H^1(T)^d$,
  \begin{subequations}
    \begin{alignat}{2}
      \label{eq:H1.boundedness}
      \seminorm[H^1(T)^d]{\elproj[T]{\ell}\VEC{v}}
      &\lesssim\seminorm[H^1(T)^d]{\VEC{v}},
      &\qquad&\text{if $m\ge 1$},
      \\
      \label{eq:L2.boundedness}
      \norm[T]{\elproj[T]{\ell}\VEC{v}}
      &\lesssim\norm[T]{\VEC{v}} + h_T \seminorm[H^1(T)^d]{\VEC{v}}
      &\qquad&\text{if $m=0$}.  
    \end{alignat}
  \end{subequations}
  Inside the proof, we note $a\lesssim b$ the inequality $a\le Cb$ with generic positive constant $C$ having the same dependencies as in \eqref{eq:elproj:approx}.
  \medskip\\
  (i) \emph{The case $m\ge 1$.} We start by observing that equation \eqref{eq:elproj:pde} implies
  \begin{equation}\label{eq:H1.boundedness:basic}
    \norm[T]{\GRADs\elproj[T]{\ell}\VEC{v}}\le\norm[T]{\GRADs\VEC{v}},
  \end{equation}
  as can be easily checked letting $\VEC{w}=\elproj[T]{\ell}\VEC{v}$ and using a Cauchy--Schwarz inequality in the right-hand side.
  We can now write
  \begin{equation}\label{eq:H1.boundedness.basic}
    \begin{aligned}
      \seminorm[H^1(T)^d]{\elproj[T]{\ell}\VEC{v}}
      &=
      \norm[T]{\GRAD\elproj[T]{\ell}\VEC{v}}
      \\
      &\le
      \norm[T]{\GRAD\elproj[T]{\ell}\VEC{v}-\vlproj[T]{0}(\GRADss\elproj[T]{\ell}\VEC{v})}
      + \norm[T]{\vlproj[T]{0}(\GRADss\VEC{v})}
      \\
      &\lesssim
      \norm[T]{\GRADs\elproj[T]{\ell}\VEC{v}} + \norm[T]{\vlproj[T]{0}(\GRADss\VEC{v})}
      \\
      &\lesssim
      \norm[T]{\GRADs\VEC{v}} + \norm[T]{\GRADss\VEC{v}}
      \lesssim\seminorm[H^1(T)^d]{\VEC{v}},
    \end{aligned}
  \end{equation}
  where we have inserted $\MAT{0} =\vlproj[T]{0}(\GRADss\VEC{v})-\vlproj[T]{0}(\GRADss\elproj[T]{\ell}\VEC{v})$ (see \eqref{eq:elproj:closure}) into the norm and used the triangle inequality to pass to the second line, the local Korn inequality \eqref{eq:korn} to pass to the third line, and we have invoked \eqref{eq:H1.boundedness:basic} together with the boundedness of $\vlproj[T]{0}$ expressed by \eqref{eq:lproj:boundedness} with $s=0$ to conclude.
  This proves \eqref{eq:H1.boundedness}.
  \medskip\\
  (ii) \emph{The case $m=0$.}
  We can write
  $$
  \norm[T]{\elproj[T]{\ell}\VEC{v}}
  \le \norm[T]{\VEC{v}} + \norm[T]{\elproj[T]{\ell}\VEC{v}-\VEC{v}}
  \lesssim \norm[T]{\VEC{v}} + h_T \norm[T]{\GRAD(\elproj[T]{\ell}\VEC{v}-\VEC{v})}
  \lesssim \norm[T]{\VEC{v}} + h_T \seminorm[H^1(T)^d]{\VEC{v}},
  $$
  where we have inserted $\pm\VEC{v}$ into the norm and used the triangle inequality in the first bound, we have used a local Poincar\'e inequality (resulting from the approximation properties \eqref{eq:lproj:approx} of $\vlproj[T]{0}$ with $s=1$ and $m=0$) for the zero-average function $\elproj[T]{\ell}\VEC{v}-\VEC{v}$ in the second bound, and we have concluded using the triangle inequality together with \eqref{eq:H1.boundedness} to write $\norm[T]{\GRAD(\elproj[T]{\ell}\VEC{v}-\VEC{v})}\le\norm[T]{\GRAD\elproj[T]{\ell}\VEC{v}}+\norm[T]{\GRAD\VEC{v}}\lesssim\seminorm[H^1(T)^d]{\VEC{v}}$.
  This proves \eqref{eq:L2.boundedness}.
\end{proof}
\begin{remark}[Uniform Korn inequality]\label{rem:uniform.korn}
  Notice that a crucial point in the first step of the above proof is that the constant in the Korn inequality invoked to pass to the third line in \eqref{eq:H1.boundedness.basic} only depends on $d$ and $\varrho$.
  This fact is non-trivial, as seen in \ref{app:korn}.
\end{remark}
\begin{corollary}[Trace approximation properties of the strain projector]\label{cor:elproj:approx.trace}
  Assume that $T\subset \Real^d$ is a polytopal set which admits a partition $\mathcal{S}_T$ into disjoint simplices $S$ of diameter $h_S$ and inradius $r_S$, and that there exists a real number $\varrho>0$ such that, for all $S\in\mathcal{S}_T$,
  \[
  \varrho^2 h_S\le \varrho h_T\le r_S.
  \]
  Assume also that $T$ is star-shaped with respect to every point in a ball of radius $\varrho h_T$.
  Let $s\in\{1,\ldots,\ell+1\}$, and denote by $\Fh[T]$ the set of faces of $T$.
  Then, there exists a real number $C$ depending only on $d$, $\varrho$, $\ell$, and $s$ such that, for all $m\in\{0,\ldots,s-1\}$ and all $\VEC{v}\in H^s(T)^d$,
  \begin{equation}\label{eq:elproj:approx.trace}
    \seminorm[{H^m(\Fh[T])^d}]{\VEC{v}-\elproj[T]{\ell}\VEC{v}}
    \le C h_T^{s-m-\frac12}\seminorm[H^s(T)^d]{\VEC{v}}.
  \end{equation}
  Here, $H^m(\Fh[T])\coloneq\{v\in L^2(\partial T)\st v_{|F}\in H^m(F)^d\mbox{ for all } F\in\Fh[T]\}$ is the broken Sobolev space on $\Fh[T]$ and $\seminorm[{H^m(\Fh[T])^d}]{{\cdot}}$ the corresponding broken seminorm.
\end{corollary}
\begin{proof}
  Under the assumptions on $T$, we have the following trace inequality (see, e.g., \cite[Lemma 1.49]{Di-Pietro.Ern:12}): For all $w\in H^1(T)$,
  \begin{equation}\label{eq:trace.T}
    \norm[L^2(\partial T)]{w}\lesssim h_T^{-\frac12}\norm[T]{w} + h_T^{\frac12}\norm[T]{\GRAD w},
  \end{equation}
  with hidden multiplicative constant in $\lesssim$ having the same dependencies as in \eqref{eq:elproj:approx.trace}.
  For $m\le s-1$, applying \eqref{eq:trace.T} to the components of $\partial^{\VEC{\alpha}}(\VEC{v}-\elproj[T]{\ell}\VEC{v})$ for every multi-index $\VEC{\alpha}\in\Natural^d$ such that $\sum_{i=1}^d\alpha_i=m$, we obtain
  $$
  \seminorm[{H^m(\Fh[T])^d}]{\VEC{v}-\elproj[T]{\ell}\VEC{v}}
  \lesssim h_T^{-\frac12}\norm[H^m(T)^d]{\VEC{v}-\elproj[T]{\ell}\VEC{v}}
  + h_T^{\frac12}\norm[H^{m+1}(T)^d]{\VEC{v}-\elproj[T]{\ell}\VEC{v}}.
  $$
  The conclusion follows using \eqref{eq:elproj:approx} for $m$ and $m+1$ to bound the terms in the right-hand side.
\end{proof}


\section*{Acknowledgements}
The work of the second author was supported by \emph{Agence Nationale de la Recherche} grants HHOMM (ANR-15-CE40-0005) and fast4hho. The work of the third author was partially supported by the Australian Government through the Australian Research Council's Discovery Projects funding scheme (project number DP170100605).
Fruitful discussions with Matthieu Hillairet (Univ. Montpellier) are gratefully acknowledged.



\begin{footnotesize}
  \bibliographystyle{plain}
  \bibliography{brho}
\end{footnotesize}

\end{document}